\documentclass[11pt]{amsart}
\usepackage{amssymb}
 \usepackage{amsmath,amscd}
\usepackage{graphicx}
\usepackage{color}
\usepackage{latexsym}\usepackage{ifthen}
\usepackage{wrapfig}
\usepackage{amssymb,amsmath,amstext}
\usepackage{epsfig}
\usepackage{graphics}
\usepackage{subfigure}
\usepackage{euscript}
\usepackage{pgf,tikz}
\usepackage{mathrsfs}
\usetikzlibrary{calc,matrix,arrows,decorations.pathmorphing,decorations.pathreplacing,3d}
\usepackage{amsthm}
\usepackage{url}
\usepackage{hyperref}
\usepackage{fullpage}

\newcommand{\midarrow}{\tikz \draw[-triangle 45] (0,0) -- +(.1,0);}

\theoremstyle{plain}%
\newtheorem{theorem}{Theorem}[section]
\newtheorem{proposition}[theorem]{Proposition}
\newtheorem{lemma}[theorem]{Lemma}
\newtheorem{corollary}[theorem]{Corollary}

\newtheorem*{theorem*}{Theorem}

\theoremstyle{definition}
\newtheorem{definition}[theorem]{Definition}
\newtheorem{remark}[theorem]{Remark}
\newtheorem{algorithm}[theorem]{Algorithm}
\newtheorem{example}[theorem]{Example}

\DeclareMathOperator{\Div}{\operatorname{Div}}

\DeclareMathOperator{\mR}{\mathbb{R}}

\DeclareMathOperator{\mC}{\mathbb{C}}

\DeclareMathOperator{\divisor}{\operatorname{div}}
\DeclareMathOperator{\red}{\operatorname{red}}
\DeclareMathOperator{\Red}{\operatorname{Red}}

\DeclareMathOperator{\slope}{\operatorname{sl}}
\DeclareMathOperator{\Comp}{\operatorname{Comp}}
\DeclareMathOperator{\Tan}{\operatorname{Tan}}
\DeclareMathOperator{\imag}{\operatorname{Im}}
\DeclareMathOperator{\In}{\operatorname{In}}
\DeclareMathOperator{\Out}{\operatorname{Out}}

\DeclareMathOperator{\Leaf}{\operatorname{Leaf}}
\DeclareMathOperator{\BifB}{\operatorname{Bif}(\mathcal{B})}
\DeclareMathOperator{\Bif}{\operatorname{Bif}_\rho}

\DeclareMathOperator{\phiLambda}{\phi^\Lambda}

\DeclareMathOperator{\BPDHone}{\operatorname{BP}^{(1)}_{\mathcal{D},\mathcal{H}}}
\DeclareMathOperator{\BPDHtwo}{\operatorname{BP}^{(2)}_{\mathcal{D},\mathcal{H}}}
\DeclareMathOperator{\BPDHthree}{\operatorname{BP}^{(3)}_{\mathcal{D},\mathcal{H}}}
\DeclareMathOperator{\BPDHfour}{\operatorname{BP}^{(4)}_{\mathcal{D},\mathcal{H}}}

\DeclareMathOperator{\LDHone}{\Lambda^{(1)}_{\mathcal{D},\mathcal{H}}}
\DeclareMathOperator{\LDHtwo}{\Lambda^{(2)}_{\mathcal{D},\mathcal{H}}}
\DeclareMathOperator{\LDHthree}{\Lambda^{(3)}_{\mathcal{D},\mathcal{H}}}
\DeclareMathOperator{\LDHfour}{\Lambda^{(4)}_{\mathcal{D},\mathcal{H}}}

\begin{document}
\title[Smoothing of Limit Linear Series of Rank One]{Smoothing of Limit Linear Series of Rank One on Saturated Metrized Complexes of Algebraic Curves}
\author{Ye Luo}
\address{School of Information Science and Engineering, Xiamen University}
\email{luoye80@gmail.com}
\thanks{Part of this work was done while Ye Luo was at Rice University and Georgia Institute of Technology.}
\author{Madhusudan Manjunath}
\thanks{Madhusudan Manjunath was supported by the Feoder-Lynen Fellowship of the Humboldt Foundation during this work and part of this work was done while the author was at the  University of California, Berkeley and Georgia Institute of Technology.}
\address{School of Mathematical Sciences, Queen Mary University of London}
\email{ m.manjunath@qmul.ac.uk}

\begin{abstract}
We investigate the smoothing problem of limit linear series of rank one on an enrichment of the notions of nodal curves and metrized complexes called saturated metrized complexes.  We give a finitely verifiable full criterion for smoothability of a limit linear series of rank one on saturared metrized complexes,  characterize the space of all such smoothings, and extend the criterion to metrized complexes. As applications, we prove that all limit linear series of rank one are smoothable on saturated metrized complexes corresponding to curves of compact-type, and prove an analogue for saturated metrized complexes of a theorem of Harris and Mumford on the characterization of nodal curves contained in a given gonality stratum. In addition, we give a full combinatorial criterion for smoothable limit linear series of rank one on saturated metrized complexes corresponding to nodal curves whose dual graphs are made of separate loops.
\end{abstract}

\maketitle

\section{Introduction}\label{S:introduction}

A saturated metrized complex is an object that encodes information about a degenerating family of smooth curves. Roughly speaking, a saturated metrized complex $\mathfrak{C}$ over a field $\kappa$ consists of a metric graph and  an algebraic curve over $\kappa$ associated to each point of the metric graph (Definition~\ref {D:MetComp}). A limit linear series of rank $r$ on $\mathfrak{C}$ consists of a linear series of rank $r$ on each associated algebraic curve of $\mathfrak{C}$ satisfying certain compatibility conditions. The main purpose of this paper is to provide a full criterion for lifting a limit linear series of rank $1$ to a linear series of the same rank on a smooth curve $C$. 

{\bf Context and Motivation:} Degeneration to singular curves has been one of the most important tools in the theory of smooth algebraic curves.  Fundamental results on algebraic curves such as the Brill-Noether Theorem and Gieseker-Petri Theorem are established via degeneration to singular curves~\cite{GH80,Gieseker82,EH86}\cite[Chapter 5]{HM98}.
While degeneration to irreducible curves such as cuspidal curves and nodal curves was considered by Castelnuovo~\cite{Castelnuovo1889} and several researchers subsequently, Eisenbud and Harris~\cite{EH86} developed a theory of degeneration of linear series to certain reducible curves called limit linear series. A limit linear series is usually denoted by limit $g^r_d$ where the integer $d$ and $r$ are called the degree and rank of the limit linear series respectively. The theory of limit linear series has numerous applications, for instance a proof of non-unirationality of $M_{23}$ \cite{EH87}, a detailed study of the monodromy of Weierstrass points \cite{EH87-2} and a proof of irreducibility of certain families of special linear series of curves \cite{EH89}. However, until recently the notion of limit linear series was largely restricted to curves of compact-type, i.e., nodal curves whose dual graph is a tree. We refer to Osserman's survey for a more recent treatment of limit linear series on curves of compact type \cite{Osserman12}.

Recently,  Amini and Baker \cite{AB12} defined a notion of limit linear series on nodal curves in general. In fact, instead of working with nodal curves per se, they considered an enrichment of nodal curves called metrized complexes (nodal curves with a metric assigned to the corresponding dual graph) and formulated a notion of limit linear series on a metrized complex.  They show that on a  metrized complex associated to a curve of compact-type their notion of limit linear series coincides with that of Eisenbud and Harris. Independently, Ossermann \cite{Osserman14} has also generalized the notion of limit linear series of Eisenbud-Harris to curves of non-compact type.

The notions of limit linear series of Eisenbud and Harris,  Amini and Baker and Osserman satisfy two key properties:
\begin{enumerate}
\item \label{smooth_item} For any family of smooth curves degenerating to  a curve of compact type (a metrized complex or a nodal curve), any family of linear series on each smooth curve in the family degenerates to a limit linear series on the curve of compact type (a metrized complex or a nodal curve respectively). This property is called the specialization property.

\item The limit linear series is formulated in terms of linear series on each irreducible component and with relations between the linear series on each irreducible component that depend on the dual graph (see Definition~\ref{D:LimLinSeries} for a precise definition).
\end{enumerate}

However,  even in the case of curves of compact type the converse of Property-(\ref{smooth_item}) does not hold in general. In other words not every limit $g^r_d$ arises as a limit of linear series. A limit linear series is said to be smoothable if it arises as a limit of linear series (see Definition~\ref{D:smoothable} for a precise definition of smoothability). Eisenbud and Harris  also considered a refinement of the notion of limit linear series called refined limit linear series. They showed that every refined limit $g^1_d$ is smoothable. For $r \geq 2$, they constructed a moduli space of limit $g^r_d$'s and showed that any limit $g^r_d$ in irreducible components of the expected dimension in the moduli space of limit $g^r_d$'s is smoothable, which they call the regeneration theorem (\cite[Theorem 3.4]{EH86}).

\subsection{Smoothing Criterion for Limit Linear Series of Rank One on Saturated Metrized Complexes}

In this paper, we consider a refinement of the notion of metrized complex called \emph{saturated metrized complexes} (see Definition~\ref{D:MetComp}) and undertake a detailed study of smoothability of limit $g^1_d$'s on saturated metrized complexes of algebraic curves. Roughly speaking, a saturated metrized complexes can be considered as  a metric graph $\Gamma$ with algebraic curves associated to all points of $\Gamma$. On the other hand, for a metrized complex, only the points in a finite subset $A$ of $\Gamma$ are associated with algebraic curves. Therefore, a saturated metrized complex can be derived from a metrized complex by inserting curves to the points in $\Gamma\setminus A$ (the process is called a saturation of the metrized complex). Conversely, a metrized complex can be derived from a saturated metrized complex by ignoring the curves associated to the points in $\Gamma\setminus A$.

The main goal in this paper is to provide a full  smoothing criterion for limit $g^1_d$'s on saturated metrized complexes. We briefly explain our approach first and state the criterion in Theorem~\ref{thm:main_criterion}.

In a recent work of \cite{ABBR13,ABBR14}, Amini, Baker, Brugall\'{e} and Rabinoff  show that harmonic morphisms between metrized complexes can essentially be lifted to finite morphisms between curves. Let $\mathbb{K}$ be an algebraically closed field of characteristic $0$  complete with respect to a non-archimedean valuation of value group $\mathbb{R}$. Let the corresponding residue field be $\kappa$. Suppose that $X$ is a smooth proper curve over $\mathbb{K}$. Let $\Sigma$ be a skeleton of the Berkovich analytification $X^{\rm an}$ of $X$.  Let $\mathfrak{C}(\Sigma)$ be the saturated metrized  complex over the residue field $\kappa$ associated to $\Sigma$ (see Appendix~\ref{subS:SkelRefMet} for a precise construction). A base point free $g^1_d$ on $X$ induces a morphism $\phi: X \rightarrow \mathbb{P}^1$ of degree $d$. By the functoriality of analytification, we have an induced map $\phi^{\rm an}: X^{\rm an} \rightarrow \mathbb{P}^1_{\rm Berk}$ where $\mathbb{P}^1_{\rm Berk}$ is the Berkovich projective line. The retraction from $X^{\rm an}$ to the skeleton $\Sigma$ induces a pseudo-harmonic morphism $\mathfrak{C}\phi$ (see Section~\ref{S:harmmor} for a precise definition) from the saturated metrized complex $\mathfrak{C}(\Sigma)$ to a saturated metrized complex $\mathfrak{C}(T)$ of genus zero where the underlying metric tree $T$ of $\mathfrak{C}(T)$ is a skeleton of $\mathbb{P}^1_{\rm Berk}$. On the other hand, by Theorem A of \cite{ABBR13} there is a harmonic morphism $\mathfrak{C}\phi^{\rm mod}$ from a modification $\mathfrak{C}(\Sigma^{\rm mod})$ of $\mathfrak{C}(\Sigma)$ to $\mathfrak{C}(T)$ which is compatible with $\phi^{\rm an}$ . More precisely, we have the following commutative diagram:
\[
\begin{tikzpicture}[scale=1.8,descr/.style={fill=white}, text height=1.2ex, text depth=0.2ex]\label{commutdia}
\node (X) at (0,0) {$X$};
\node (Xan) at (0,-0.6) {$X^{\rm an}$};
\node (Cmod) at (-1.2,-1.2) {$\mathfrak{C}(\Sigma^{\rm mod})$};
\node (C) at (0,-1.2) {$\mathfrak{C}(\Sigma)$};
\node (P1) at (1.2,0) {$\mathbb{P}^1$};
\node (P1Berk) at (1.2,-0.6) {$\mathbb{P}^1_{\rm Berk}$};
\node (CT) at (1.2,-1.2) {$\mathfrak{C}(T)$};
\path[->,font=\scriptsize,>=angle 90]
(X) edge node[auto]{$\phi$} (P1)
(Xan) edge node[auto]{$\phi^{\rm an}$} (P1Berk)
(Cmod) edge[out=-45,in=-135] node[auto]{$\mathfrak{C}\phi^{\rm mod}$} (CT)
(C) edge node[auto]{$\mathfrak{C}\phi$} (CT)
(X) edge (Xan)
(Xan) edge  (Cmod)
(Xan) edge (C)
(P1) edge  (P1Berk)
(P1Berk) edge (CT)
(Cmod) edge (C);
\end{tikzpicture}
\]

Recall that we can represent a $g^1_d$ on $X$ by $(D,H)$ where $D$ is an effective divisor of degree $d$ and $H$ is a two-dimensional linear subspace of $H^0(X,\mathcal{L}(D))$. The specialization $(\mathcal{D},\mathcal{H})$ of $(D,H)$ to $\mathfrak{C}(\Sigma)$ represents a limit $g^1_d$ on $\mathfrak{C}(\Sigma)$ which is smoothable. On the other hand, the lifting theorem (Theorem B, \cite{ABBR13}) guarantees that any harmonic morphism onto a genus-$0$ (saturated) metrized complex can be lifted to a finite morphism onto a projective line. Therefore, to investigate whether a limit $g^1_d$ represented by $(\mathcal{D},\mathcal{H})$ on a saturated metrized complex is smoothable, we must characterize when there exist a genus-$0$ saturated metrized complex $\mathfrak{C}(T)$ and a modification $\mathfrak{C}(\Sigma^{\rm mod})$ of $\mathfrak{C}(\Sigma)$ such that there is a harmonic morphism $\mathfrak{C}\phi^{\rm mod}:\mathfrak{C}(\Sigma^{\rm mod})\rightarrow\mathfrak{C}(T)$ which is ``compatible'' with the data $(\mathcal{D},\mathcal{H})$.

An important aspect is that starting from a smoothable limit $g^1_d$ represented by $(\mathcal{D},\mathcal{H})$ on $\mathfrak{C}(\Sigma)$, there generally exist different choices of $\mathfrak{C}(\Sigma)^{\rm mod}$, $\mathfrak{C}(T)$ and $\mathfrak{C}\phi^{\rm mod}:\mathfrak{C}(\Sigma^{\rm mod})\rightarrow\mathfrak{C}(T)$ in the above commutative diagram. Therefore, to give a full smoothing criterion, the main challenge is to determine all possible $\mathfrak{C}(\Sigma^{\rm mod})$, $\mathfrak{C}(T)$ and $\mathfrak{C}\phi^{\rm mod}$ which are compatible with the data $(\mathcal{D},\mathcal{H})$.

In this paper, the key to addressing this challenge is to first reorganize the information in $(\mathcal{D},\mathcal{H})$ and associate a combinatorial object called \emph{global diagram} on $\Gamma$. A global diagram consists of (a) a ``piecewise-constant differential form'' defined on $\Gamma$  and (b) a partition of the set of tangent directions at every point in $\Gamma$. It turns out that the failure of exactness of this differential form is an obstruction to smoothing $(\mathcal{D},\mathcal{H})$. We say that $(\mathcal{D},\mathcal{H})$ is solvable if this differential form is exact. Hence, by integrations of the exact differential form for a solvable $(\mathcal{D},\mathcal{H})$, we derive a rational function $\rho$ on $\Gamma$ with everywhere nonzero slopes. The bifurcation tree $\mathcal{B}$  associated $\rho$ is a metric tree whose points correspond to superlevel components of $\Gamma$ at all values of $\rho$. Moreover, there is a canonical projection $\pi_\mathcal{B}$ from $\Gamma$ to $\mathcal{B}$. By properly gluing the bifurcation tree $\mathcal{B}$ along its branches, we can derive a metric tree which is called a partition tree with respect to $\mathcal{B}$. (See Section~\ref{S:BifParTrees}  for precise definitions of bifurcation trees and partition trees.) We denote by $\LDHone$ the space of all such partition trees with respect to $\mathcal{B}$. One observation is that in the above commutative diagram, the underlying metric tree $T$ of $\mathfrak{C}(T)$ can only possibly be a partition tree. 

More smoothing obstructions arise from the compatibility between $\pi_{\mathcal{B}}$ and the data in $(\mathcal{D},\mathcal{H})$.  We organize these obstructions into three additional levels and associate a subspace (denoted by  $\LDHtwo$, $\LDHthree$ and $\LDHfour$ respectively) of $\LDHone$  to each of these three additional levels (Section~\ref{S:ParTreeComp}).  In particular, $\LDHone\supseteq\LDHtwo\supseteq\LDHthree\supseteq\LDHfour$. We will show that any genus zero saturated  metrized complex $\mathfrak{C}(T)$ that can arise in the above commutative diagram must have an underlying metric tree $T$ in $\LDHfour$, and conversely we can build such a $\mathfrak{C}(T)$ from any metric tree $T$ in $\LDHfour$. See Figure \ref{F:Lattice} for an example of $\LDHone$, $\LDHtwo$, $\LDHthree$ and $\LDHfour$.

In summary,  we have the following theorem as the smoothing criterion.
\begin{theorem} {\rm ({\bf Smoothing Criterion})} \label{thm:main_criterion}
A pre-limit $g^1_d$ (or a limit $g^1_d$) represented by $(\mathcal{D},\mathcal{H})$ on a saturated metrized complex is smoothable if and only if the space of metric trees $\LDHfour$ associated to $(\mathcal{D},\mathcal{H})$ is nonempty.
\end{theorem}

One feature that distinguishes the study of smoothability of limit $g^1_d$s in comparison with the lifting of harmonic morphisms of \cite{ABBR13,ABBR14} is the appearance of the space of trees $\LDHone$, $\LDHtwo$, $\LDHthree$ and $\LDHfour$. 
These spaces are naturally endowed with a partial order (see Appendix~\ref{S:TreeSpace}) and lead to the several directions for future work.  For example, they are essentially spaces of phylogenetic trees and seem to be related to similar spaces occurring in other contexts such as the space of trees, studied by \cite{BHV01} and in the context of dynamics in \cite{DM08}.   Another direction of investigation is that the work of Abramovich-Caporaso-Payne \cite{ACP12} suggests that  $\LDHfour$ can be interpreted in terms of the Berkovich analytification of the moduli space of all smoothings of the underlying limit $g^1_d$. 

Furthermore, the above criterion is an effective characterization of smoothable limit $g^1_d$'s, since it is not only a full criterion but  also finitely verifiable given the data $(\mathcal{D},\mathcal{H})$ (for finite verifications, see Section~\ref{S:Finite} and an equivalent form of the smoothing criterion in Theorem~\ref{T:main}).

We also expect the global diagram technique employed in setting up the smoothing criterion to be a useful tool for characterizing the gonality stratum  and for studying moduli spaces of metrized complexes (and  tropical curves) with a given gonality. These topics will be pursued in the future. 

We would like to mention that Omid Amini has independently obtained a smoothing theorem on limit $g^1_d$s  for limit linear series in the framework developed in an upcoming paper~\cite{A14_1} (also see \cite{A14_2})  which is a refinement of ~\cite{AB12}. The problem of smoothing divisors of rank two and higher has also been studied by Cartwright recently~\cite{Cartwright12}. Cartwright considers lifting of rank two divisors on tropical curves (and metrized complexes).  Given a matroid and an infinite field $\kappa$, he constructs a graph (and a metrized complex) and a divisor of rank two called a matroid divisor on it and shows that the problem of lifting the matroid divisor to a smooth curve over $\kappa[[t]]$ is equivalent to realizability of the underlying matroid over $\kappa$. These results show that the lifting problem for higher rank divisors (rank two and higher) is sensitive to the underlying field and is  evidence for the difficulty in generalizing our smooth criterion for higher rank divisors. In addition, Cartwright-Jensen-Payne~\cite {CJP14} and Jensen-Ranganathan~\cite{JR17} have proved higher rank lifting theorems for divisors on tropical curves which are general and special chain of loops respectively.

\subsection{Applications}
We present four concrete applications of the smoothing criterion.

{\bf Application 1: Saturated Metrized Complexes of Compact-Type}

In Subsection \ref{S:CompactType}, we prove a version of the smoothing theorem of Eisenbud and Harris \cite{EH86} for curves of compact-type in the setting of saturated metrized complexes with underlying metric graphs being trees (we call them  saturated metrized complexes of compact type) as follows:
on a saturated metrized complex of compact type, every diagrammatic pre-limit linear series of rank one must be smoothable.

{\bf Application 2: Saturated Metrized Complexes with Genus-$g$ Underlying Metric Graphs containing $g$ Separate Loops}

In Subsection~\ref{S:GenusOne}, we follow the same philosophy in the proof of the above Eisenbud-Harris Theorem and generalize it to a full combinatorial characterization of smoothable limit $g^1_d$'s on a saturated metrized complex whose underlying metric graph $\Gamma$ containing $g$ separate loops. Note that the case for metric graphs made of chain of loops which are used for tropical proofs of  the Brill-Noether Theorem \cite{CDPR12} and Gieseker-Petri Theorem \cite{JP14} falls into this category.

{\bf Application 3: Saturated Metrized Complexes of  Harris-Mumford Type}

Harris and Mumford in Theorem~5 of \cite{HM82} consider specific families of nodal curves and prove a smoothing theorem for limit $g^1_d$s for curves in these families. Using this, they obtain a partial characterization of the gonality stratum. We refer to saturated metrized complexes corresponding to these families of nodal curves as Harris-Mumford saturated metrized complexes and discuss them in detail in Subsection~\ref{S:harrismumapplication}. The setting of Theorem~5 of \cite{HM82} is slightly different from ours since they work with nodal curves rather than saturated metrized complexes. In particular, the Harris-Mumford types of saturated metrized complexes have flower-like underlying metric trees (see Figure~\ref{F:HarrisMumford}). In Theorem \ref{T:HarrisMumford} we prove an analogue of  \cite[Theorem 5]{HM82}, while Theorem~\ref{thm:main_criterion} can actually be considered as a generalization of Theorem~5 of \cite{HM82} in the above sense.

{\bf Application 4: Extending the Smoothing Criterion to  Metrized Complexes}

Theorem~\ref{thm:main_criterion} also suggests the following approach to an analogous smoothing criterion for  limit $g^1_d$'s on metrized complexes: a limit $g^1_d$ on a metrized complex is smoothable if and only if this limit $g^1_d$ can be extended to a smoothable limit $g^1_d$ on a saturated metrized complex which is a saturation of the original metrized complex. We have more detailed discussions and give a concrete example in Subsection~\ref{S:SmoothingMC} to show that even in the setting of metrized complexes, the smoothing criterion is still finitely verifiable. 

\subsection{Outline of the Rest of the Paper}
The rest of the paper is organized as follows. In Section~\ref{S:SmoothingThm}, we explain some notions and terminologies essential to the paper.   In Section~\ref{S:NonArch}, we study the relation between smoothable limit $g^1_d$'s and harmonic morphisms between saturated metrized complexes. In Section~\ref{S:BifParTrees}, we  discuss the notions of bifurcation trees and partition trees. In Section~\ref{S:ParTreeComp}, we study properties of the spaces $\LDHone$, $\LDHtwo$, $\LDHthree$ and $\LDHfour$ that arise from different levels of obstructions of the limit $g^1_d$ from being smoothable. In Section~\ref{S:SmoothingProof},  we prove the smoothing criterion. Section \ref{application_sect} is devoted to examples and applications of our smoothing criterion.

\section{Some Notions and Terminologies related to the Smoothing Criterion} \label{S:SmoothingThm}
In this section, we first present the basic notions and terminologies, and then state an alternative version of the smoothing criterion which is directly verifiable.

 Let $\mathbb{K}$ be an algebraically closed field of characteristic $0$  complete with respect to a non-trivial non-archimedean absolute value. We assume that the value group of $\mathbb{K}$ is the group of real numbers $(\mathbb{R},+)$.  Let $R$ be the valuation ring of $\mathbb{K}$, and $\kappa$ be the residue field of $\mathbb{K}$ that we also assume to be algebraically closed and of characteristic $0$. (On the other hand, starting from $\kappa$, we can also construct such a field $\mathbb{K}$ with value group $\mathbb{R}$ and residue field $\kappa$ using generalized Puiseux series~\cite{Markwig07}.)
For an algebraically closed field $K$, we let $K_\infty=K\bigcup\{\infty\}$ be the projective line $\mathbb{P}^1_K$ with coordinate. The above notations and assumptions will be applied throughout this paper.
 
 \subsection{Saturated Metrized Complexes}

\begin{definition}{\rm {\bf (Saturated Metrized Complex)}}\label{D:MetComp}
 A \emph{saturated metrized complex} $\mathfrak{C}$ over an algebraically closed field $\kappa$ consists of the following data:

\begin{itemize}

\item A metric graph $\Gamma$.

\item A smooth complete irreducible algebraic curve $C_p$ over $\kappa$ associated to each point $p\in\Gamma$ such that $C_p$ is a projective line except for points in a finite subset of $\Gamma$. (For simplicity, we also use $C_p$ to represent the set of closed points of $C_p$.)

 \item For every point  $p \in \Gamma$, there is an injection $\red_p: \Tan_\Gamma(p) \to C_p$ called the reduction map at $p$ where $\Tan_\Gamma(p)$ is the set of tangent directions on $\Gamma$ incident to $p$,  $\red_p(t)$ is called the marked point in $C_p$ associated to the tangent direction $t\in\Tan_\Gamma(p)$, and  $A_p=\imag(\red_p)$ is the set of marked points of $C_p$. For convenience, we let $\Red:\coprod_{p\in\Gamma}\Tan_\Gamma(p)\to\coprod_{p\in\Gamma}C_p$ be defined as $\Red(t)=\red_p(t)$ when $t\in\Tan_\Gamma(p)$. 
 
\end{itemize}
The \emph{genus} $g(\mathfrak{C})$ of $\mathfrak{C}$ is defined as $g(\Gamma)+\sum_{p \in \Gamma} g(C_p)$ where $g(\Gamma)$ is the genus (the first Betti number) of the metric graph $\Gamma$ and $g(C_p)$ is the genus of the algebraic curve $C_p$.  The genus of a saturated metrized complex is finite since $C_p$ has genus zero for all but a finite number of points $p\in\Gamma$.
\end{definition}

\begin{remark} \label{R:restriction}
The notion of saturated metrized complex follows the same philosophy as that of metrized complex in \cite{AB12} while in the latter, only the points in a finite vertex set $A$ of $\Gamma$ are associated with curves. Hence, by ignoring the curves associated to the points in $\Gamma\setminus A$ for a vertex set $A$ of $\Gamma$, we can derive a metrized complex $\mathfrak{C}'$ from a saturated metrized complex $\mathfrak{C}$. We say that $\mathfrak{C}'$ is the \emph{restriction} of $\mathfrak{C}$ to $A$. Conversely, given a metrized complex $\mathfrak{C}'$ and a saturated metrized complex $\mathfrak{C}$, we say that $\mathfrak{C}$ is a \emph{saturation} of $\mathfrak{C}'$ if $\mathfrak{C}'$ is a restriction of $\mathfrak{C}$ and they have the same genus (the newly inserted curves in $\mathfrak{C}$ are all projective lines). 
\end{remark}

\begin{remark}
Saturated metrized complexes appear as the inverse image of a skeleton in the map from the Huber adic space to the Berkovich analytification of a curve~\cite{Foster15}, and are closely related to ``exploded tropicalizations''~\cite{Payne09}.
\end{remark}

\subsection{Divisor Theory on a Saturated Metrized Complex}
Here we give a natural extension of the  divisor theory on metrized complexes in  \cite{AB12} to a divisor theory on saturated metrized complexes. This divisor theory on metrized complexes introduced in  \cite{AB12} combines the conventional divisor theory on algebraic curves and an analogous divisor theory  on metric graphs \cite{GK08,Luo13}.

Let $\mathfrak{C}$ be a saturated metrized complex with underlying metric graph $\Gamma$ and algebraic curve $C_p$ at point $p$.  A \emph{pseudo-divisor} $\mathcal{D}$ on a saturated metrized complex $\mathfrak{C}$ is the data $(D_\Gamma,\{D_p\}_{p \in \Gamma})$ where $D_{\Gamma}$ is a divisor on $\Gamma$ and $D_p$ is a divisor on the curve $C_p$ satisfying the relation $D_{\Gamma}(p)=\deg(D_p)$ for every point $p \in \Gamma$. We also say that $D_\Gamma$ is the tropical part of $\mathcal{D}$ and $D_p$ is the $C_p$-part of $\mathcal{D}$. We say that a point $u\in C_q$ is a \emph{supporting point} of $\mathcal{D}$ if $D_q(u)\neq 0$. The \emph{degree} of $\mathcal{D}$ is defined to be the degree of $D_{\Gamma}$. If in addition $D_p=0$ for all but finitely many points $p\in\Gamma$, then we call $\mathcal{D}$ a \emph{divisor} on $\mathfrak{C}$.  We say $\mathcal{D}$ is \emph{effective} if $D_p$ is effective for all points $p \in \Gamma$.

Note that for any pseudo-divisor $\mathcal{D}$ on a saturated metrized complex, $D_p$ will have degree $0$ for all but finitely many points $p \in \Gamma$. However, $D_p$ can be non-zero for infinitely many points in $\Gamma$, which is  unconventional from the viewpoint of divisor theory. The notion of divisors on a saturated metrized complex rectifies this aspect.

\begin{remark} \label{R:DivFreeAb}
All divisors on $\mathfrak{C}$ form a group  isomorphic to $\bigoplus_{p\in\Gamma}\Div(C_p)$ (the free abelian group on $\coprod_{p\in\Gamma}C_p$). Therefore, we also write a divisor $\mathcal{D}$ on $\mathfrak{C}$ formally as $\mathcal{D}=\sum_{p\in\Gamma}D_p$ where $D_p$ is the $C_p$-part of $\mathcal{D}$.
\end{remark}

A \emph{pseudo-rational function} $\mathfrak{f}=(f_{\Gamma},\{f_p\}_{p \in \Gamma})$  where $f_{\Gamma}$ is a tropical rational function (piecewise-linear function with integer slopes) on $\Gamma$ and $f_p$ is a rational function on the algebraic curve $C_p$. We also say that $f_{\Gamma}$ is the \emph{tropical part} of $\mathfrak{f}$ and $f_p$ is the \emph{$C_p$-part} of $\mathfrak{f}$. We say that $\mathfrak{f}$ is \emph{nonzero} if $f_p$ is nonzero for all $p\in\Gamma$.  The \emph{principal pseudo-divisor} $\divisor(\mathfrak{f})$ associated with a nonzero pseudo-rational function $\mathfrak{f}$ is defined as $(\divisor(f_\Gamma), \{\divisor(f_p)+ \divisor_p(f_\Gamma)\}_{p \in \Gamma})$ where $\divisor_p(f_\Gamma)=\sum_{t \in \Tan_\Gamma(p)} \slope_t(f_\Gamma) (\red_p(t))$ and $\slope_t(f_\Gamma)$ is the outgoing slope of the function $f_\Gamma$ along the tangent direction $t$.   A \emph{rational function} is a pseudo-rational function whose associated principal pseudo-divisor is a divisor. Divisors $\mathcal{D}_1$ and $\mathcal{D}_2$ are \emph{linearly equivalent} if they differ by a principal divisor.

As in the case of principal divisors on an algebraic curve or on a metric graph, the set of the principal divisors on a saturated metrized complex forms an abelian group under addition.

\subsection{Limit Linear Series on a Saturated Metrized Complex}

\begin{definition}{\rm {\bf (Pre-Limit Linear Series and Limit Linear Series)}} \label{D:LimLinSeries}
Let $\mathfrak{C}$ be a a saturated metrized complex.
\begin{enumerate}
\item  A \emph{pre-limit linear series} of rank $r$ and degree $d$ (also known as a pre-limit $g^{r}_d$) on $\mathfrak{C}$   is represented by the data $(\mathcal{D}, \mathcal{H})$  where $\mathcal{D}= (D_\Gamma,\{ D_p\}_{p \in \Gamma})$ is an effective divisor of degree $d$ on $\mathfrak{C}$ and  $\mathcal{H} = \{ H_p\}_{p \in \Gamma}$ where $H_p$ is an $(r+1)$-dimensional subspace of the function field of $C_p$.

\item A \emph{limit linear series} of rank $r$  and degree $d$ (also known as a limit $g^{r}_d$) on $\mathfrak{C}$  is a pre-limit $g^r_d$ represented by $(\mathcal{D}, \mathcal{H})$ with $\mathcal{H}=\{ H_p\}_{p \in \Gamma}$ which satisfies the following additional condition: for every effective divisor $\mathcal{E}=(E_{\Gamma},\{E_p\}_{p \in \Gamma})$ on $\mathfrak{C}$ of degree $r$ such that the support of $E_p$ does not intersect the set $A_p$ of marked points of $C_p$ for every $p \in \Gamma$, there exists a rational function $\mathfrak{f}=(f_{\Gamma},\{f_p\}_{p \in \Gamma})$ such that $f_p \in H_p$ for all points $p \in \Gamma$ and $\mathcal{D}-\mathcal{E}+{\rm div}(\mathfrak{f}) \geq 0$.
\end{enumerate}
\end{definition}	

Our definition of limit linear series is a slight modification of the notion of crude limit linear series in \cite[Section 5.3]{AB12}. For instance, we impose the additional restriction that the support of $E_p$ does not intersect the set $A_p$. This restriction is justified by Theorem~\ref{T:DHspecialization} which states that a specialization of linear series is a limit linear series.

Throughout this paper, unless otherwise specified, we let $\mathcal{D}=(D_\Gamma,\{D_p\}_{p \in \Gamma})$ and $\mathcal{H}=\{H_p\}_{p \in \Gamma}$ for convenience. When we say $F=\{f_p\}_{p\in\Gamma}$ is an element of $\mathcal{H}$, it means $f_p\in H_p$ for all $p\in\Gamma$. In addition, we may also consider $F$ as a function from $\coprod_{p\in\Gamma}C_p$ to $\kappa$ which sends $u\in C_p$ to $f_p(u)$. 

We say a  pre-limit $g^r_d$ (respectively, a limit $g^r_d$) represented by $(\mathcal{D}, \mathcal{H})$ is \emph{refined} if $(\mathcal{D}, \mathcal{H})$ has the following additional properties:
\begin{enumerate}
\item The constant function is contained in $H_p$ for every point $p \in \Gamma$.
\item For every point $p$, the support of $D_p$ is disjoint from the set $A_p$ of marked points of $C_p$.
\end{enumerate}

The two conditions are justified by Lemma~\ref{L:SmoothRefined} that a smoothable pre-limit $g^r_d$ must be refined.


\subsection{Alternative Version of the Smoothing Criterion}

We have the following equivalent version of the smoothing criterion in Theorem~\ref{thm:main_criterion}. This version is finitely verifiable (see Section~\ref{S:Finite}).

\begin{theorem}\label{T:main}{\rm ({\bf Smoothing Criterion, Version II})}
A pre-limit $g^1_d$ (or a limit $g^1_d$) represented by $(\mathcal{D}, \mathcal{H})$ on a saturated metrized complex is smoothable if and only if $(\mathcal{D}, \mathcal{H})$ is diagrammatic,  solvable and satisfies the intrinsic global compatibility conditions.
\end{theorem}

We primarily employ this version in the rest of the paper, in particular in Section~\ref{application_sect}. The equivalence of these two versions of the smoothing criterion follows from Proposition~\ref{P:LDH_BPDH}. In the rest of this section, we explain the terms ``diagrammatic'', ``solvable'' and ``intrinsic global compatibility'' conditions that appear in the smoothing criterion.

In Subsection~\ref{S:SmoothingMC}, we extend this smoothing criterion on saturated metrized complexes to the case for metrized complexes in a natural way. 

\subsection{Diagrams Induced by $(\mathcal{D}, \mathcal{H})$ and Solvability} 

For a  pre-limit $g^1_d$ represented by $(\mathcal{D},\mathcal{H})$ be refined, we extract the information in the space $H_p$ and associate a combinatorial object called local diagram
to a point $p\in\Gamma$. A global diagram is formed by assembling the local diagrams at all points in $\Gamma$ in a ``continuous'' way.

First let us give  a more precise description of local and global diagrams on a metric graph $\Gamma$.
\begin{enumerate}
\item A \emph{local diagram} at a point $p$ of $\Gamma$  is made of following data:
\begin{itemize}
\item A nonzero integer $m(p,t)$ called the \emph{multiplicity} associated to each tangent direction $t \in \Tan_\Gamma(p)$, where $\Tan_\Gamma(p)$ is the set of tangent directions emanating from $p$. We refer to those tangent directions with negative multiplicities as incoming tangent directions and denote the set they form by $\In_\Gamma(p)$. Similarly, we refer to those tangent directions with positive multiplicities as outgoing tangent directions and denote the set they form by $\Out_\Gamma(p)$.
\item  The elements in $\Tan_\Gamma(p)$ are partitioned into equivalence classes with one of the equivalence classes being exactly ${\rm In}(p)$. We refer to this partition of $\Tan_\Gamma(p)$ as the \emph{local partition} at $p$ and refer to the tangent directions that belong to the same equivalence class as \emph{locally equivalent}.
\end{itemize}
\item An open neighborhood of a point $p \in \Gamma$ is called a simple neighborhood if it is simply connected and every point in the neighborhood except possibly $p$ has valence two. For a simple neighborhood $U$ of a point $p \in \Gamma$,  a local diagram at a point $p$ induces a local diagram at any point $q$ in $U\setminus \{p\}$ as follows. Suppose that $q$ lies along the tangent direction $t \in \Tan_\Gamma(p)$. Assign the integer $-m(p,t)$ to the tangent direction at $q$ corresponding to the edge joining $p$ and $q$ and assign the integer $m(p,t)$ to the other tangent direction at $q$. Assign the two tangent directions at $q$ to different equivalence classes. 
\item A \emph{global diagram} on a metric graph $\Gamma$ is a collection of local diagrams at all the points in $\Gamma$ such that the local diagrams satisfy the following continuity property: for every point $p$, there is a simple neighborhood $U$ of $p$ such that for every point $q \in U\setminus\{p\}$ the local diagram induced by $p$ at $q$ coincides with the local diagram at $q$.
\item A global diagram on $\Gamma$ is called \emph{solvable} if there exists a tropical rational function $\rho$ on $\Gamma$ such that the outgoing slope ${\rm sl}_{t}(\rho)$ of $\rho$ along the tangent direction $t \in \Tan_\Gamma(p)$ for any point $p \in \Gamma$ coincides with the multiplicity. Formally, this means that the  differential equation ${\rm sl}_{t}(\rho)=m(p,t)$ is satisfied.  This equation is referred to as the \emph{characteristic equation} of the global diagram.
\end{enumerate}

\begin{remark} \label{R:GlobalDiagram}
Recall that a vertex set of $\Gamma$ (the set of all points of valence at least three) induces an edge-weighted graph called a \emph{model} of $\Gamma$ (see e.g. \cite{BPR11}).
Since $\Gamma$ is compact, the multiplicity aspect of a global diagram can be represented in terms of the following data: a model for $\Gamma$, an orientation of each edge of the model and a non-zero integer associated to each edge called the \emph{multiplicity} of that edge. In other words, a global diagram induces a piecewise-constant integer-valued differential form on $\Gamma$ and this differential form is exact if the global diagram is solvable.
\end{remark}

\begin{remark}
When a global diagram is solvable, the solution of the characteristic equation is unique up to a translation by a constant. Since $m(p,t)$ is nowhere zero for all $p$ and $t\in\Tan_\Gamma(p)$, the solution must have everywhere nonzero slopes. 
\end{remark}

For a two-dimensional linear space $H_p$ of rational functions on $C_p$ which contains the constant functions, we can naturally associate a local diagram defined to each point of $\Gamma$. The details of such a procedure are presented in the following construction. This construction translate the algebraic data encoded in $\mathcal{H}$ to combinatorial data on $\Gamma$. Furthermore, the two aspects of the definition of local diagrams, multiplicity and local partition, are both motivated from the notion of harmonic morphisms (Subsection~\ref{S:harmmor}).

 \textbf{Construction of local diagrams from $\mathcal{H}$:}\label{R:LocalHp}
 In the following, given a refined pre-limit $g^1_d$ represented by $(\mathcal{D},\mathcal{H})$, we associate a local diagram to every point on the metric graph using the data $\mathcal{H}$.  From the definition of a refined pre-limit $g^1_d$, we know that the two dimensional linear space $H_p$ of rational functions on $C_p$ has a basis $\{1,f_p\}$ where $f_p$ is a nonconstant rational function on $C_p$.  Let $\bar{f}_p:C_p\rightarrow \kappa_\infty$ be the function on $C_p$ extending $f_p$ to its poles.  We construct the local diagram at $p$ as follows:  for a tangent direction $t$ of $C_p$, we let the multiplicity $m(p,t)$ be the ramification index of $\bar{f}_p$ at the marked point corresponding to the tangent direction $t$ with sign `$-$' if $\red_p(t)$ is a pole of $f_p$ and sign `$+$' otherwise.  The local partition at $p$ is defined by declaring that two tangent directions are locally equivalent if and only if their marked points are in the same level set of $f_p$.
 Hence  $\In_\Gamma(p)$ is the set of tangent directions whose corresponding marked points are the poles of $f_p$  and the elements of $\In_\Gamma(p)$ are all locally equivalent. In this way, we construct a local diagram at $p$ from $f_p$. Moreover, since different choices of nonconstant functions $f_p$ in $H_p$ afford the same local diagram at $p$, we also say that this local diagram at $p$ is induced by $H_p$.

\textbf{Compatibility between $\mathcal{D}$ and $\mathcal{H}$:}
For a nonconstant function $f_p\in H_p$, we say that $D^+_{f_p}$ and $D^-_{f_p}$ are the effective and non-effective parts of $\divisor(f_p)$ respectively if $\divisor(f_p)=D^+_{f_p}-D^-_{f_p}$ and both $D^+_{f_p}$ and $D^-_{f_p}$ are effective with no overlapping supporting points. Note that by the construction of the local diagrams from $f_p$, the marked point associated to an incoming tangent direction in $\In_\Gamma(p)$ is a pole of $f_p$ and we must have $D^-_{f_p}+ \Sigma_{t\in\In_\Gamma(p)} m(p,t)(\red_p(t))\geqslant 0$ (note that $m(p,t)<0$ for $t\in\In_\Gamma(p)$). By comparing $D_p$ with $D^-_{f_p}+\Sigma_{t\in\In_\Gamma(p)} m(p,t)(\red_p(t))$, we introduce the compatibility between $D_p$ and $H_p$. More precisely, we say that $H_p$ is \emph{compatible} with $D_p$ or $\mathcal{D}$ if $D_p\geqslant D^-_{f_p}+\Sigma_{t\in\In_\Gamma(p)} m(p,t)(\red_p(t))$. We say that a supporting point of $D_p-\Sigma_{t\in\In_\Gamma(p)} m(p,t)(\red_p(t)) - D^-_{f_p}$ is a \emph{base point} of $(\mathcal{D},\mathcal{H})$ when $D_p$ is compatible with $H_p$. In this sense, a supporting point of $D_p$ is either a base point of $(\mathcal{D},\mathcal{H})$ or a pole of a nonconstant rational function in $H_p$. We say $\mathcal{D}$ and $\mathcal{H}$ are compatible if $D_p$ and $H_p$ are compatible for all $p\in\Gamma$. Note that the total number of base points must be finite when $\mathcal{D}$ and $\mathcal{H}$ are compatible.

\begin{example}
Figure~\ref{F:Local} illustrates the construction of a local diagram at $p\in\Gamma$ using a $2$-dimensional linear space $H_p$ of rational functions on a genus-1 curve $C_p$. We suppose that each $H_p$ contains a constant function. Consider a nonconstant rational function $f_p\in H_p$. Suppose $f_p$ has degree $3$ and the poles of $f_p$ are $u_1$, $u_2$ and $u_3$ with all ramification indices being $1$, and points $u_1$ and $u_2$ are marked points with associated tangent directions $t^u_1$ and $t^u_2$. Then $t^u_1$ and $t^u_2$ are incoming tangent directions with $m(p,t^u_1)=m(p,t^u_2)=-1$. Let $c$ and $c'$ be distinct values in $\kappa$ and suppose $f_p^{-1}(c)=\{v_1,v_2,v_3\}$ with all ramification indices being $1$ and $f_p^{-1}(c')=\{w_1,w_2\}$ with ramification indices of $w_1$ and $w_2$ being $1$ and $2$ respectively. Suppose $v_1$, $w_1$ and $w_2$ are marked points with associated tangent directions $t^v_1$, $t^w_1$ and $t^w_2$ respectively. Then
 $t^v_1$, $t^w_1$ and $t^w_2$ are outgoing tangent directions with $m(p,t^v_1)=1$, $m(p,t^w_1)=1$ and $m(p,t^w_2)=2$. Moreover, the corresponding local partition of $\Tan_\Gamma(p)$ is $\{\{t^u_1,t^u_2\},\{t^v_1\},\{t^w_1,t^w_2\}\}$. Moreover, a divisor $D_p$ on $C_p$ is compatible with $H_p$ if and only if $u_3$ is a supporting point of $D_p$.

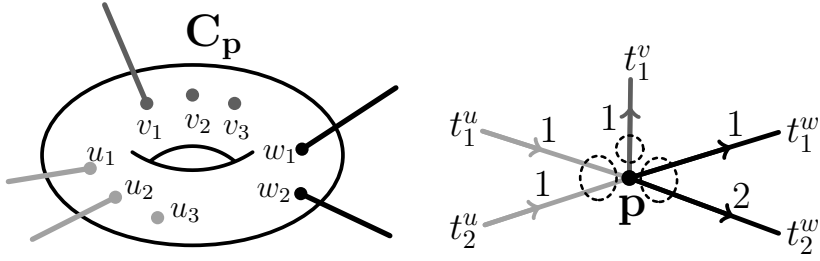
\begin{figure}[tbp]
\centering
\definecolor{light}{rgb}{0.37,0.37,0.37}
\definecolor{lighter}{rgb}{0.63,0.63,0.63}
\begin{tikzpicture}[line cap=round,line join=round,>=to,x=0.8cm,y=0.8cm]
\clip(0.5,1.3) rectangle (14.5,5.6);
\draw [rotate around={0.:(4.,3.)},line width=1.4pt] (4.,3.) ellipse (2.5 and 1.5);
\draw [shift={(4.,2.045)},line width=1.4pt]  plot[domain=0.8847484459534889:2.2568442076363042,variable=\t]({1.*1.105*cos(\t r)+0.*1.105*sin(\t r)},{0.*1.105*cos(\t r)+1.*1.105*sin(\t r)});
\draw [shift={(4.,4.566666666666666)},line width=1.4pt]  plot[domain=4.129475391428955:5.295302569340424,variable=\t]({1.*1.8166666666666664*cos(\t r)+0.*1.8166666666666664*sin(\t r)},{0.*1.8166666666666664*cos(\t r)+1.*1.8166666666666664*sin(\t r)});
\draw [line width=2pt,color=lighter] (2.3,2.78)-- (0.94,2.56);
\draw [line width=2pt,color=lighter] (2.72,2.3)-- (1.34,1.6);
\draw [line width=2pt,color=light] (3.24,3.86)-- (2.54,5.5);
\draw [line width=2pt] (5.82,3.1)-- (7.38,4.12);
\draw [line width=2pt] (5.8,2.36)-- (7.28,1.66);
\draw (2.07,3.307) node[anchor=north west] {\large $u_1$};
\draw (2.70,2.747) node[anchor=north west] {\large $u_2$};
\draw (3.46,2.36) node[anchor=north west] {\large $u_3$};
\draw (2.90,3.74) node[anchor=north west] {\large $v_1$};
\draw (3.70,3.84) node[anchor=north west] {\large $v_2$};
\draw (4.34,3.74) node[anchor=north west] {\large $v_3$};
\draw (5,3.34) node[anchor=north west] {\large $w_1$};
\draw (4.9,2.70) node[anchor=north west] {\large $w_2$};
\draw (3.7,5.5) node[anchor=north west] {\LARGE $\mathbf{C_p}$};
\draw [line width=1.8pt,color=light] (11.26,2.62)-- (11.28,4.26);
\draw [line width=1.8pt,color=lighter] (11.26,2.62)-- (8.82,3.4);
\draw [line width=1.8pt] (11.26,2.62)-- (13.74,3.38);
\draw [line width=1.8pt,color=lighter] (11.26,2.62)-- (8.86,1.84);
\draw [line width=1.8pt] (11.26,2.62)-- (13.74,1.7);
\draw [->,line width=1.8pt,color=lighter] (8.86,1.84) -- (9.784583741429966,2.140489715964739);
\draw [->,line width=1.8pt,color=light] (11.26,2.62) -- (11.274423333957904,3.802713384548106);
\draw [->,line width=1.8pt] (11.26,2.62) -- (12.9545034922563,3.139283328272092);
\draw [->,line width=1.8pt] (11.26,2.62) -- (12.991689915389893,1.9775989023553624);
\draw [->,line width=1.8pt,color=lighter] (8.82,3.4) -- (9.817968731303692,3.080977208845541);
\draw [rotate around={-88.56037399565767:(10.744719174055081,2.6199727248739078)},line width=1pt,dash pattern=on 1.5pt off 2pt] (10.744719174055081,2.6199727248739078) ellipse (0.2931348800710146cm and 0.241510347438542cm);
\draw [rotate around={-88.56037399565767:(11.75056279545427,2.589190831869807)},line width=1pt,dash pattern=on 1.5pt off 2pt] (11.75056279545427,2.589190831869807) ellipse (0.2931348800709107cm and 0.2415103474384564cm);
\draw [line width=1pt,dash pattern=on 1.5pt off 2pt] (11.265878556458151,3.102041629568428) circle (0.18cm);
\draw (8.1,3.9) node[anchor=north west] {\Large $t^u_1$};
\draw (8.1,2.2) node[anchor=north west] {\Large $t^u_2$};
\draw (11,5.0) node[anchor=north west] {\Large $t^v_1$};
\draw (13.7,3.8) node[anchor=north west] {\Large $t^w_1$};
\draw (13.7,2.1) node[anchor=north west] {\Large $t^w_2$};
\draw (9.6,3.8) node[anchor=north west] {\Large $1$};
\draw (9.5,2.9) node[anchor=north west] {\Large $1$};
\draw (10.6,4) node[anchor=north west] {\Large $1$};
\draw (12.7,3.9) node[anchor=north west] {\Large $1$};
\draw (12.8,2.7) node[anchor=north west] {\Large $2$};
\draw (10.9,2.5) node[anchor=north west] {\LARGE $\mathbf{p}$};
\begin{scriptsize}
\fill [lighter] (2.3,2.78) circle (2.5pt);
\fill [lighter] (2.72,2.3) circle (2.5pt);
\fill [light] (3.24,3.86) circle (2.5pt);
\fill [black] (5.82,3.1) circle (2.5pt);
\fill [black] (5.8,2.36) circle (2.5pt);
\fill [light] (4.,4.) circle (2.5pt);
\fill [light] (4.7,3.86) circle (2.5pt);
\fill [lighter] (3.42,1.96) circle (2.5pt);
\fill [black] (11.26,2.62) circle (3pt);
\end{scriptsize}
\end{tikzpicture}
\caption{A local diagram (right) at $p$ derived from $H_p$ on the curve $C_p$ (left) of genus $1$.}\label{F:Local}
\end{figure}
\end{example}

\begin{definition}{\rm{\bf (Diagrammatic and Solvable $(\mathcal{D},\mathcal{H})$)}} \label{D:DiaSolvDH}
We say that a refined pre-limit $g^1_d$ (respectively limit $g^1_d$) represented by $(\mathcal{D},\mathcal{H})$ on a saturated metrized complex $\mathfrak{C}=(\Gamma,\{C_p\}_{p \in \Gamma})$ is \emph{diagrammatic}  if $\mathcal{D}$ and $\mathcal{H}$ are compatible and the local diagrams induced by $\mathcal{H}$ form a global diagram. In addition, if the global diagram induced by $(\mathcal{D},\mathcal{H})$  is solvable, then we say that $(\mathcal{D},\mathcal{H})$ is \emph{solvable}.
\end{definition}

\subsection{Intrinsic Global Compatibility Conditions} \label{S:IGC}
In general, the extra condition in the definition of a (diagrammatic) limit $g^1_d$ over that of a (diagrammatic) pre-limit $g^1_d$ represented by $(\mathcal{D},\mathcal{H})$  does not guarantee solvability (see the example in Section~\ref{S:Nonsolvable}).  However, if $(\mathcal{D},\mathcal{H})$ is smoothable, then it must be solvable (see Theorem~\ref{T:main}).

On the other hand, solvability is only a necessary condition for $(\mathcal{D},\mathcal{H})$ being smoothable. In particular, it does not fully utilize information in $H_p$'s. In order to determine the smoothability of $(\mathcal{D},\mathcal{H})$, we will construct the bifurcation tree $\mathcal{B}$ and the natural surjection $\pi_{\mathcal{B}}: \Gamma \rightarrow \mathcal{B}$ from a solution to the characteristic equation and introduce the intrinsic global compatibility conditions which are compatibility conditions between $H_p$'s and the map $\pi_{\mathcal{B}}$.

Let $\rho$ be a rational function on $\Gamma$ with everywhere nonzero slopes.  For a point $p$ in $\Gamma$, recall that $\Tan_\Gamma(p)$ is the set of tangent directions emanating from $p$.  By $\Tan^{\rho+}_\Gamma(p)$, we denote the set of tangent directions in $\Tan_\Gamma(p)$ where $\rho$ locally increases. We canonically associate to $\rho$ a pair $(\mathcal{B},\pi_{\mathcal{B}})$ where (1) $\mathcal{B}$ is a metric tree called the \emph{bifurcation tree} with respect to $\rho$, and (2) $\pi_\mathcal{B}:\Gamma\rightarrow\mathcal{B}$ is a canonical projection from $\Gamma$ onto $\mathcal{B}$ (see details of the construction in Section~\ref{subS:BifTree}). In addition, $\mathcal{B}$ has a distinguished point called the \emph{root}. Moreover, $\pi_\mathcal{B}$ induces a push forward map $\pi_{\mathcal{B}*}$ from the tangent directions at all points in $\Gamma$ to the tangent directions at all points in $\mathcal{B}$. More precisely, if we let $\Tan^+_\mathcal{B}(x)$  be the set of forward tangent directions at $x\in\mathcal{B}$ (meaning that the distance function from the root increases along these directions), then by the construction in Section~\ref{subS:BifTree}, we have $\pi_{\mathcal{B}*}(t)\in\Tan^+_\mathcal{B}(\pi_\mathcal{B}(p))$ for all $p\in\Gamma$ and $t\in\Tan^{\rho+}_\Gamma(p)$.

 Using the above notions, we formulate the intrinsic global compatibility conditions. For a solvable  $(\mathcal{D},\mathcal{H})$ with a solution $\rho$, let $\mathcal{B}$ be the bifurcation tree  and $\pi_{\mathcal{B}}$ be the corresponding projection from $\Gamma$ onto $\mathcal{B}$.

\begin{definition}{\rm{\bf (Intrinsic Global Compatibility Conditions) }} \label{D:IGC}
A collection of non-constant rational functions $G=\{g_p\}_{p\in\Gamma}\in\mathcal{H}$ is called \emph{admissible} if it has one of the following equivalent properties:
\begin{enumerate}
 \item There is a function $\eta:\coprod_{x\in\mathcal{B}} \Tan^+_\mathcal{B}(x)\rightarrow \kappa$ such that $G\circ\Red(t) = \eta\circ\pi_\mathcal{B*}(t)$ for all  $t\in\coprod_{p\in\Gamma}\Tan^{\rho+}_\Gamma(p)$.
 \item For each pair of tangent directions $t_1\in \Tan^{\rho+}_\Gamma(p_1)$ and $t_2\in\Tan^{\rho+}_\Gamma(p_2)$ such that $\pi_\mathcal{B}(p_1)=\pi_\mathcal{B}(p_2)$ and $\pi_\mathcal{B*}(t_1)=\pi_\mathcal{B*}(t_2)$, we have  $g_{p_1}(\red_{p_1}(t_1))=g_{p_2}(\red_{p_2}(t_2))$.
\end{enumerate}
We say that $(\mathcal{D},\mathcal{H})$ satisfies the \emph{intrinsic global compatibility conditions} if $\mathcal{H}$ contains an \emph{admissible} collection of non-constant rational functions.
\end{definition}

\subsection{Finite Verification of Intrinsic Global Compatibility Conditions}\label{S:Finite}
In this section, we will show that the intrinsic global compatibility conditions can be verified in finitely many steps which makes  Theorem~\ref{T:main} an  effective smoothing criterion. 

A point $p\in\Gamma$ is called an \emph{ordinary point} of $\rho$ if (1) its valence is two, (2) the slopes of $\rho$ at $p$ in the two tangent directions have the same magnitude and opposite signs, and (3) the curve $C_p$ has genus $0$. Denote the set of ordinary points by $\mathcal{O}_{\rho}$. The points in  $\mathcal{E}_{\rho}:=\Gamma\setminus\mathcal{O}_{\rho}$ and the values in the image of $\rho$ restricted to $\mathcal{E}_{\rho}$ are called \emph{exceptional points} and \emph{exceptional values} of $\rho$ respectively. Note that $\mathcal{E}_{\rho}$ is a finite set.

The intrinsic global compatibility conditions are finitely verifiable since the verification can be restricted to the  set $\mathcal{E}_\rho$ of exceptional points which is finite. This is because for any ordinary point $p\in\mathcal{O}_{\rho}$, there is only one forward tangent direction $t$  in $\Tan^{\rho+}_\Gamma(p)$,  and by Lemma~\ref{L:tune} we can always find some non-constant $g_p\in H_p$ such that  $g_p(\red_p(t))=c$ for whatever desirable  $c\in\kappa$.

\begin{lemma} \label{L:tune}
	Let $C$ be a curve over $\kappa$ and $H$ is the linear subspace of the function field of $C$  spanned by $\{1,f\}$ where $f$ is a non-constant rational function on $C$. 
\begin{enumerate}
\item Suppose $u$ is a point on $C$ which is not a pole of $f$. Then for any value $c\in\kappa$, the space of all rational functions in $H$ taking value $c$ at the point $u$ is a line in $H$. 
\item Suppose $u_1$ and $u_2$ are points on $C$ which are not poles of $f$ and $f(u_1)\neq f(u_2)$. Then for any distinct values $c_1,c_2\in\kappa$, we can always find a unique non-constant $g\in H$ such that $g(u_1)=c_1$ and $g(u_2)=c_2$. 
\end{enumerate}
\end{lemma}
\begin{proof}
All rational functions in $g\in H$ can be expressed as a linear combination of $1$ and $f$, i.e., $g=\alpha+\beta f\in H$ for some $\alpha,\beta\in \kappa$. For (1), the space $\{g\in H\mid g(u)=c \}=\{\alpha + \beta f\mid \alpha+\beta f(u)=c$ is a line in $H$. For (2), the linear equations $\alpha+\beta f(u_1)=c_1$ and $\alpha+\beta f(u_2)=c_2$ have a unique solution for $\alpha$ and $\beta$. 
\end{proof}

More accurately, we have the following algorithm to determine whether $(\mathcal{D},\mathcal{H})$ is smoothable (see also Example~\ref{E:BPDH}).
\begin{algorithm} \label{A:IGC}																							
\textbf{Input:} A diagrammatic pre-limit  $g^1_d$ represented by $(\mathcal{D},\mathcal{H})$. \textbf{Output:} Whether  $(\mathcal{D},\mathcal{H})$  is smoothable.
\begin{enumerate}
\item Determine whether $(\mathcal{D},\mathcal{H})$ is solvable. If not, then $(\mathcal{D},\mathcal{H})$ is not smoothable. If yes, let $\mathcal{E}_{\rho}$ be the set of exceptional points.
 \item Fix a basis $\{1,f_p\}$ of $H_p$ for all exceptional points $p\in\mathcal{E}_\rho$ and consider a collection of finitely many variables $\{\alpha_p,\beta_p\}_{p\in\mathcal{E}_\rho}$. 
 \item For each pair $(t_1,t_2)$ such that $t_1\in \Tan^{\rho+}_\Gamma(p_1)$, $t_2\in\Tan^{\rho+}_\Gamma(p_2)$, $p_1,p_2\in\mathcal{E}_{\rho}$, $\pi_\mathcal{B}(p_1)=\pi_\mathcal{B}(p_2)$ and $\pi_{\mathcal{B}*}(t_1)=\pi_{\mathcal{B}*}(t_2)$, add a linear restriction 
 $$\alpha_{p_1} + f_{p_1}(\red_{p_1}(t_1))\beta_{p_1}=\alpha_{p_2} + f_{p_2}(\red_{p_2}(t_2))\beta_{p_2}.$$

 \item $(\mathcal{D},\mathcal{H})$ satisfies the intrinsic global compatibility conditions and thus is smoothable  if and only if  there exists a solution for $\{\alpha_p,\beta_p\}_{p\in\mathcal{E}_\rho}$ satisfying the linear restrictions in (3) and $\beta_p\neq 0$ for all $p\in\mathcal{E}_\rho$. In particular, in the case $\{\alpha_p,\beta_p\}_{p\in\mathcal{E}_\rho}$ is a solution.  If we let $g_p=\alpha_p+\beta_p f_p$ for $p\in\mathcal{E}_\rho$ and $g_p=f_p$ for $p\in\mathcal{O}_\rho$, then $\{g_p\}_{p\in\Gamma}$ is admissible. 
\end{enumerate}
\end{algorithm}

\section{Morphisms of  Saturated Metrized Complexes and Their Relations to Smoothability}\label{S:NonArch}

In this section, we give precise definitions of pseudo-harmonic morphisms and harmonic morphisms of saturated metrized complexes and smoothability of a pre-limit $g^1_d$ (Definition~\ref{D:smoothable}) and study their close relations to  smoothable limit $g^1_d$s (Subsection~\ref{subS:SmoothProp}).  We use notions from the theory of  Berkovich analytic spaces with explanations in Appendix~\ref{S:Berkovich}. The reader is urged to use references (e.g. \cite{Berkovich12}, \cite{BPR11}) with an elaborate treatment of this analytical construction.

\subsection{Pseudo-Harmonic Morphisms Harmonic Morphisms of Saturated Metrized Complexes}\label{S:harmmor}

We give a natural extension of the notion of harmonic morphism of metrized complexes from \cite{ABBR13} and \cite{ABBR14} to saturated metrized complexes. We start with the notion of pseudo-harmonic morphism of saturated metrized complexes.

Let $\mathfrak{C}$ and $\mathfrak{C}'$ be saturated metrized complexes. The underlying metric graphs of $\mathfrak{C}$ and $\mathfrak{C}'$ are $\Gamma$ and $\Gamma'$ respectively, and the associated curves of $\mathfrak{C}$ and $\mathfrak{C}'$ are $\{C_p\}_{p\in\Gamma}$ and $\{C'_q\}_{q\in\Gamma'}$ respectively.

\begin{definition}{\rm {\bf (Pseudo-harmonic morphisms of saturated metrized complexes)}} \label{D:pseu_har_morph}
A \emph{pseudo-harmonic morphism} between $\mathfrak{C}$ and $\mathfrak{C}'$ consists of the data $\{\phi_\Gamma, \{ \phi_p \}_{p \in \Gamma }\}$ where $\phi_\Gamma: \Gamma \rightarrow \Gamma'$ is a continuous finite surjective piecewise-linear map with integral slopes (which is called a \emph{pseudo-harmonic morphism} between $\Gamma$ and $\Gamma'$) and $\phi_p: C_p \rightarrow C'_{\phi_\Gamma(p)}$ is a finite morphism of curves that satisfies the following compatibility conditions:
\begin{enumerate}{\label{compat_prop}}
\item  For all $p\in\Gamma$, two tangent directions $t_1, t_2 \in \Tan_\Gamma(p)$ are mapped to the same tangent direction $t' \in \Tan_{\Gamma'}(\phi_\Gamma(p))$ by the induced map of $\phi_\Gamma$ if and only if the marked points corresponding to $t_1$ and $t_2$ are mapped to the marked point corresponding to the tangent direction $t'$ by $\phi_p$.

\item  For all $p\in\Gamma$ and all tangent directions $t \in \Tan_\Gamma(p)$, the expansion factor $d_t(\phi_\Gamma)$ of $\phi_\Gamma$ along $t$ coincides with the ramification index of $\phi_p$ at the marked point corresponding to the tangent direction $t$. Here the expansion factor $d_t(\phi_\Gamma)$ is the absolute value of the slope of $\phi_\Gamma$ along $t$, i.e., the ratio of the length between $\phi_\Gamma(p)$ and $\phi_\Gamma(q)$ over the length between $p$ and $q$ where $q$ is near $p$ in the direction $t$.
\end{enumerate}

\qed
\end{definition}

\begin{definition}{\rm {\bf (Harmonic morphisms of saturated metrized complexes)}} \label{D:har_morph}
 A pseudo-harmonic morphism $\{\phi_\Gamma, \{ \phi_p \}_{p \in \Gamma }\}$ between  saturated metrized complexes $\mathfrak{C}$ and $\mathfrak{C}'$ is called a \emph{harmonic morphism} at a point $p\in\Gamma$ if $\phi_\Gamma$ is harmonic at $p$ and the degree of $\phi_\Gamma$ at $p\in\Gamma$ is the same as the degree of $\phi_p$. More explicitly, we say $\phi_\Gamma$ is harmonic at $p$ if it is a pseudo-harmonic morphism of metric graphs satisfying the following \emph{balancing} condition:
 for any tangent direction $t' \in \Tan_{\Gamma'}(\phi(p))$, the  sum of the expansion factors $d_t(\phi_\Gamma)$ over all tangent directions $t$ in $\Tan_\Gamma(p)$ that map to $t'$,  i.e., the integer $$\sum_{t \in \Tan_\Gamma(p),~t \mapsto t'}d_t(\phi_\Gamma),$$ is independent of $t'$ and is called the \emph{degree} of $\phi_\Gamma$ at $q$.

 We say $\phi_\Gamma$ is a harmonic morphism of metric graphs if $\phi_\Gamma$ is harmonic at each $p\in\Gamma$, and $\{\phi_\Gamma, \{ \phi_p \}_{p \in \Gamma }\}$ is a harmonic morphism of saturated metrized complexes if $\{\phi_\Gamma, \{ \phi_p \}_{p \in \Gamma }\}$ is harmonic at  each $p\in\Gamma$.
\qed
\end{definition}

\begin{remark}
More precisely, our notion of harmonic morphism of saturated metrized complexes corresponds to the notion of \emph{finite} harmonic morphism of metrized complexes in \cite{ABBR13}.
\end{remark}

The notion of harmonic morphism allows us to define the notion of isomorphism of saturated metrized complexes.
Two saturated metrized complexes $\mathfrak{C_1}$ and $\mathfrak{C_2}$ are \emph{isomorphic} if there is a harmonic morphism  $\mathfrak{C} \phi_1: \mathfrak{C}_1 \rightarrow \mathfrak{C}_2$ and a harmonic morphism $\mathfrak{C}\phi_2 : \mathfrak{C_2} \rightarrow \mathfrak{C_1}$ such that $\mathfrak{C}\phi_2 \circ \mathfrak{C\phi_1}$ and $\mathfrak{C}\phi_1 \circ \mathfrak{C\phi_2}$ are identity maps on $\mathfrak{C}_1$ and $\mathfrak{C}_2$ respectively.

The following theorem summarizes the lifting results of saturated metrized complexes which is a direct corollary of the lifting theorems for metrized complexes in Amini et al. \cite{ABBR13}. See Appendix~\ref{subS:SkelRefMet} for the association of a saturated metrized complex to a skeleton of a Berkovich analytic curve.

\begin{theorem} \label{T:lifting}
We have the following lifting properties for saturated metrized complexes and harmonic morphisms of saturated metrized complexes:
\begin{enumerate}

  \item Let $\mathfrak{C}$ be a saturated metrized complex of curves over $\kappa$. There exists a smooth curve $X/\mathbb{K}$ and a skeleton $\Sigma$ of the Berkovich analytification $X^{\rm an}$ of $X$ such that $\mathfrak{C}$ is isomorphic to the associated saturated metrized complex of $\Sigma$.
  \item If $\mathfrak{C}\phi:\mathfrak{C}\rightarrow\mathfrak{C'}$ is a harmonic morphism of saturated metrized complexes where $\mathfrak{C'}$ is isomorphic to the saturated metrized complex associated to a skeleton of the Berkovich analytification $X'^{\rm an}$ of a smooth curve $X'/\mathbb{K}$, then there exists a finite morphism $\phi:X\rightarrow X'$ of curves over $\mathbb{K}$ such that $\mathfrak{C}$ is isomorphic to the associated saturated metrized complex of a skeleton of $X^{\rm an}$ and $\phi$ induces $\mathfrak{C}\phi$.
\end{enumerate}
\end{theorem}

\begin{proof}
For part 1, we can choose a vertex set $V$ of the underlying metric graph $\Gamma$ of $\mathfrak{C}$ such that the associated curves to points in $\Sigma\setminus V$ are all projective lines over $\kappa$ with two marked points. Let $\mathfrak{C}_0$ be the metrized complex derived from restricting $\mathfrak{C}$ to $V$ (Remark~\ref{R:restriction}). We can lift $\mathfrak{C}_0$ to a smooth curve $X$ over $\mathbb{K}$ with skeleton $\Sigma$ by the lifting theorem of metrized complexes. It is then straight forward to verify that $\mathfrak{C}$ is isomorphic to the associated saturated metrized complex of $\Sigma$.

For part 2, we choose vertex sets of $\mathfrak{C}$ and $\mathfrak{C}'$ fine enough to derive metrized complexes $\mathfrak{C}_0$ and $\mathfrak{C}'_0$ respectively such that the harmonic morphism $\mathfrak{C}\phi$ of saturated metrized complex can be restricted to harmonic morphism $\mathfrak{C}\phi_0:\mathfrak{C}_0\rightarrow \mathfrak{C}'_0$ of metrized complexes. By the lifting theorem of harmonic morphism of metrized complexes, we can  lift $\mathfrak{C}\phi_0$ to a finite morphism $\phi:X\rightarrow X'$. Then $\mathfrak{C}\phi$ is induced by $\phi$.
\end{proof}

The rest of this subsection contains the notion of pullback of a harmonic morphism and the notion of modification of saturated metrized complex.  This will be used in the proof of the smoothing critrion (Section~\ref{S:SmoothingProof}).

\begin{remark}{\rm {\bf (Pullback divisor and pullback function of a harmonic morphism)} \label{R:pullback}}
Let $\mathfrak{C}\phi=(\phi_\Gamma,\{\phi_p\}_{p\in\Gamma})$ be a harmonic morphism between saturated metrized complexes $\mathfrak{C}$ and $\mathfrak{C'}$ whose underlying metric graphs are $\Gamma$ and $\Gamma'$ respectively.  Let $p'$ be a point in $\Gamma'$ and $u'$ be a point in the associated curve $C'_{p'}$ of $p'$ in $\mathfrak{C'}$. Let $E_{u'}$ be the degree one effective divisor on $\mathfrak{C'}$ whose only supporting point is $u'$.
Then we can naturally associate a \emph{pullback divisor} $\mathfrak{C}\phi^*(E_{u'})\in\Div(\mathfrak{C})$ of $E_{u'}$ defined as follows:
(1) the tropical part of $\mathfrak{C}\phi^*(E_{u'})$ is the pullback divisor $\phi_\Gamma^*((p'))\in\Div(\Gamma)$ of the divisor $(p')\in\Div(\Gamma')$, (2)
the $C_p$-part of $\mathfrak{C}\phi^*(E_{u'})$ is the pullback divisor $\phi_p^*((u'))\in\Div(C_p)$ of the divisor $(u')\in\Div(C'_{p'})$ if $p\in\phi_\Gamma^{-1}(p')$, and (3) the $C_p$-part of $\mathfrak{C}\phi^*(E_{u'})$ is $0$ if $p\notin\phi_\Gamma^{-1}(p')$. Note that the properties of harmonic morphisms guarantee that $\mathfrak{C}\phi^*(E_{u'})$ is a well-defined divisor on $\mathfrak{C}$. We may also simply call the pullback divisor of $E_{u'}$ as the pullback divisor of the point $u'$ sometimes. Since we can formally write $E_{u'}=(u')$ (Remark~\ref{R:DivFreeAb}), we can also formally write
$$\mathfrak{C}\phi^*((u'))=\sum_{p\in\phi_\Gamma^{-1}(p')}\phi_p^*((u')).$$

Moreover, by letting $\mathfrak{C}\phi^*$ preserve linear combinations, we can naturally associate a pullback divisor $\mathfrak{C}\phi^*(\mathcal{D'})$ on $\mathfrak{C}$ to all divisors $\mathcal{D}'$ on $\mathfrak{C'}$.

On the other hand, if $\mathfrak{f}'=(f'_{\Gamma'},\{f'_{p'}\}_{p' \in \Gamma'})$ is a rational function on $\mathcal{C}'$, then we can also pullback $\mathfrak{f}'$ to a rational function $\mathfrak{C}\phi^*(\mathfrak{f}')$ on $\mathfrak{C}$ in a natural way: the tropical part of $\mathfrak{C}\phi^*(\mathfrak{f}')$ is the pullback function $\phi_\Gamma^*(f'_{\Gamma'})$ of $f'_{\Gamma'}$, and the $C_p$-part of $\mathfrak{C}\phi^*(\mathfrak{f}')$ is the pullback function $\phi_p^*(f'_{\phi_\Gamma(p)})$ of the rational function $f'_{\phi_\Gamma(p)}$. It is straightforward to verify that the principal divisor associated to the pullback function of $\mathfrak{f}'$ is the same as the pullback divisor of the principal divisor associated to $\mathfrak{f}'$.

\qed
\end{remark}

A metric graph $\Gamma^\text{mod}$ is called a \emph{modification} of a metric graph $\Gamma$ if $\Gamma$ is isometric to a subgraph of $\Gamma^\text{mod}$ and the genus of $\Gamma^\text{mod}$ is the same as the genus of $\Gamma$.

A saturated metrized complex $\mathfrak{C}^\text{mod}$ is called a \emph{modification} of a saturated metrized complex $\mathfrak{C}$ if  (1) the metric graph $\Gamma^\text{mod}$ underlying $\mathfrak{C}^\text{mod}$ is a modification of the metric graph $\Gamma$ underlying $\mathfrak{C}$, (2) $g(\mathfrak{C}^\text{mod})=g(\mathfrak{C})$, and (3) for each point $p\in\Gamma$, when $p$ is also considered as a point in $\Gamma^\text{mod}$, the curve associated to $p$ in $\mathfrak{C}$ is identical to the curve associated to $p$ in $\mathfrak{C}^\text{mod}$ and the reduction map of $\mathfrak{C}$ at $p$ is identical to the reduction map of $\mathfrak{C}^\text{mod}$ at $p$ restricted to $\Tan_\Gamma(p)$ (note that in this setting we have $\Tan_{\Gamma^\text{mod}}(p)\supseteq\Tan_\Gamma(p)$).

\begin{remark}
If $\mathfrak{C}^\text{mod}$ is a modification of $\mathfrak{C}$ as saturated metrized complexes with underlying metric graphs $\Gamma^\text{mod}$ and $\Gamma$ respectively. Then by the retraction map from $\Gamma^\text{mod}$ to $\Gamma$, a divisor on $\Gamma^\text{mod}$ also naturally retracts to a unique divisor on $\Gamma$. In addition, a divisor on $\mathfrak{C}^\text{mod}$ naturally retracts to a divisor on $\mathfrak{C}$ and any specialization map of divisors factors through the retraction map of divisors.
\end{remark}

 \subsection{Smoothability}

The following theorem is the analogue of the specialization theorem (Theorem 5.9 in \cite{AB12}) for saturated metrized complexes.

In Appendix~\ref{S:Berkovich}, we show that for a smooth curve $X/\mathbb{K}$, we can associate a saturated metrized complex $\mathfrak{C}(\Sigma)$ to a Berkovich skeleton $\Sigma$ of $X^{\rm an}$, and there exist a specialization map $\tau_{*}$ that takes a divisor on $X$ to a divisor on $\mathfrak{C} (\Sigma)$ and a reduction map that takes a rational function on $X$ to a rational function on $\mathfrak{C}(\Sigma)$.

 \begin{theorem} \label{T:DHspecialization}
 For any $g^r_d$ on $X$ represented by the pair $(D,H)$ where $H$ is an $(r+1)$-dimensional linear space of rational functions on $X$, the data $(\tau_{*}(D), \{H_p\}_{p \in \Gamma})$ where  $H_p$ is the image of $H$ under the reduction map at $p$ is a limit $g^r_d$ on $\mathfrak{C}(\Sigma)$.
 \end{theorem}

\begin{proof}
By Lemma~\ref{dimpre_lem}, the dimension of the space $H$ is preserved by the specialization map. From Lemma~\ref{L:specialization}, we know that for any effective divisor $\mathcal{E}=(E_{\Gamma}, \{ E_p \}_{p \in \Gamma})$ such that $E_p$ has support in $S_p$ for every $p \in \Gamma$, there exists an effective divisor $E$ on $X$  such that $\tau_*(E)=\mathcal{E}$.  Since $(D,H)$ represents a $g^r_d$ on $X$, there must be a rational function $f \in H$ such that $D-E+{\rm div}(f) \geq 0$. We apply the specialization map to this inequality. Using the property that  the specialization map is a homomorphism between divisor groups that preserves effective divisors combined with Theorem~\ref{T:PoicareLelong} , we conclude that  $\tau_*(D)-\mathcal{E}+{\rm div}(\mathfrak{f}) \geq 0$.
\end{proof}

The following definition of smoothability of a pre-limit (or limit) $g^r_d$ represented by $(\mathcal{D}, \mathcal{H})$ accounts for whether $(\mathcal{D}, \mathcal{H})$ can be ``lifted'' to some $g^r_d$ represented by $(D,H)$.

\begin{definition}{\rm{{ \bf (Smoothable Pre-Limit Linear Series and Limit Linear Series)}}}\label{D:smoothable}
A  pre-limit $g^r_d$ (respectively, a limit $g^r_d$) represented by $(\mathcal{D}, \mathcal{H})$ on a saturated metrized complex $\mathfrak{C}$ is said to be \emph{smoothable} if there exists a smooth proper curve $X$ over $\mathbb{K}$ and  a skeleton $\Sigma$ of the Berkovich analytification $X^{\rm an}$  of $X$ such that $\mathfrak{C}$ is isomorphic to the saturated metrized complex associated to $\Sigma$ and there exists a $g^r_d$ on $X$ which is represented by $(D,H)$ such that the associated pre-limit $g^r_d$ (respectively, a limit $g^r_d$) on $\mathfrak{C}$  is represented by $(\mathcal{D},\mathcal{H})$.
\end{definition}

\begin{remark}
Since $\mathbb{K}$ is a large field with value group $\mathbb{R}$, we have no restrictions on the edge lengths of the underlying metric graph of $\mathfrak{C}$ and the above definition of smoothability is in the most general form.
\end{remark}
\begin{remark}
Theorem~\ref{T:DHspecialization} actually tells us that we do not need to distinguish the notion of smoothable pre-limit linear series and the notion of smoothable limit linear series. The extra restriction on limit linear series over pre-limit linear series is guaranteed by smoothability.
\end{remark}

\subsection{Smoothable Pre-Limit $g^1_d$ and Harmonic Morphisms}\label{subS:SmoothProp}
Let $K$ be any algebraically closed field. Let $H$ be a two-dimensional linear space of rational functions on a smooth proper curve $X$ over $K$. Assume constant functions are contained in $H$. Then all nonconstant rational functions in $H$ have the same poles and same order on the poles. We say the degree of $H$ is the degree of any nonconstant function in $H$. Therefore, $H$ defines a morphism $\phi:X\rightarrow \mathbb{P}^1_K$ of the same degree where $\mathbb{P}^1_K$ has a marked point $\infty$ such that all nonconstant rational functions $f$ in $H$ factor through $\phi$ via a degree one rational function on $\mathbb{P}^1_K$ such that all poles of $f$ map to $\infty$ in $\mathbb{P}^1_K$. Conversely, given a morphism $\phi:X\rightarrow \mathbb{P}^1_K$ where $\mathbb{P}^1_K$ has a marked point $\infty$, we know that the linear space $\mathcal{L}((\infty))$ associated to the divisor $(\infty)$ is a degree one two-dimensional linear space of rational functions on $\mathbb{P}^1_K$ which pulls back to a two-dimensional linear space $H$ of rational functions on $X$ by $\phi$ and has the same degree as $\phi$. Note that $\mathcal{L}((\infty))$ contains constant functions on $\mathbb{P}^1_K$ and $H$ contains constant functions on $X$.

Suppose $\mathfrak{C}(T)$ is a genus zero saturated metrized complex over $\kappa$ with underlying metric tree $T$. Then any two distinct effective divisors on $\mathfrak{C}(T)$ differ by a principal divisor associated to a degree one rational function on $\mathfrak{C}(T)$. Moreover, by embedding $T$ isometrically into the analytification of $\mathbb{P}^1_\mathbb{K}$ and using the lifting theorem (Theorem~\ref{T:lifting}), we may consider $T$ as a skeleton of $\mathbb{P}^1_\mathbb{K}$ whose associated saturated metrized complex is naturally isomorphic to $\mathfrak{C}(T)$. In addition, since $\mathfrak{C}(T)$ is of genus zero, all divisors and rational functions on $\mathfrak{C}(T)$ are liftable to $\mathbb{P}^1_\mathbb{K}$ with linear equivalence respected.

Now let $\mathfrak{C}$ be a saturated metrized complex with underlying metric graph $\Gamma$.

\begin{lemma} \label{L:SmoothRefined}
A smoothable pre-limit $g^r_d$ on $\mathfrak{C}$ is refined limit $g^r_d$.
\end{lemma}
\begin{proof}
A smoothable pre-limit  $g^r_d$ is a limit $g^r_d$ by Theorem~\ref{T:DHspecialization}.

We show that a smoothable limit $g^r_d$ satisfies the two properties of refined limit $g^r_d$. For the first property of refined limit $g^1_d$, we note that the constant function is contained in $H$ where $(D,H)$ is any smoothing of the limit $g^r_d$ and it follows from the Poincar\'e-Lelong Formula (Theorem~\ref{T:PoicareLelong}). Using the characterization of the image of the specialization map obtained in Lemma~\ref{L:specialization}, we deduce that a smoothable limit $g^r_d$ satisfies the second property of a refined limit $g^r_d$.
\end{proof}

\begin{remark}\label{R:smoothdiag}
For a smoothable pre-limit $g^1_d$ represented by $(\mathcal{D},\mathcal{H})$ on $\mathfrak{C}$, we must have $\mathcal{D}$ and $\mathcal{H}$ compatible to each other.
\end{remark}

\begin{theorem}\label{T:SmoothingHar}
 A pre-limit $g^1_d$ represented by $(\mathcal{D},\mathcal{H})$ on $\mathfrak{C}$ is smoothable if and only if there exists a modification $\mathfrak{C}^{\rm mod}$ of $\mathfrak{C}$ and a harmonic morphism $\mathfrak{C}\phi^{\rm mod}=(\phi_{\Gamma^{\rm mod}},\{\phi_p\}_{p \in \Gamma^{\rm mod}})$ from $\mathfrak{C}^{\rm mod}$ to a genus zero saturated metrized complex $\mathfrak{C}(T)$ such that
 \begin{enumerate}
   \item \label{Smoothing_Cond1} $\mathcal{D}$ is the retract onto $\mathfrak{C}$ of the pullback divisor on $\mathfrak{C}^{\rm mod}$ by $\mathfrak{C}\phi^{\rm mod}$ over a degree one effective divisor $(u')$ on $\mathfrak{C}(T)$, and
   \item \label{Smoothing_Cond2} for each $p\in\Gamma$, if $g_p$ is the $C_p$-part of the pullback function of a rational function on $\mathfrak{C}(T)$ of degree $\leqslant1$ with only possible pole at $u'$, then $g_p\in H_p$.
 \end{enumerate}
\end{theorem}

\begin{proof}
First by Lemma~\ref{L:SmoothRefined} and Remark~\ref{R:smoothdiag}, we may assume that $(\mathcal{D},\mathcal{H})$ is refined and let us first assume $(\mathcal{D},\mathcal{H})$ is base-point free.

If $(\mathcal{D},\mathcal{H})$ is smoothable, there exists a smooth proper curve $X$ over $\mathbb{K}$ and a skeleton  $\Sigma(X^{\rm an},V)$ of $X$ such that the saturated metrized complex associated to $(X^{\rm an},V)$ is isomorphic to $\mathfrak{C}$ (we identify them in the following for simplicity of discussion and thus $\Gamma=\Sigma(X^{\rm an},V)$). Let $(D,H)$ represent a $g^1_d$ on $X$ corresponding to $(\mathcal{D},\mathcal{H})$, which is also base-point free.

Consider a map $\phi: X \rightarrow \mathbb{P}^1_\mathbb{K}$ defined by $H$ where $\mathbb{P}^1_\mathbb{K}$ is marked with a point $\infty_\mathbb{K}$ and $D$ is the pullback divisor of $\phi$ over the point $\infty_\mathbb{K}$ of $\mathbb{P}^1_\mathbb{K}$. Then $H$ is the pullback of the two-dimensional linear space $\mathcal{L}((\infty_\mathbb{K}))$ associated to $(\infty_\mathbb{K})\in\Div(\mathbb{P}^1_\mathbb{K})$ by $\phi$.
The analytification functor induces a map $\phi^{\rm an}: X^{\rm an} \rightarrow \mathbb{P}^1_{\rm Berk}$, where $\mathbb{P}^1_{\rm Berk}$ is the Berkovich analytification  of $\mathbb{P}^1_\mathbb{K}$. We restrict the map $\phi^{\rm an}$ to the skeleton $\Gamma$ to obtain a surjective map $\phi_\Gamma: \Gamma \rightarrow T$ where $T$ is a skeleton of $\mathbb{P}^1_{\rm Berk}$. Let $\Gamma^{\rm mod}$ be $(\phi^{\rm an})^{-1}(T)$ which is a skeleton of $X$ such that $\Gamma^{\rm mod}\supseteq \Gamma$. Restricting the map $\phi^{\rm an}$ to the skeleton $\Gamma^{\rm mod}$, we get a map $\phi_{\Gamma^{\rm mod}}: \Gamma^{\rm mod} \rightarrow T$. Let $\mathfrak{C}^{\rm mod}$ be a saturated metrized complex associated to $X$ with skeleton $\Gamma^{\rm mod}$. Let $\mathfrak{C}(T)$ be a saturated metrized complex associated to $\mathbb{P}^1_{\rm Berk}$ with skeleton $T$.

The reduction to $\mathfrak{C}^{\rm mod}$ of $H$ and the reduction to $\mathfrak{C}(T)$ of $\mathcal{L}((\infty_\mathbb{K}))$ define maps $\phi_p: C_p \rightarrow C'_{\phi_{\Gamma^{\rm mod}}(p)}$ for all $p\in \Gamma^{\rm mod}$ which respect $\phi$. Then the data $(\phi_\Gamma, \{\phi_p\}_{p \in \Gamma})$ satisfies the compatibility conditions of a pseudo-harmonic morphism (Definition~\ref{D:pseu_har_morph}) and the data $(\phi_{\Gamma^{\rm mod}}, \{\phi_p\}_{p \in \Gamma^{\rm mod}})$ satisfies the conditions of a harmonic morphism (Definition~\ref{D:har_morph}). Furthermore, the specialization of $D$ to $\mathfrak{C}^{\rm mod}$ is the pullback divisor  $\mathcal{D}^{\rm mod}$ by $(\phi_{\Gamma^{\rm mod}}, \{\phi_p\}_{p \in \Gamma^{\rm mod}})$ over a degree one effective divisor, denoted by $(u')$, which is the specialization of the divisor $(\infty_\mathbb{K})\in\Div(\mathbb{P}^1_{\mathbb{K}})$ to $\mathfrak{C}(T)$. In addition, $\mathcal{D}$ is the specialization of $D$ to $\mathfrak{C}$ which is also the retract of $\mathcal{D}^{\rm mod}$ to $\mathfrak{C}$. On the other hand, since $H$ is the pullback of $\mathcal{L}((\infty_\mathbb{K}))$ by $\phi$ and the space of rational functions $\mathfrak{f}'$ on $\mathfrak{C}(T)$ of degree $\leqslant1$ with only possible pole at $u'$ is exactly the reduction to $\mathfrak{C}(T)$ of $\mathcal{L}((\infty_\mathbb{K}))$, the $C_p$-part of the pullback function of $\mathfrak{f}'$ must be an element of $H_p$, the reduction of $H$ at $C_p$.

Conversely, suppose that there is a harmonic morphism $\mathfrak{C}\phi^{\rm mod}=(\phi_{\Gamma}, \{\phi_p\}_{p \in \Gamma^{\rm mod}})$ between saturated metrized complexes $\mathfrak{C}^{\rm mod}$ which is a modification of $\mathfrak{C}$ and a genus zero metrized complex $\mathfrak{C}(T)$.  Let $\mathcal{D}$ be the retract onto $\mathfrak{C}$ of the pullback divisor $\mathcal{D}^{\rm mod}$ over a point $u'\in C'_{p'}$ in $\mathfrak{C}(T)$ by $\mathfrak{C}\phi^{\rm mod}$. For each $p'\in T$, restricted to the $C'_{p'}$-parts, the rational functions $\mathfrak{f}'$ on $\mathfrak{C}(T)$ of degree $\leqslant1$ with only possible pole at $u'$ make up a two dimensional linear space $H'_{p'}$ of rational functions on $C'_{p'}$. Pulling back $H'_{p'}$ for all $p'\in T$ by $\mathfrak{C}\phi^{\rm mod}$, we obtain two dimensional linear spaces $H_p$ for all $p\in\Gamma$. In this way, we get the data $(\mathcal{D},\mathcal{H})$ where $\mathcal{H}=\{H_p\}$.

By the lifting theorem (Theorem~\ref{T:lifting}), we can lift $\mathfrak{C}\phi^{\rm mod}: \mathfrak{C}^{\rm mod}\rightarrow\mathfrak{C}(T)$ to a finite morphism $\phi: X \rightarrow \mathbb{P}^1_\mathbb{K}$ of $\mathbb{K}$-curves. Mark a point in $\mathbb{P}^1_\mathbb{K}$ by $\infty_\mathbb{K}$ whose reduction to $\mathfrak{C}(T)$ is $u'$. Let $D$ be the pullback divisor of $\phi$ over the point $\infty_\mathbb{K}$ and $H$ be the pullback of the two-dimensional linear space $\mathcal{L}((\infty_\mathbb{K}))$ associated to the divisor $(\infty_\mathbb{K})$ by $\phi$. Then the lifting theorem also guarantees that $(\mathcal{D},\mathcal{H})$ can be smoothed to $(D,H)$.

Suppose that $(\mathcal{D},\mathcal{H})$ has base points $u_1,\cdots,u_m$ with orders $\alpha_1,\cdots,\alpha_m$ respectively. Then $\mathcal{D}'=\mathcal{D}-\sum_{i=1}^m\alpha_i(u_i)$ is base-point free. Since $(\mathcal{D},\mathcal{H})$ is smoothable if and only if $(\mathcal{D}',\mathcal{H})$ is smoothable, the smoothing criterion on $(\mathcal{D}',\mathcal{H})$ can be extended to the smoothing criterion on $(\mathcal{D},\mathcal{H})$.

\end{proof}

\section{Bifurcation Trees and Partition Trees}\label{S:BifParTrees}

In this section, we investigate in detail the definitions and properties of bifurcation trees and partition trees which are notions employed in the smoothing criterion. A solvable diagrammatic pre-limit $g^1_d$ has a solution $\rho$ (unique up addition by a constant function) to its characteristic equation, from which we can canonically construct a rooted metric tree called the bifurcation tree $\mathcal{B}$ and a projection $\pi_\mathcal{B}:\Gamma\rightarrow \mathcal{B}$  for $\rho$ (Subsection~\ref{subS:BifTree}). Partition trees are derived from the bifurcation tree $\mathcal{B}$ by suitably gluing the branches of  $\mathcal{B}$ (Subsection~\ref{subS:ParTree}). In particular, any metric tree $T$ underlying the genus-$0$ metrized complex $\mathfrak{C}(T)$ in the commutative diagram in Section~\ref{S:introduction} must be a partition tree. The treatment will be expanded in Section~\ref {S:ParTreeComp} and \ref{S:SmoothingProof} leading to the proof of the smoothing criterion.

\subsection{Bifurcation trees} \label{subS:BifTree}

 Let $\rho$ be a rational function on $\Gamma$ with everywhere nonzero slopes and let $\hat{\rho}:=\rho- \min \rho$ be the normalized function of $\rho$ with minimum value zero.   For a real number $c$ and $* \in \{\geq,\leq,<,>,=\}$, the set $S^\rho_{*c}$ is defined as $\{ p \in \Gamma |~\rho(p) * c \}$. Denote the set of connected components of $S^\rho_{*c}$ by $\Comp(S^\rho_{*c})$.

For each value $c\in\imag \rho$, the connected components of $S^\rho_{\geqslant c}$ are called \emph{closed superlevel components} at $c$, and the connected components of $S^\rho_{>c}$ are called \emph{open superlevel components} at $c$.

\begin{remark} \label{R:SLS}
Here we summarize some facts about closed and open superlevel components that are immediate from their definition.
\begin{enumerate}

\item For $c\in\imag\rho$, for any open superlevel component $\beta\in \Comp(S^\rho_{>c})$, there exists a unique closed superlevel component $\alpha\in \Comp(S^\rho_{\geqslant c})$ such that $\alpha\supseteq\beta$. 
\item For each $c,c'\in\imag\rho$ such that $c'\leqslant c$ and $\alpha\in \Comp(S^\rho_{\geqslant c})$, there exists a unique element $\alpha'\in \Comp(S^\rho_{\geqslant c'})$ such that $\alpha'\supseteq\alpha$.
\item $\Comp(S^\rho_{\geqslant \min_{p \in \Gamma}\rho(p)})$ is a singleton whose element is the whole metric graph $\Gamma$. 
\item For $\alpha_1\in \Comp(S^\rho_{\geqslant c_1})$ and $\alpha_2\in \Comp(S^\rho_{\geqslant c_2})$, there exists a largest $c_3\in\imag\rho$ such that there exists $\alpha_3\in \Comp(S^\rho_{\geqslant c_3})$ with $\alpha_3\supseteq \alpha_1\bigcup\alpha_2$. In particular, $c_3\leqslant\min(c_1,c_2)$ and $\alpha_3$ is the unique smallest closed superlevel component containing $\alpha_1\bigcup\alpha_2$. 

\end{enumerate}

\end{remark}

We define the notion of bifurcation tree associated to $\rho$ as follows (also see Example~\ref{E:Bif}).

\begin{definition}\label{D:BifTree}{\rm{\bf (Bifurcation Tree)}}
 Consider a rational function $\rho$ with everywhere nonzero slopes. The \emph{bifurcation tree} $\mathcal{B}$ with respect to $\rho$ is a rooted metric tree constructed in the following way:
 \begin{enumerate}
 \item By abuse of notation, we also use $\mathcal{B}$ to represent the set of points of $\mathcal{B}$. We identify the set of points of $\mathcal{B}$ with the set of all closed superlevel components of $\rho$ by the bijection  $\iota_\mathcal{B}:\mathcal{B}\rightarrow\coprod_{c\in\imag \rho}\Comp(S^\rho_{\geqslant c})$.

 \item We assign a metric structure $d_\mathcal{B}$ to $\mathcal{B}$. For $x_1,x_2\in\mathcal{B}$, let $x_1\vee x_2$ be the (unique) element in $\mathcal{B}$ such that $\iota_\mathcal{B}(x_1\vee x_2)$ is the smallest closed superlevel component which contains $\iota_\mathcal{B}(x_1)\bigcup\iota_\mathcal{B}(x_2)$. Suppose $\iota_\mathcal{B}(x_1)\in\Comp(S^\rho_{\geqslant c_1})$, $\iota_\mathcal{B}(x_2)\in\Comp(S^\rho_{\geqslant c_2})$ and $\iota_\mathcal{B}(x_1\vee x_2)\in\Comp(S^\rho_{\geqslant c_3})$. Then we let $d_\mathcal{B}(x_1,x_2)=c_1+c_2-2c_3$.

 \item The root $r(\mathcal{B})$ of $\mathcal{B}$ corresponds to the unique closed superlevel component at $\min_{p \in \Gamma}\rho(p)$, which is the whole metric graph $\Gamma$.
 \end{enumerate}
\end{definition}

For $x\in\mathcal{B}$, if $\iota_\mathcal{B}(x)\in\Comp(S^\rho_{\geqslant c})$ where $c\in\imag\rho$, we let $d^\rho_\mathcal{B}(x)=c$ and $d^0_\mathcal{B}(x)=c-\min_{p \in \Gamma}\rho(p)$. Note that $d^\rho_\mathcal{B}(r(\mathcal{B}))=\min_{p \in \Gamma}\rho(p)$, $d^0_\mathcal{B}(r(\mathcal{B}))=0$ and $d^0_\mathcal{B}(x) = d_\mathcal{B}(r(\mathcal{B}),x)$. We now show that $\mathcal{B}$ is well-defined as a metric tree with  an associated partial order.

\begin{proposition}
$\mathcal{B}$ constructed in Definition~\ref{D:BifTree} is a rooted metric tree.
\end{proposition}
\begin{proof}
Here we give a proof following  a general construction of parametrized rooted trees and  rooted $\mR$-trees  as discussed in Appendix B5 of \cite{BR10}. We will show that $\mathcal{B}$ can be constructed as a tree by gluing subsets of $\mathcal{B}$ which are isometric to line segments and then $d_\mathcal{B}$ is a well defined metric on $\mathcal{B}$.

We first note that a partial order can be associated to $\mathcal{B}$, i.e., for two points $x$ and $x'$, we say $x\geqslant x'$ if $\iota_\mathcal{B}(x)\supseteq \iota_\mathcal{B}(x')$. Clearly $d^\rho_\mathcal{B}(x)\leqslant d^\rho_\mathcal{B}(x')$  if  $x\geqslant x'$. By Remark~\ref{R:SLS}, this partial order is well-defined and it is easy to verify that  $\mathcal{B}$ is a join-semilattice under the join operation $\vee$ where the $x\vee x'$  for any two elements $x$ and $x'$ in $\mathcal{B}$  corresponds to the smallest closed superlevel component which contains $\iota_\mathcal{B}(x)\bigcup\iota_\mathcal{B}(x')$ as in Definition~\ref{D:BifTree}. Note that $x\geqslant x'$ whenever $x\vee x' = x$ and $r(\mathcal{B})$ is the unique maximal element. 

For $x_1,x_2\in\mathcal{B}$, suppose that $x_1\geqslant x_2$ which means $x_1\vee x_2 = x_1$.   Let $d^\rho_\mathcal{B}(x_1)=c_1$ and $d^\rho_\mathcal{B}(x_2)=c_2$. By definition, we have $d_\mathcal{B}(x_1,x_2)=c_1+c_2-2c_1=c_2-c_1$.
We claim that $X=\{x\in\mathcal{B}\mid x_1\geqslant x \geqslant x_2\}$ is isometric to a closed segment of length $c_2-c_1$. First, note that $d^\rho_\mathcal{B}(x)\in[c_1,c_2]$ for any $x\in X$. 
Now for each $c\in [c_1,c_2]$, there exists a unique $x\in \mathcal{B}$ such that $d^\rho_\mathcal{B}(x) = c$ and $x\geqslant x_2$, i.e., $\iota_\mathcal{B}(x)$ is the unique closed superlevel component in $\Comp(S^\rho_{\geqslant c})$ which contains $\iota_\mathcal{B}(x_2)$ (Remark~\ref{R:SLS}). On the other hand we must have $x_1\geqslant x$ at the same time and thus $x\in X$. Therefore, by sending $c$ to $x$, we can define a bijection $\phi: [c_1,c_2]\to X$. It remains to show that $\phi$ is an isometry. Actually, by an analogous argument as above, we see that for any $y_1,y_2\in X$, we either have $y_1\geqslant  y_2$ or $y_2 \geqslant  y_1$ and $d_\mathcal{B}(y_1,y_2)=d_2-d_1$ if $y_1\geqslant  y_2$ and $d^\rho_\mathcal{B}(y_1)=d_1$ and $d^\rho_\mathcal{B}(y_2)=d_2$. It follows that $\phi$ is an isometry. Here, we write $[x_1,x_2]$ to represent $X$ as  a closed line segment, and let $(x_1,x_2]=[x_1,x_2]\setminus \{x_1\}$, $[x_1,x_2)=[x_1,x_2]\setminus \{x_2\}$ and $(x_1,x_2)=[x_1,x_2]\setminus \{x_1,x_2\}$. 

Since $\rho$ is a rational function with everywhere nonzero slopes, there are only finitely many points $x_1,\cdots,x_m$ in $\Gamma$ at which $\rho$ takes  local maximum values. Let $X_i :=\{x\in\mathcal{B}\mid  x \geqslant x_i\}$ for $i=1,\cdots,m$. Then $X_i = [r(\mathcal{B}),x_i]$. Note that we must have $\mathcal{B} = \bigcup_{i=1}^m X_i$. Let us reconstruct $\mathcal{B}$ by gluing $X_i$'s one by one. First let us glue $X_1$ and $X_2$. Note that $x_1\vee x_2\geqslant x_1, x_2$ and thus $x_1\vee x_2 \in X_1\bigcap X_2$. This means that $X_1\bigcap X_2 = [r(\mathcal{B}), x_1\vee x_2]$ and $X_1\bigcup X_2 = [r(\mathcal{B}), x_1\vee x_2]\coprod (x_1\vee x_2, x_1]\coprod  (x_1\vee x_2, x_2]$. Note that it follows that $X_1\bigcup X_2$ is a (topological) tree. Let us do this construction in general. Suppose we have derived $X'=X_1\bigcup \cdots \bigcup X_i$ as a tree already. Consider $Y=\{x_{i+1}\vee x_j\mid j=1,\cdots, i\}$. Clearly $Y$ is a subset of $X_{i+1}$ which means $Y$ is totally ordered. Let $y$ be the minimum element of $Y$. Then $X' \bigcap X_{i+1}=[r(\mathcal{B}),y]$ and $X_1\bigcup \cdots \bigcup X_{i+1}=X'\bigcup X_{i+1}=X'\coprod (y,x_{i+1}]$ which is also a tree.  Thus $\mathcal{B}$ is a tree.

It remains to show that $d_\mathcal{B}$ is a metric on the whole tree $\mathcal{B}$.   We  verify from the definition that for each $x_1,x_2\in \mathcal{B}$, $d_\mathcal{B}(x_1,x_2)= d_\mathcal{B}(x_1,x_1\vee x_2)+d_\mathcal{B}(x_2,x_1\vee x_2)$. For each $x_3\in\mathcal{B}$, consider $x_1\vee x_3$ and $x_2\vee x_3$. Without loss of generality, we may assume: (1) $x_1\vee x_3 > x_2\vee x_3$ or (2) $x_1\vee x_3 = x_2\vee x_3$. For case (1), we have $x_1\vee x_2\vee x_3 = x_1\vee x_2 = x_1\vee x_3$, and thus
\begin{align*}
d_\mathcal{B}(x_1,x_2)=&d_\mathcal{B}(x_1,x_1\vee x_2 \vee x_3)+d_\mathcal{B}(x_1\vee x_2 \vee x_3,x_2\vee x_3)+d_\mathcal{B}(x_2\vee x_3, x_2) \\
\leqslant & (d_\mathcal{B}(x_1,x_1\vee x_2 \vee x_3)+d_\mathcal{B}(x_1\vee x_2 \vee x_3,x_2\vee x_3)+d_\mathcal{B}(x_2\vee x_3, x_3))\\
&+(d_\mathcal{B}(x_2\vee x_3, x_2)+d_\mathcal{B}(x_2\vee x_3, x_3))\\
=&d_\mathcal{B}(x_1,x_3)+d_\mathcal{B}(x_2,x_3)
\end{align*}
where equality holds if and only if $x_3\geqslant x_2$.
For case (2), we have $x_1\vee x_2\vee x_3 = x_1\vee x_3 = x_2\vee x_3\geqslant x_1\vee x_2$, and thus
\begin{align*}
d_\mathcal{B}(x_1,x_2)=&d_\mathcal{B}(x_1,x_1\vee x_2)+d_\mathcal{B}(x_2,x_1\vee x_2)\\
\leqslant & (d_\mathcal{B}(x_1,x_1\vee x_2)+d_\mathcal{B}(x_1\vee x_2,x_1\vee x_2\vee x_3)+d_\mathcal{B}(x_3,x_1\vee x_2\vee x_3))\\
&+(d_\mathcal{B}(x_2,x_1\vee x_2)+d_\mathcal{B}(x_1\vee x_2,x_1\vee x_2\vee x_3)+d_\mathcal{B}(x_3,x_1\vee x_2\vee x_3))\\
=&d_\mathcal{B}(x_1,x_3)+d_\mathcal{B}(x_2,x_3)
\end{align*}
where equality holds if and only if $x_3=x_1\vee x_2$. Therefore, the triangle equality is satisfied and $d_\mathcal{B}$ is a metric.

\end{proof}

\begin{remark}
In the above proof, note that the leaves of $\mathcal{B}$ other than $r(\mathcal{B})$ are in one-to-one correspondence with those closed superlevel sets which are singletons, and in one-to-one correspondence with local maximum points of $\rho$ (which we may also call sink points of $\rho$). Denote the set of leaves of $\mathcal{B}$ by $\Leaf(\mathcal{B})$. We call a point $x$ of $\mathcal{B}$ with $|\Tan^+_\mathcal{B}(x)|\geqslant2$ a \emph{bifurcation point} of $\mathcal{B}$, and denote the set of bifurcation points by $\BifB$. Then $\emph(\mathcal{B})\bigcap\BifB=\emptyset$ and $\Leaf(\mathcal{B})\bigcup\BifB$ is the set of points of valence other than $2$ in $\mathcal{B}$ together with $r(\mathcal{B})$. Note that we have either $r(\mathcal{B})\in\Leaf(\mathcal{B})$ or $r(\mathcal{B})\in\BifB$.
We call the image of $d^\rho_\mathcal{B}$ restricted to the minimal vertex set of $\mathcal{B}$ the set of \emph{bifurcation values}, denoted by $\Bif$. Then $\Bif$ is finite and we have $\Bif\subseteq\mathcal{E}_\rho$, where $\mathcal{E}_\rho$ is the set of exceptional points of $\rho$.
\end{remark}

For a point $p$ in $\Gamma$, recall that $\Tan^{\rho+}_\Gamma(p)$ is the set of tangent directions in $\Tan_\Gamma(p)$ emanating from $p$ where $\rho$ locally increases. Similarly, we let $\Tan^{\rho-}_\Gamma(p)$ be the set of tangent directions where $\rho$ locally decreases. Then $\Tan_\Gamma(p)=\Tan^{\rho+}_\Gamma(p)\coprod\Tan^{\rho-}_\Gamma(p)$.

Let $\mathcal{T}$ be a metric tree rooted at $r(\mathcal{T})$. For a point $x$ in $\mathcal{T}$, we say a tangent direction $t\in\Tan_\mathcal{T}(x)$ is a forward (respectively backward) tangent direction at $x$ if the distance function from the root increases (respectively decreases) along $t$. Denote by $\Tan^+_\mathcal{T}(x)$ (respectively $\Tan^-_\mathcal{T}(x)$) the set of forward (respectively backward) tangent directions at $x$. Note that $\Tan^-_\mathcal{T}(x)$ is empty if $x$ is the root of $\mathcal{T}$ and a singleton otherwise.

We state without proofs of the following lemmas (Lemma ~\ref{L:pi}  and Lemma ~\ref{L:VecIota}) which follow naturally from  the construction of the bifurcation tree $\mathcal{B}$ with respect to $\rho$.  Lemma~\ref{L:pi} states that $\rho$ factors through $d^\rho_\mathcal{B}$ by the canonical projection $\pi_\mathcal{B}:\Gamma\rightarrow \mathcal{B}$. Lemma~\ref{L:VecIota} states that the set of forward tangent directions on $\mathcal{B}$ can be identified with the set of all open superlevel components of $\rho$.
\begin{lemma} \label{L:pi}
For $p\in\Gamma$, there is a unique closed superlevel component $\alpha$ at $\rho(p)$ which contains $p$. By sending $p$ to $\iota^{-1}_\mathcal{B}(\alpha)$, it induces a canonical projection $\pi_\mathcal{B}:\Gamma\rightarrow \mathcal{B}$. Moreover, the map $\pi_\mathcal{B}$ is continuous, piecewise-linear, surjective and satisfies $\rho = d^\rho_\mathcal{B}\circ\pi_\mathcal{B}$.
\end{lemma}

\begin{lemma} \label{L:VecIota}
There is a canonical bijection $\vec{\iota}_\mathcal{B}:\coprod_{x\in\mathcal{B}} \Tan^+_\mathcal{B}(x)\rightarrow \coprod_{c\in\imag \rho}\Comp(S^\rho_{> c})$. In particular,  $\Tan^+_\mathcal{B}(x)$ is in bijection with $\{\beta\in\Comp(S^\rho_{> d^\rho_\mathcal{B}(x)})|\beta\subseteq \iota_\mathcal{B}(x)\}$.
\end{lemma}

\begin{remark}
The projection  $\pi_\mathcal{B}$ naturally induces a pushforward map $\pi_{\mathcal{B}*}: \coprod_{p\in\Gamma}\Tan_\Gamma(p)\rightarrow\coprod_{x\in\mathcal{B}}\Tan_\mathcal{B}(x)$. In particular, (1) if $t\in \Tan^{\rho-}_\Gamma(p)$, then $\pi_{\mathcal{B}*}(t)$ is the unique element in $\Tan^-_\mathcal{B}(\pi(p))$; (2) if $t\in \Tan^{\rho+}_\Gamma(p)$, then $\pi_{\mathcal{B}*}(t)\in\Tan^+_\mathcal{B}(\pi_{\mathcal{B}}(p))$ and more precisely $\vec{\iota}_\mathcal{B}(\pi_{\mathcal{B}*}(t))$ is the unique open superlevel component of $\rho$ with $p$ on its boundary and $t$ pointing inwards. Note that $\pi_{\mathcal{B}*}$ is surjective.
\end{remark}

\begin{example} \label{E:Bif}
In Figure~\ref{F:Bif}, suppose a vertex set of $\Gamma$ (upper panel) is $\{o_1,o_2,p_1,p_2,p_3,q_1,q_2,q_3\}$ and all edges have length $1$. Then a global diagram on $\Gamma$ with all multiplicities being $1$ along directions marked by the arrows is solvable and we suppose a solution is $\rho$ with the corresponding bifurcation tree $\mathcal{B}$ and the canonical projection $\pi_\mathcal{B}$. In particular, as shown by the vertical dashed lines, the point $x$ is the root of the bifurcation tree corresponding to the unique closed superlevel set at $\min_{p\in\Gamma}(\rho)$ (the whole metric graph $\Gamma$) which is also the image of $o_1$ and $o_2$ under $\pi_\mathcal{B}$; $y_1=\pi_\mathcal{B}(p_1)$ corresponds to the closed superlevel component $\{p_1\}$; $y_2=\pi_\mathcal{B}(p_2)=\pi_\mathcal{B}(p_3)$ corresponds to the closed superlevel component which is the union of all closed edges connecting $p_2$, $p_3$, $q_1$, $q_2$ and $q_3$; for $i=1,2,3$, $z_i=\pi_\mathcal{B}(q_i)$ corresponds to closed superlevel component $\{q_i\}$. 
Note that $z_1\vee z_2 = z_1\vee z_3 = z_2\vee z_3=y_2$ and $y_1\vee z_1 = y_1\vee z_2 = y_1\vee z_3=x$. Then for $i,j=1,2,3$ and $i\neq j$, we have $d_\mathcal{B}(z_i,z_j) = d_\mathcal{B}(z_i,y_2)+d_\mathcal{B}(z_j,y_2) =1+1=2$ and $d_\mathcal{B}(z_i,y_1) = d_\mathcal{B}(z_i,x)+d_\mathcal{B}(y_1,x) = 2+1=3$. 
In addition, the tangent direction from $y_2$ to $z_1$ correspponds to the open superlevel component $(p_2,q_1]\bigcup(p_3,q_1]$, the tangent direction from $y_2$ to $z_2$ correspponds to the open superlevel component $(p_2,q_2]\bigcup(p_3,q_2]$, and the tangent direction from $y_2$ to $z_3$ correspponds to the open superlevel component $(p_2,q_3]\bigcup(p_3,q_3]$. Moreover, $\Leaf(\mathcal{B})=\{y_1,z_1,z_2,z_3\}$ and $\BifB=\{x,y_2\}$.

\begin{figure}[tbp]
\centering
\begin{tikzpicture}[>=to,x=1.5cm,y=0.6cm]
\coordinate (x) at (0,3);
    \fill [black] (x) circle (2.5pt);
    \draw (x) node[anchor=east] {\Large $x$};
\coordinate (y1) at (1,4);
    \fill [black] (y1) circle (2.5pt);
    \draw (y1) node[anchor=south east] {\Large $y_1$};
\coordinate (y2) at (1,2);
    \fill [black] (y2) circle (2.5pt);
    \draw (y2) node[anchor=north east] {\Large $y_2$};
\coordinate (z1) at (3,4);
    \fill [black] (z1) circle (2.5pt);
    \draw (z1) node[anchor=west] {\Large $z_1$};
\coordinate (z2) at (3,2);
    \fill [black] (z2) circle (2.5pt);
    \draw (z2) node[anchor=west] {\Large $z_2$};
\coordinate (z3) at (3,0);
    \fill [black] (z3) circle (2.5pt);
    \draw (z3) node[anchor=west] {\Large $z_3$};

\draw [line width=1.5pt] (x)-- (y1);
\draw [line width=1.5pt] (x)-- (y2);
\draw [line width=1.5pt] (y2)-- (z1);
\draw [line width=1.5pt] (y2)-- (z2);
\draw [line width=1.5pt] (y2)-- (z3);

\coordinate (o1) at (0,9);
    \fill [black] (o1) circle (2.5pt);
    \draw (o1) node[anchor=east] {\Large $o_1$};
\coordinate (o2) at (0,8);
    \fill [black] (o2) circle (2.5pt);
    \draw (o2) node[anchor=east] {\Large $o_2$};
\coordinate (p1) at (1,9.5);
    \fill [black] (p1) circle (2.5pt);
    \draw (p1) node[anchor=south east] {\Large $p_1$};
\coordinate (p2) at (1,8);
    \fill [black] (p2) circle (2.5pt);
    \draw (p2) node[anchor=south east] {\Large $p_2$};
\coordinate (p3) at (1,7);
    \fill [black] (p3) circle (2.5pt);
    \draw (p3) node[anchor=north east] {\Large $p_3$};
\coordinate (q1) at (3,9.5);
    \fill [black] (q1) circle (2.5pt);
    \draw (q1) node[anchor=west] {\Large $q_1$};
\coordinate (q2) at (3,7.5);
    \fill [black] (q2) circle (2.5pt);
    \draw (q2) node[anchor=west] {\Large $q_2$};
\coordinate (q3) at (3,5);
    \fill [black] (q3) circle (2.5pt);
    \draw (q3) node[anchor=west] {\Large $q_3$};

\begin{scope}[line width=1.6pt, every node/.style={sloped,allow upside down}]
  \draw (o1) -- node {\midarrow} (p1);
  \draw (o2) -- node {\midarrow} (p1);
  \draw (o2) -- node {\midarrow} (p2);
  \draw (o2) -- node {\midarrow} (p3);
  \draw (p2) -- node {\midarrow} (q1);
  \draw (p2) -- node {\midarrow} (q2);
  \draw (p2) -- node {\midarrow} (q3);
  \draw (p3) -- node {\midarrow} (q1);
  \draw (p3) -- node {\midarrow} (q2);
  \draw (p3) -- node {\midarrow} (q3);
\end{scope}

\draw [dashed, line width=1.6pt] (x) -- (o1);
\draw [dashed, line width=1.6pt] (y2) -- (p1);
\draw [dashed, line width=1.6pt] (z3) -- (q1);

\draw (-0.8,3) node {\LARGE $\mathcal{B}$};
\draw (-0.8,7.5) node {\LARGE $\Gamma$};
\draw [-stealth, line width=1pt] (-0.8,6.5) -- (-0.8,4);
\draw (-0.8,5.25) node[anchor=west] {\LARGE $\pi_\mathcal{B}$};
\end{tikzpicture}
\caption{An illustration of a bifurcation tree $\mathcal{B}$ of $\Gamma$ and the canonical projection $\pi_\mathcal{B}$.}\label{F:Bif}
\end{figure}
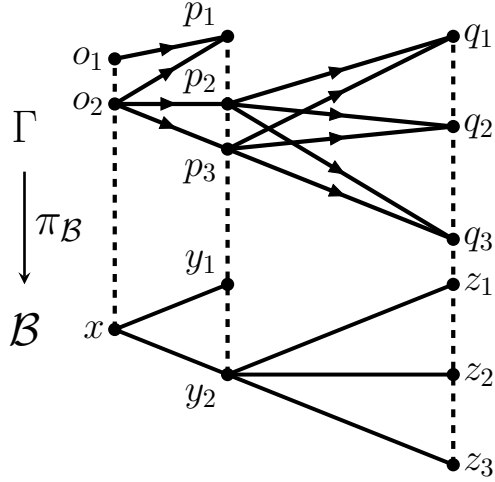
\end{example}

\subsection{Partition trees}\label{subS:ParTree}

We now generalize the notion of bifurcation trees to objects called partition trees. Partition trees are the elements in  the spaces $\LDHone$, $\LDHtwo$, $\LDHthree$ and $\LDHfour$ treated in Section \ref{S:ParTreeComp}.

For each rooted metric tree $\mathcal{T}$ with root $r(\mathcal{T})$, we can define a distance function $d^0_\mathcal{T}: \mathcal{T}\to \mR$ that takes a point $x\in\mathcal{T}$ to the distance between $r(\mathcal{T})$ and $x$. Note that this  is consistent to the definition of  $d^0_\mathcal{B}$  where $\mathcal{B}$ is some bifurcation tree $\mathcal{B}$ with respect to $\rho$ introduced in the previous subsection.

\begin{definition}
Let $\rho: \Gamma \rightarrow \mathbb{R}$ be a rational function on $\Gamma$ with everywhere nonzero slopes.
we call $(\mathcal{T},\pi_\mathcal{T})$ or simply $\mathcal{T}$ a \emph{partition tree} with respect to $\rho$ if
 $\mathcal{T}$  is a rooted metric tree and  $\pi_{\mathcal{T}}:\Gamma\to\mathcal{T}$ is a continuous finite surjection (finite means all fibers are finite) such that $\hat{\rho} = d^0_\mathcal{T}\circ\pi_\mathcal{T}$ where $\hat{\rho}=\rho-\min_{p \in \Gamma}\rho(p)$. We let  $\Lambda_\rho$ be the set of all partition trees with respect to $\rho$.
\end{definition}

\begin{example}
Let $\mathcal{B}$ be the bifurcation tree with respect to $\rho$ and $\pi_\mathcal{B}$ be the canonical projection from $\Gamma$ onto $\mathcal{B}$. Then $(\mathcal{B},\pi_{\mathcal{B}})\in\Lambda_\rho$, following from  Lemma~\ref{L:pi} directly.
\end{example}

\begin{example}
The segment $\imag\rho$ can be considered as a metric graph with root $\min_{p \in \Gamma}\rho(p)$. Clearly $(\imag\rho,\rho)\in\Lambda_\rho$.
\end{example}

The following proposition says that  all partition trees can be constructed by gluing the bifurcation tree properly.

\begin{proposition} \label{P:CanoSurj}
Let $\rho: \Gamma \rightarrow \mathbb{R}$ be a rational function on a metric graph $\Gamma$ with everywhere nonzero slopes and $\mathcal{B}$ be the bifurcation tree with respect to $\rho$. Let $\hat{\rho}=\rho-\min_{p \in \Gamma}\rho(p)$. Let $\mathcal{T}$  be a rooted metric tree and  $\pi_{\mathcal{T}}:\Gamma\to\mathcal{T}$ be a continuous finite surjection. Then
$(\mathcal{T},\pi_\mathcal{T})\in\Lambda_\rho $ if and only if there exists a continuous surjection $\Theta^{\mathcal{B}}_{\mathcal{T}}:
\mathcal{B}\to\mathcal{T} $ such that the following diagram commute:

\[
\begin{tikzpicture}[scale=1.5,
back line/.style={solid},
cross line/.style={preaction={draw=white, -,line width=4pt}},
text height=1.6ex, text depth=0.5ex]
\node (Gamma) at (-2.5,0) {$\Gamma$};
\node (B) at (-1,0) {$\mathcal{B}$};
\node (T) at (1,0) {$\mathcal{T}$};
\node (ImRho) at (2.5,0) {$\imag\hat{\rho}$};
\path[->,font=\scriptsize,>=angle 90]
(Gamma) edge node[pos=0.7,above]{$\pi_{\mathcal{B}}$} (B)
(Gamma) edge[bend right, back line] node[pos=0.4,above]{$\pi_{\mathcal{T}}$} (T)
(Gamma) edge[bend left=40] node[auto]{$\hat{\rho}$} (ImRho)

(B) edge node[pos=0.5,above]{$\Theta^{\mathcal{B}}_{\mathcal{T}}$} (T)

(B) edge[bend left, cross line] node[pos=0.4,above]{$d^0_\mathcal{B}$} (ImRho)
(T) edge node[auto]{$d^0_{\mathcal{T}}$} (ImRho);

\end{tikzpicture}
\]
\end{proposition}

\begin{proof}
If the diagram commutes, then $(\mathcal{T},\pi_\mathcal{T})\in\Lambda_\rho $ by the definition of partition trees.

Now suppose $(\mathcal{T},\pi_\mathcal{T})\in\Lambda_\rho $. Then $\hat{\rho} = d^0_\mathcal{T}\circ\pi_\mathcal{T}$ and we just need to find the canonical projection $\Theta^{\mathcal{B}}_{\mathcal{T}}$ from $\mathcal{B}$ to $\mathcal{T}$ such that $\pi_{\mathcal{T}}=\Theta^{\mathcal{B}}_{\mathcal{T}}\circ\pi_{\mathcal{B}}$ and $d^0_{\mathcal{B}}=d^0_{\mathcal{T}}\circ \Theta^{\mathcal{B}}_{\mathcal{T}}$.

For a point $x\in\mathcal{B}$, let $c = d^0_{\mathcal{B}} (x)$. Note that $\iota_\mathcal{B}(x)\in\Comp(S^{\hat{\rho}}_{\geqslant c})$ and by definition of $\pi_\mathcal{B}$, we have $\pi_\mathcal{B}^{-1}(x) = \partial \iota_\mathcal{B}(x)$ where $\partial \iota_\mathcal{B}(x)$ is the set of boundary points of $\iota_\mathcal{B}(x)$. Claim that for any two points $p,q\in \partial \iota_\mathcal{B}(x)$, we have $\pi_\mathcal{T}(p)=\pi_\mathcal{T}(q)$. Note that $d^0_{\mathcal{T}}(\pi_\mathcal{T}(p)) = d^0_{\mathcal{T}}(\pi_\mathcal{T}(q)) = \hat{\rho}(p) = \hat{\rho}(q) = c$ since $\hat{\rho} = d^0_\mathcal{B}\circ\pi_\mathcal{B}=d^0_\mathcal{T}\circ\pi_\mathcal{T}$. 

Let $Y = \{y\in \mathcal{T}\mid d^0_\mathcal{T}(y)\geqslant c\}$ and $\partial Y  = \{y\in \mathcal{T}\mid d^0_\mathcal{T}(y)= c\}$. Note that each connected component of $Y$ has exactly one boundary point in $\partial Y$ since $\mathcal{T}$ is a metric tree.   Clearly $\pi_\mathcal{T}(p),\pi_\mathcal{T}(q)\in \partial Y$ and $\pi_\mathcal{T}(\iota_\mathcal{B}(x))\subseteq Y$. Since $\iota_\mathcal{B}(x)$ is connected and $\pi_\mathcal{T}$ is continuous, $\pi_\mathcal{T}(\iota_\mathcal{B}(x))$ must be connected which is only possible when $\pi_\mathcal{T}(\iota_\mathcal{B}(x))$ is contained in one connected component of $Y$. This implies $\pi_\mathcal{T}(p)=\pi_\mathcal{T}(q)$. 

Define $\Theta^{\mathcal{B}}_{\mathcal{T}}(x)$ to be this point $\pi_\mathcal{T}(p)=\pi_\mathcal{T}(q)\in\mathcal{T}$. The above argument shows that $\Theta^{\mathcal{B}}_{\mathcal{T}}$ is well-defined. The continuity of $\Theta^{\mathcal{B}}_{\mathcal{T}}$ and commutativity of the diagram also naturally follow.
\end{proof}

\begin{remark} \label{R:partition}
Consider the bifurcation tree $\mathcal{B}$ and a partition tree $\mathcal{T}$ with respect to $\rho$. The canonical projection $\Theta^{\mathcal{B}}_{\mathcal{T}}$ induces a partition $P_c$ of $(d^\rho_\mathcal{B})^{-1}(c)$ for any $c\in\imag\rho$ as follows: $x_1\sim x_2$ in $P_c$ if and only if $\Theta^{\mathcal{B}}_{\mathcal{T}}(x_1)=\Theta^{\mathcal{B}}_{\mathcal{T}}(x_2)$. (See Appendix~\ref{S:TreeSpace} for further discussions.)
\end{remark}

\begin{remark} 
Like the pushforward $\pi_{\mathcal{B}*}$ induced by the canonical projection $\pi_\mathcal{B}$, we also have the pushforward map $\pi_{\mathcal{T}*}: \coprod_{p\in\Gamma}\Tan_\Gamma(p)\rightarrow\coprod_{x\in\mathcal{T}}\Tan_\mathcal{T}(x)$ such that (1) if $t\in \Tan^{\rho-}_\Gamma(p)$, then $\pi_{\mathcal{T}*}(t)$ is the unique element in $\Tan^-_\mathcal{T}(\pi_\mathcal{T}(p))$; (2) if $t\in \Tan^{\rho+}_\Gamma(p)$, then $\pi_{\mathcal{T}*}(t)\in\Tan^+_\mathcal{T}(\pi_{\mathcal{T}}(p))$.
\end{remark}

\begin{example} \label{E:Partition}
In Figure~\ref{F:Partition}, we show an example of partition tree $\mathcal{T}$ based on the bifurcation tree $\mathcal{B}$ constructed in Example~\ref{E:Bif}. In particular, $\mathcal{T}$ is constructed by gluing edges $y_2z_1$ and $y_2z_2$ of $\mathcal{B}$ (the grey edges) isometrically into edge $y'_2y'_{12}$ of $\mathcal{T}$. As a result, $\pi_\mathcal{T}$ maps all the edges of $\Gamma$ connecting $p_2$ and $p_3$ with $q_1$ and $q_2$ to the edge $y'_2y'_{12}$.

\begin{figure}[tbp]
\centering
\begin{tikzpicture}[>=to,x=1.5cm,y=0.6cm]
\coordinate (x) at (0,3);
    \draw (x) node[anchor=east] {\Large $x$};
\coordinate (y1) at (1,4);
    \draw (y1) node[anchor=south east] {\Large $y_1$};
\coordinate (y2) at (1,2);
    \draw (y2) node[anchor=north east]   {\Large $y_2$};
\coordinate (z1) at (3,4);
    \draw (z1) node[anchor=west] {\Large $z_1$};
\coordinate (z2) at (3,2);
    \draw (z2) node[anchor=west] {\Large $z_2$};
\coordinate (z3) at (3,0);
    \draw (z3) node[anchor=west] {\Large $z_3$};

\draw [line width=1.5pt] (x)-- (y1);
\draw [line width=1.5pt] (x)-- (y2);
\draw [line width=1.5pt, color=gray] (y2)-- (z1);
\draw [line width=1.5pt, color=gray] (y2)-- (z2);
\draw [line width=1.5pt] (y2)-- (z3);
\fill [black] (x) circle (2.5pt);
\fill [black] (y1) circle (2.5pt);
\fill [black] (y2) circle (2.5pt);
\fill [black] (z1) circle (2.5pt);
\fill [black] (z2) circle (2.5pt);
\fill [black] (z3) circle (2.5pt);

\coordinate (xx) at (5,3);
      \draw (xx) node[anchor=east] {\Large $x'$};
\coordinate (yy1) at (6,4);
    \draw (yy1) node[anchor=south east] {\Large $y'_1$};
\coordinate (yy2) at (6,2);
    \draw (yy2) node[anchor=north east] {\Large $y'_2$};
\coordinate (zz12) at (8,2);
    \draw (zz12) node[anchor=west] {\Large $z'_{12}$};
\coordinate (zz3) at (8,0);
    \draw (zz3) node[anchor=west] {\Large $z'_3$};

\draw [line width=1.5pt] (xx)-- (yy1);
\draw [line width=1.5pt] (xx)-- (yy2);
\draw [line width=1.5pt, color=gray] (yy2)-- (zz12);
\draw [line width=1.5pt] (yy2)-- (zz3);
\fill [black] (xx) circle (2.5pt);
\fill [black] (yy1) circle (2.5pt);
\fill [black] (yy2) circle (2.5pt);
\fill [black] (zz12) circle (2.5pt);
\fill [black] (zz3) circle (2.5pt);

\coordinate (o1) at (2,9);
    \fill [black] (o1) circle (2.5pt);
    \draw (o1) node[anchor=east] {\Large $o_1$};
\coordinate (o2) at (2,8);
    \fill [black] (o2) circle (2.5pt);
    \draw (o2) node[anchor=east] {\Large $o_2$};
\coordinate (p1) at (3,9.5);
    \fill [black] (p1) circle (2.5pt);
    \draw (p1) node[anchor=south east] {\Large $p_1$};
\coordinate (p2) at (3,8);
    \fill [black] (p2) circle (2.5pt);
    \draw (p2) node[anchor=south east] {\Large $p_2$};
\coordinate (p3) at (3,7);
    \fill [black] (p3) circle (2.5pt);
    \draw (p3) node[anchor=north east] {\Large $p_3$};
\coordinate (q1) at (5,9.5);
    \fill [black] (q1) circle (2.5pt);
    \draw (q1) node[anchor=west] {\Large $q_1$};
\coordinate (q2) at (5,7.5);
    \fill [black] (q2) circle (2.5pt);
    \draw (q2) node[anchor=west] {\Large $q_2$};
\coordinate (q3) at (5,5);
    \fill [black] (q3) circle (2.5pt);
    \draw (q3) node[anchor=west] {\Large $q_3$};

\draw [line width=1.5pt] (o1) -- (p1);
\draw [line width=1.5pt] (o2) -- (p1);
\draw [line width=1.5pt] (o2) -- (p2);
\draw [line width=1.5pt] (o2) -- (p3);
\draw [line width=1.5pt, color=gray] (p2) -- (q1);
\draw [line width=1.5pt, color=gray] (p2) -- (q2);
\draw [line width=1.5pt] (p2) -- (q3);
\draw [line width=1.5pt, color=gray] (p3) -- (q1);
\draw [line width=1.5pt, color=gray] (p3) -- (q2);
\draw [line width=1.5pt] (p3) -- (q3);
\fill [black] (p2) circle (2.5pt);
\fill [black] (p3) circle (2.5pt);
\fill [black] (q1) circle (2.5pt);
\fill [black] (q2) circle (2.5pt);

\draw (5.5,0.5) node {\LARGE $\mathcal{T}$};
\draw (0.5,0.5) node {\LARGE $\mathcal{B}$};
\draw (1.2,8) node {\LARGE $\Gamma$};
\draw [-stealth, dashed, line width=1.8pt] (2.5,6.5) -- (1.5,4.5);
\draw [-stealth, dashed, line width=1.8pt] (5.5,6.5) -- (6.5,4.5);
\draw [-stealth, dashed, line width=1.8pt] (3.5,1.5) -- (5,1.5);
\draw (2,5.5) node[anchor=south east] {\Large $\pi_\mathcal{B}$};
\draw (6,5.5) node[anchor=south west] {\Large $\pi_\mathcal{T}$};
\draw (4.25,1.5) node[anchor=north] {\Large $\Theta^\mathcal{B}_\mathcal{T}$};
\end{tikzpicture}
\caption{An illustration of a projection $\pi_\mathcal{T}$ from $\Gamma$ to a partition tree $\mathcal{T}$.}\label{F:Partition}
\end{figure}
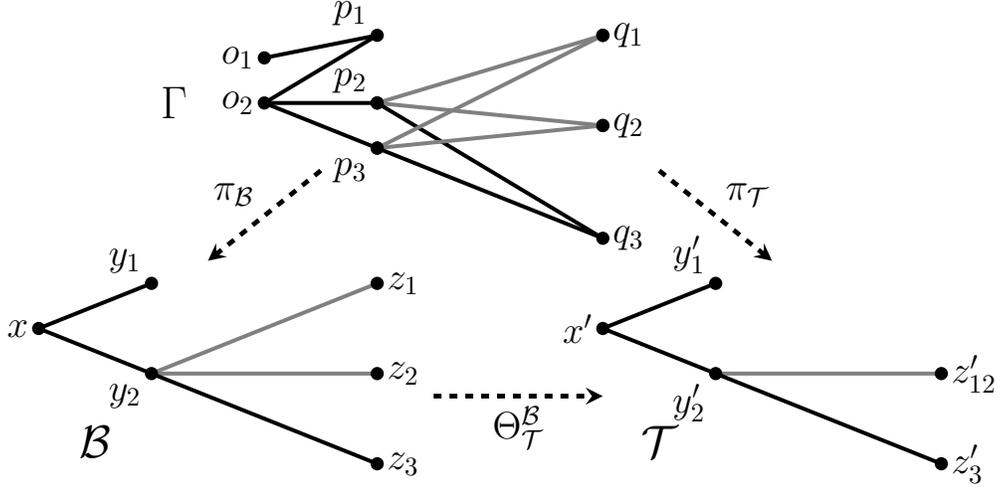
\end{example}

\section{Obstructions of Smoothability and Spaces  $\LDHone$, $\LDHtwo$, $\LDHthree$ and $\LDHfour$ of Partition Trees}\label{S:ParTreeComp}

\subsection{An Example of a Non-Solvable Limit $g^1_d$} \label{S:Nonsolvable}
In this subsection, we will show that the additional restriction on limit $g^1_d$ over pre-limit $g^1_d$ does not guarantee solvability by presenting an example of non-solvable limit $g^1_d$.

\begin{figure}[tbp]
 \centering
\begin{tikzpicture}[>=to,x=2cm,y=2cm]

\coordinate (v1) at (0,0);
    \fill [black] (v1) circle (2.5pt);
    \draw (v1) node[anchor=east] {\Large $v_1$};
\coordinate (v2) at (3,1);
    \fill [black] (v2) circle (2.5pt);
    \draw (v2) node[anchor=south] {\Large $v_2$};
\coordinate (v3) at (3.96824583655,3/4);
    \fill [black] (v3) circle (2.5pt);
    \draw (v3) node[anchor=west] {\Large $v_3$};
\coordinate (v4) at (3,-1);
    \fill [black] (v4) circle (2.5pt);
    \draw (v4) node[anchor=north] {\Large $v_4$};
\coordinate (v5) at (1,-1/3);
    \fill [black] (v5) circle (2.5pt);
    \draw (v5) node[anchor=north east] {\Large $v_5$};
\coordinate (v6) at (2,-2/3);
    \fill [black] (v6) circle (2.5pt);
    \draw (v6) node[anchor=north east] {\Large $v_6$};
\coordinate (v7) at (1,1/3);
    \fill [black] (v7) circle (2.5pt);
    \draw (v7) node[anchor=south east] {\Large $v_7$};
\coordinate (v8) at (2,2/3);
    \fill [black] (v8) circle (2.5pt);
    \draw (v8) node[anchor=south east] {\Large $v_8$};

\begin{scope}[line width=1.6pt, every node/.style={sloped,allow upside down}]
  \draw (v1) -- node {\midarrow} (v7);
  \draw (v7) -- node {\midarrow} (v8);
  \draw (v8) -- node {\midarrow} (v2);
  \draw (v3) -- node {\midarrow} (v2);
  \draw (v3) -- node {\midarrow} (v4);
  \draw (v6) -- node {\midarrow} (v4);
  \draw (v5) -- node {\midarrow} (v6);
  \draw (v1) -- node {\midarrow} (v5);

\end{scope}
\draw [rotate around={0:(v1)},line width=1.2pt,dash pattern=on 2pt off 2pt] (0.2,0) ellipse (0.05 and 0.2);
\draw [rotate around={3:(v2)},line width=1.2pt,dash pattern=on 2pt off 2pt] (3,0.92) ellipse (0.5 and 0.03);
\draw [rotate around={20:(v4)},line width=1.2pt,dash pattern=on 2pt off 2pt] (3,-0.8) ellipse (0.45 and 0.06);
\draw [rotate around={60:(v3)},line width=1.2pt,dash pattern=on 2pt off 2pt] (3.75,3/4) ellipse (0.08 and 0.08);
\draw [rotate around={-13:(v3)},line width=1.2pt,dash pattern=on 2pt off 2pt] (3.75,3/4) ellipse (0.08 and 0.08);
\end{tikzpicture}
 \caption{An example of a diagrammatic limit $g^1_d$ such that the characteristic equation associated to the global diagram does not have a solution.}
 \label{F:level_one}
\end{figure}
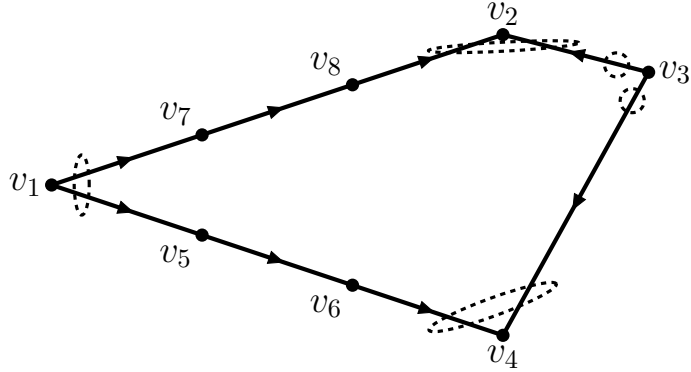

Consider the global diagram on a cycle shown in Figure \ref{F:level_one} with multiplicity $1$ on each edge. The metric graph $\Gamma$ has the lengths: $\ell_{v_2,v_3}=\ell_{v_2,v_8}$, $\ell_{v_4,v_3}=\ell_{v_4,v_5}$,  $\ell_{v_1,v_5}=\ell_{v_1,v_7}$ and $\ell_{v_1,v_6}=\ell_{v_1,v_8}$.  The algebraic curve $C_p$ at every point $p \in \Gamma$ is a projective line over $\kappa$. Let the two outgoing tangent directions $t^o_{1,1}$ and $t^o_{1,2}$ at $v_1$ be locally equivalent, the two outgoing tangent directions $t^o_{3,1}$ and $t^o_{3,2}$ at $v_3$ be in different local equivalence classes, the two incoming tangent directions $t^i_{2,1}$ and $t^i_{2,2}$ at $v_2$ be locally equivalent, and the two incoming tangent directions $t^i_{4,1}$ and $t^i_{4,2}$ at $v_4$ be locally equivalent. For all $p\in\Gamma\setminus\{v_1,v_2,v_3,v_4\}$, there is only one incoming tangent direction $t^i_p$ and one outgoing tangent direction $t^o_p$ at $p$, while $t^i_p$ and $t^o_p$ are in different local equivalence classes. We verify that this global diagram is not solvable.

On the other hand, from this global diagram, we can construct a diagrammatic pre-limit $g^1_d$ represented by $(\mathcal{D},\mathcal{H})$ in the following way. Let $D_{\Gamma}=2(v_1)+(v_3)$. Let $D_{v_1}=(x_{1,1})+(x_{1,2})$ where $x_{1,1}$ and $x_{1,2}$ are two distinct non-marked points in $C_{v_1}$, $D_{v_3}=(x_3)$ where $x_3$ is a non-marked point on $C_{v_3}$, and $D_p=0$ for all $p\in\Gamma\setminus\{v_1,v_3\}$. Using the approach shown in Remark~\ref{R:LocalHp}, we may construct $H_p$'s conversely using the local diagrams at $p$ induced from the global diagram. More precisely, let $f_{v_1}$ be a rational function on $C_{v_1}$ whose associated divisor is $(\red_{v_1}(t^o_{1,1}))+(\red_{v_1}(t^o_{1,2}))-(x_{1,1})-(x_{1,2})$, $f^{(1)}_{v_3}$ be a rational function on $C_{v_3}$ whose associated divisor is $(\red_{v_3}(t^o_{3,1}))-(x_3)$, $f^{(2)}_{v_3}$ be a rational function on $C_{v_3}$ whose associated divisor is $(\red_{v_3}(t^o_{3,2}))-(x_3)$, $f_{v_2}$ be a rational
 function on $C_{v_2}$ whose associated divisor is $(x_{2,1})+(x_{2,2})-(\red_{v_2}(t^i_{2,1}))-(\red_{v_2}(t^i_{2,2}))$ where $x_{2,1}$ and $x_{2,2}$ are two non-marked points in $C_{v_2}$, $f_{v_4}$ be a rational function on $C_{v_4}$ whose associated divisor is $(x_{4,1})+(x_{4,2})-(\red_{v_4}(t^i_{4,1}))-(\red_{v_4}(t^i_{4,2}))$ where $x_{4,1}$ and $x_{4,2}$ are two non-marked points in $C_{v_4}$. For all $p\in\Gamma\setminus\{v_1,v_2,v_3,v_4\}$, let $f_p$ be a rational function on $C_p$ whose associated divisor is
$(\red_p(t^o_p))-(\red_p(t^i_p))$. Then we let $H_p$ be a linear space of rational functions on $C_p$ with a basis $\{1,f_p\}$ for all $p\in\Gamma$ (for $p=v_3$,  we choose $f_{v_3}$ to be either $f^{(1)}_{v_3}$ or $f^{(2)}_{v_3}$, noting that the rational functions $1$, $f^{(1)}_{v_3}$ and $f^{(2)}_{v_3}$ are linear dependent).

We claim that $(\mathcal{D},\mathcal{H})$ constructed this way represents a limit $g^1_d$. To this end, we must show that  for every effective divisor $\mathcal{E}=(u,z_u)$ of degree one on the saturated metrized complex where $u\in\Gamma$ and ${z_u} \in C_u$, there exists a rational function ${\mathfrak g}=(g_{\Gamma},\{g_p\}_{p \in \Gamma})$ such that $g_p \in H_p$ and $\mathcal{D}+\divisor(\mathfrak{g})-\mathcal{E} \geq 0$.

We first specify $g_{\Gamma}$. We  describe $g_\Gamma$ in terms of a series of chip-firing moves:  (1) if $u$ lies in $[v_1,v_8]$ or $[v_1,v_6]$, we can fire both chips from $v_1$ till one of the chips hits $u$; (2) if  $u$ lies in $[v_2,v_8]$ or $[v_2,v_3]$, we first fire both chips from $v_1$ till one of the chips hits $v_8$ and then fire the chip at $v_3$ and $v_8$ simultaneously till one of the chips hits $u$; (3) if $u$ lies in $[v_4,v_5]$ or $[v_4,v_3]$, we fire first both chips from $v_1$ till a chip hits $v_5$ and then fire the chips at $v_3$ and  $v_5$ simultaneously till one of the chips hits $u$.

Now let us specify $g_p$'s. (1) First, let $g_u=f_u-f_u(z_u)$. Then $g_u$ has the same poles as $f_u$ and has a zero at $z_u$. (2) For all the points $p\in\Gamma$ such that $g_\Gamma$ is locally constant, we let $g_p$ be a constant. (3) If the slope of $g_\Gamma$ along $t^o_{3,1}$ is nonzero, let $g_{v_3}=f^{(1)}_{v_3}$. If the slope of $g_\Gamma$ along $t^o_{3,2}$ is nonzero, let $g_{v_3}=f^{(2)}_{v_3}$. (4) For all remaining points $p\in\Gamma$, we let $g_p=f_p$.

For this choice of ${\mathfrak{g}}$, we have  $\mathcal{D}-\mathcal{E}+\divisor({\mathfrak g}) \geq 0$.

Therefore, $(\mathcal{D},\mathcal{H})$ represents a non-solvable limit $g^1_d$ and we have the following proposition.

\begin{proposition}
There exists a non-smoothable diagrammatic limit $g^1_d$. In particular, there exists a diagrammatic limit $g^1_d$ that is not solvable.
\end{proposition}

\subsection{Four Levels of Obstructions of Pre-limit $g^1_d$'s from Being Smoothable}\label{subS:obstruct}
The following two subsections are a preparation for the proof of the smoothing criterion in Section~\ref{S:SmoothingProof}.

We say a diagrammatic pre-limit $g^1_d$ represented by $(\mathcal{D},\mathcal{H})$ satisfies Level-I restriction if it is solvable and in the following we will introduce Level-II, Level-III and Level-IV restrictions which form additional obstructions of $(\mathcal{D},\mathcal{H})$  from being smoothable. 

Now assume that  $(\mathcal{D},\mathcal{H})$  is solvable with a solution $\rho$ and the corresponding bifurcation tree $\mathcal{B}$. Recall that $(\mathcal{D},\mathcal{H})$ satisfies the intrinsic global compatibility conditions if and only if $\mathcal{H}$ contains an admissible collection $\{g_p\}_{p\in\Gamma}$ of non-constant rational functions $g_p\in H_p$.

\begin{definition}
A \emph{bifurcation partition system} $\{{\vec{P}}_x\}_{x\in\mathcal{B}}$ on the bifurcation tree $\mathcal{B}$ is a collection of partitions $\vec{P}_x$ of $\Tan^+_\mathcal{B}(x)$ for all points $x\in\mathcal{B}$.
\end{definition}

Note that there are only finitely many possible bifurcation partition systems on $\mathcal{B}$ since $\Tan^+_\mathcal{B}(x)$ is a singleton for all but finitely many points $x\in\mathcal{B}$ and for the exceptions, $\vec{P}_x$ is a partition of a finite set.
 
\begin{remark}
Suppose $\mathcal{H}$ contains an admissible collection $G=\{g_p\}_{p\in\Gamma}$ of non-constant rational functions $g_p\in H_p$ (Definition~\ref{D:IGC}). This means that for each $x\in \mathcal{B}$, if $\Tan^+_\mathcal{B}(x)=\{t_1,\cdots,t_n\}$, then we can assign values $c_1, \cdots, c_n \in\kappa$ to $t_1,\cdots,t_n$ respectively such that $G\circ \Red (t)=c_i$ for  $i=1,\cdots,n$ and each $t\in\pi_{\mathcal{B}*}^{-1}(t_i)$. Canonically, we can associate a bifurcation partition system $\{{\vec{P}}_x\}_{x\in\mathcal{B}}$ to $G$ by letting  $t_i$ and $t_j$ be equivalent in $\vec{P}_x$ if and only if $c_i=c_j$. Moreover, 
we say $\{\vec{P}_x\}_{x\in\mathcal{B}}$ is \emph{globally compatible} to $\mathcal{H}$ (or $(\mathcal{D},\mathcal{H})$ even if this compatibility does not depend on $\mathcal{D}$) if $\mathcal{H}$ contains an admissible $G=\{g_p\}_{p\in\Gamma}$ such that $\{\vec{P}_x\}_{x\in\mathcal{B}}$ is exactly the bifurcation partition system associated to $G$. (In this case, to be more specific, we sometimes say $\{\vec{P}_x\}_{x\in\mathcal{B}}$ is globally compatible to $\mathcal{H}$ via $G$.) In addition, the intrinsic global compatibility conditions on $(\mathcal{D},\mathcal{H})$ can be restated equivalently as that there exists some bifurcation partition system globally compatible to $(\mathcal{D},\mathcal{H})$. 
\end{remark}

We state Level-II, Level-III and Level-IV restrictions as follows:
 
  \begin{enumerate}
  \item We say that a bifurcation partition system $\{\vec{P}_x\}_{x\in\mathcal{B}}$ is \emph{Level-II compatible} (respectively, \emph{Level-III compatible}) to $(\mathcal{D},\mathcal{H})$ if it satisfies the following property: for each point $p\in\Gamma$, the tangent directions $t_1$ and $t_2$ in $\Tan^+_\Gamma(p)$ are equivalent in the local diagram at $p$ (we also say $t_1$ and $t_2$ are locally equivalent), if (respectively, if and only if) $\pi_{\mathcal{B}*}(t_1)$ and $\pi_{\mathcal{B}*}(t_2)$ are equivalent in $\vec{P}_{\pi_\mathcal{B}(p)}$.
  \item We also call a bifurcation partition system Level-IV compatible to $(\mathcal{D},\mathcal{H})$ if it is globally compatible to $(\mathcal{D},\mathcal{H})$. 
  \item We say that $(\mathcal{D},\mathcal{H})$ satisfies Level-II (respectively, Level-III and Level-IV) restriction if there exists a bifurcation partition system which is Level-II (respectively, Level-III and Level-IV) compatible to $(\mathcal{D},\mathcal{H})$. Note that by the smoothing criterion (Theorem~\ref{T:main}) actually means that $(\mathcal{D},\mathcal{H})$ is smoothable if and only if Level-IV is satisfied.
 \end{enumerate}
 
The motivation for introducing the two intermediate compatibility levels (Level-II and Level-III) will be addressed in more details in our subsequent work. Note that Level-II and Level-III compatibilities can be determined purely combinatorially given the local diagrams induced by $\mathcal{H}$.

We denote the set of all bifurcation partition systems on $\mathcal{B}$ by $\BPDHone$, the sets of bifurcation partition systems Level-II, Level-III, and Level-IV compatible with $(\mathcal{D},\mathcal{H})$ by $\BPDHtwo$, $\BPDHthree$ and $\BPDHfour$ respectively.

\begin{lemma} \label{L:BPDH_4levels}
``Level~IV $\Rightarrow$ Level~III $\Rightarrow$ Level~II $\Rightarrow$ Level~I''  and $\BPDHone\supseteq\BPDHtwo\supseteq\BPDHthree\supseteq\BPDHfour$.
\end{lemma}

\begin{proof}
To see this note that we only need to verify Level~IV implies Level~III. 
Note that an equivalent way to say that  a bifurcation partition system $\{\vec{P}_x\}_{x\in\mathcal{B}}$ is globally (Level-IV) compatible to $\mathcal{H}$ via $\{g_p\}_{p\in\Gamma}\in H$is as follows: for each pair of tangent directions $t_1\in \Tan^{\rho+}_\Gamma(p_1)$ and $t_2\in\Tan^{\rho+}_\Gamma(p_2)$ such that $\pi_\mathcal{B}(p_1)=\pi_\mathcal{B}(p_2)$, we have the equivalence of the statements (1) $g_{p_1}(\red_{p_1}(t_1))=g_{p_2}(\red_{p_2}(t_2))$, and
(2) $\pi_{\mathcal{B}*}(t_1)$ and $\pi_{\mathcal{B}*}(t_2)$ in the same equivalence class in $\vec{P}_x$ where $x=\pi_\mathcal{B}(p_1)=\pi_\mathcal{B}(p_2)$.

Then by specializing to cases $p_1=p_2$, we conclude that Level~IV  implies Level~III.
\end{proof}

\begin{example} \label{E:BPDH}
Consider a saturated metrized complex $\mathfrak{C}=(\Gamma,\{C_p\}_{p \in \Gamma})$ where the underlying metric graph $\Gamma$ is as shown in Figure~\ref{F:Bif} and all $C_p$'s are projective lines over $\mC$. 
Suppose that $(\mathcal{D},\mathcal{H})$ is a solvable diagrammatic limit $g^1_d$ with solution $\rho$, the corresponding bifurcation tree $\mathcal{B}$ and the canonical projection  $\pi_\mathcal{B}:\Gamma\rightarrow \mathcal{B}$ as in  Figure~\ref{F:Bif}  (see Example~\ref{E:Bif}). For each $p\in\Gamma$, fix a basis $\{1,f_p\}$ of $H_p$ where $f_p$ is a non-constant rational function. 

Here we apply Algorithm~\ref{A:IGC} to check the smoothability of $(\mathcal{D},\mathcal{H})$, i.e., whether $\mathcal{H}$ contains an admissible $G=\{g_p\}_{p\in\Gamma}$. Note that by the algorithm, we only need to check the finite set of exceptional points (see definition in Subsection~\ref{S:Finite}) which in this case is the vertex set, i.e.,   $\mathcal{E}_{\rho}=\{o_1,o_2,p_1,p_2,p_3,q_1,q_2,q_3\}$. For adjacent $p,q\in\mathcal{E}_{\rho}$, we let $t_{pq}$ represent the tangent direction in $\Tan_\Gamma(p)$  emanating from $p$ towards $q$ and then $\red_p(t_{pq})$ is the marked point on $C_p$ associated to $t_{pq}$. Note that 
\begin{align*}
x&=\pi_\mathcal{B}(o_1)=\pi_\mathcal{B}(o_2), \\
 y_1&=\pi_\mathcal{B}(p_1), \quad y_2=\pi_\mathcal{B}(p_2)=\pi_\mathcal{B}(p_3), \\
 z_1&=\pi_\mathcal{B}(q_1), \quad z_2=\pi_\mathcal{B}(q_2), \quad z_3=\pi_\mathcal{B}(q_3),
\end{align*}
and
\begin{align*}
t_{xy_1}&=\pi_{\mathcal{B}*}(t_{o_1p_1})=\pi_{\mathcal{B}*}(t_{o_2p_1}), \\
t_{xy_2}&=\pi_{\mathcal{B}*}(t_{o_2p_2})=\pi_{\mathcal{B}*}(t_{o_2p_3}), \\
t_{y_2z_1}&=\pi_{\mathcal{B}*}(t_{p_2q_1})=\pi_{\mathcal{B}*}(t_{p_3q_1}), \\
t_{y_2z_2}&=\pi_{\mathcal{B}*}(t_{p_2q_2})=\pi_{\mathcal{B}*}(t_{p_3q_2}), \\
t_{y_2z_3}&=\pi_{\mathcal{B}*}(t_{p_2q_3})=\pi_{\mathcal{B}*}(t_{p_3q_3}).
\end{align*}

Consider a collection of finitely many variables $\{\alpha_p,\beta_p\}_{p\in\mathcal{E}_\rho}$. Then the algorithm reduces the intrinsic global compatibility conditions to the solvability of the following linear equations of variables $\{\alpha_p,\beta_p\}_{p\in\mathcal{E}_\rho}$:

\begin{align*}
\alpha_{o_1} + f_{o_1}(\red_{o_1}(t_{o_1p_1}))\beta_{o_1}&=\alpha_{o_2} + f_{o_2}(\red_{o_2}(t_{o_2p_1}))\beta_{o_2}\\
 f_{o_2}(\red_{o_2}(t_{o_2p_2}))&= f_{o_2}(\red_{o_2}(t_{o_2p_3}))\\
\alpha_{p_2} + f_{p_2}(\red_{p_2}(t_{p_2q_1}))\beta_{p_2}&=\alpha_{p_3} + f_{p_3}(\red_{p_3}(t_{p_3q_1}))\beta_{p_3}\\ 
 \alpha_{p_2} + f_{p_2}(\red_{p_2}(t_{p_2q_2}))\beta_{p_2}&=\alpha_{p_3} + f_{p_3}(\red_{p_3}(t_{p_3q_2}))\beta_{p_3}\\ 
 \alpha_{p_2} + f_{p_2}(\red_{p_2}(t_{p_2q_3}))\beta_{p_2}&=\alpha_{p_3} + f_{p_3}(\red_{p_3}(t_{p_3q_3}))\beta_{p_3}.
\end{align*}
Now suppose
\begin{align*}
& f_{o_1}(\red_{o_1}(t_{o_1p_1}))=1, \\
 & f_{o_2}(\red_{o_2}(t_{o_2p_1}))=2,  \quad  f_{o_2}(\red_{o_2}(t_{o_2p_2}))= f_{o_2}(\red_{o_2}(t_{o_2p_3})) = 1, \\
& f_{p_2}(\red_{p_2}(t_{p_2q_1})) =  f_{p_2}(\red_{p_2}(t_{p_2q_2})) = 1, \quad  f_{p_2}(\red_{p_2}(t_{p_2q_3}))  = -2,\\
& f_{p_3}(\red_{p_3}(t_{p_3q_1})) =  f_{p_3}(\red_{p_3}(t_{p_3q_2})) = 2,  \quad f_{p_3}(\red_{p_3}(t_{p_3q_3})) = -1.
 \end{align*}
Then  the above system of  linear equations is solvable and thus $(\mathcal{D},\mathcal{H})$ is smoothable by Algorithm~\ref{A:IGC}. In particular, we have a solution $\alpha_{o_1} =0,\beta_{o_1}=2, \alpha_{o_2} =0,\beta_{o_2}=1,\alpha_{p_2} =2,\beta_{p_2}=1,\alpha_{p_3} =1,\beta_{p_3}=1$. Therefore, if we let $g_{o_1}=2f_{o_1},g_{o_2}=f_{o_2},g_{p_1}=f_{p_1},g_{p_2}=2+f_{p_2},g_{p_3}=1+f_{p_3},g_{q_1}=f_{q_1},g_{q_2}=f_{q_2},g_{q_3}=f_{q_3}$ and $g_p = f_p$ for all $p\notin \mathcal{E}_{\rho}$, then $G=\{g_p\}_{p\in\Gamma}$ is admissible. 

Moreover, let $\mathcal{P}$ be the unique bifurcation partition system  which has a partition $\{\{t_{xy_1}\},\{t_{xy_2}\}\}$ at $x$ and $\{\{t_{y_2z_1},t_{y_2z_2}\},\{t_{y_2z_3}\}\}$ at $y_2$, and $\mathcal{P}'$ be the unique bifurcation partition system  which has a partition $\{\{t_{xy_1}\},\{t_{xy_2}\}\}$ at $x$ and $\{\{t_{y_2z_1}\},\{t_{y_2z_2}\},\{t_{y_2z_3}\}\}$ at $y_2$. Then,  $\mathcal{P}$ is  the bifurcation partition system  associated to $G$,  $\BPDHthree=\BPDHfour=\{\mathcal{P}\}$ and $\BPDHtwo=\{\mathcal{P},\mathcal{Q}\}$.

\qed
\end{example}

\subsection{The Spaces $\LDHone$, $\LDHtwo$, $\LDHthree$ and $\LDHfour$ of Partition Trees}

Suppose that $(\mathcal{D},\mathcal{H})$ represents a solvable diagrammatic pre-limit $g^1_d$ with a solution $\rho$ and the corresponding bifurcation tree $\mathcal{B}$. As we've defined four levels of compatibilities between bifurcation partition systems and $(\mathcal{D},\mathcal{H})$ in the previous subsection, here we define four levels of compatibilities between bifurcation trees and $(\mathcal{D},\mathcal{H})$, and construct the spaces $\LDHone$, $\LDHtwo$, $\LDHthree$ and $\LDHfour$ of partition trees as follows.

\begin{enumerate}
\item Let $\LDHone:=\Lambda_\rho$. (We let $\LDHone = \emptyset$ when $(\mathcal{D},\mathcal{H})$ is non-solvable.) In other words, $\LDHone$ is the space of all possible partition trees with respect to a solution $\rho$ of the global diagram defined by $(\mathcal{D},\mathcal{H})$. 

\item We say that a partition tree $\mathcal{T}$ is Level-II (respectively, Level-III) compatible with $(\mathcal{D},\mathcal{H})$ if for every  point $p \in \Gamma$ and each pair of tangent directions $t_1,t_2\in\Tan^+_\Gamma(p)$, we have $t_1$ is locally equivalent to $t_2$ if (respectively, if and only if) $\pi_{\mathcal{T}*}(t_1)=\pi_{\mathcal{T}*}(t_2)$.

\item We say that a partition tree $\mathcal{T}$ is globally (or Level-IV) compatible with $(\mathcal{D},\mathcal{H})$ if there exists a collection $G=\{g_p\}_{p\in\Gamma}$ of non-constant functions $g_p\in H_p$ such that one of the following equivalent statements are satisfied:
\begin{enumerate}
\item There is a function $\xi:\coprod_{x\in\mathcal{T}} \Tan^+_\mathcal{T}(x)\rightarrow \kappa$ such that $\xi$ is injective restricted to $\Tan^+_\mathcal{T}(x)$ for each $x\in\mathcal{T}$, and $G\circ\Red(t) = \xi\circ\pi_\mathcal{T*}(t)$ for all  $t\in\coprod_{p\in\Gamma}\Tan^{\rho+}_\Gamma(p)$. 
\item Whenever $t_1\in \Tan^{\rho+}_\Gamma(p_1)$ and $t_2\in\Tan^{\rho+}_\Gamma(p_2)$ where $\pi_\mathcal{T}(p_1)=\pi_\mathcal{T}(p_2)$, we have $g_{p_1}(\red_{p_1}(t_1))=g_{p_2}(\red_{p_1}(t_2))$ if and only if $\pi_{\mathcal{T}*}(t_1)=\pi_{\mathcal{T}*}(t_2)$.
\end{enumerate}
 In this sense, to be more specific, we also say that $\mathcal{T}$ and $(\mathcal{D},\mathcal{H})$ are globally compatible via $\{g_p\}_{p\in\Gamma}$. 
\item Denote by $\LDHtwo$, $\LDHthree$ and $\LDHfour$ the spaces of partition trees Level-II, Level-III and Level-IV compatible with $(\mathcal{D},\mathcal{H})$ respectively.
\item As in Lemma~\ref{L:BPDH_4levels}, one can easily see that $\LDHone\supseteq\LDHtwo\supseteq\LDHthree\supseteq\LDHfour$. An example of $\LDHone$, $\LDHtwo$, $\LDHthree$ and $\LDHfour$ for a simple metric graphs is shown in Example~\ref{E:lattice} in the Appendix.
\end{enumerate}

There is a natural map $\phiLambda:\LDHone\rightarrow\BPDHone$ such that for every $x\in\mathcal{B}$, every two forward tangent directions $t_1$ and $t_2$ in $\Tan^+_\mathcal{B}(x)$  are equivalent in $\phiLambda(\mathcal{T})$ if and only if they are pushed forward to the same tangent direction in $\Tan^+_\mathcal{T}(\Theta^\mathcal{B}_\mathcal{T}(x))$ by $\Theta^\mathcal{B}_{\mathcal{T}}$. 

On the other hand, given a  bifurcation partition system $\{\vec{P}_x\}_{x\in\mathcal{B}}$, we want to construct a partition tree. For a small enough $\delta$ (precisely, we can let $\delta$ be less than the minimal distance between two distinct exceptional values of $\rho$), we derive a metric in the following way: for each point $x\in\BifB$ and each equivalence class $E\subseteq\Tan^+_\mathcal{B}(x)$ in $\vec{P}_x$ of $\Tan^+_\mathcal{B}(x)$, we glue isometrically all the segments of length $\delta$ with one endpoint being $x$ and the other being in a forward tangent direction of $\mathcal{B}$ in $E$. Then it can be easily seen that $\mathcal{T}$ is a partition tree induced by the gluing and $\phiLambda(\mathcal{T})=\{\vec{P}_x\}_{x\in\mathcal{B}}$. We call the partition tree constructed in this way the $\delta$-\emph{glued partition tree} with respect to $\{\vec{P}_x\}_{x\in\mathcal{B}}$. An example is shown in Figure~\ref{F:delta_glued} where a $\delta$-glued partition tree on the right panel is derived from the bifurcation tree in Figure~\ref{F:Bif} and a bifurcation partition system on the left panel.
\begin{figure}[tbp]
\centering
\begin{tikzpicture}[>=to,x=2cm,y=1cm]
\coordinate (x) at (0,2);
    \fill [black] (x) circle (2.5pt);
    \draw (x) node[anchor=east] {\Large $x$};
\coordinate (y1) at (1,3);
    \fill [black] (y1) circle (2.5pt);
    \draw (y1) node[anchor=south] {\Large $y_1$};
\coordinate (y2) at (1,1);
    \fill [black] (y2) circle (2.5pt);
    \draw (y2) node[anchor=north] {\Large $y_2$};
\coordinate (z1) at (2,2);
    \fill [black] (z1) circle (2.5pt);
    \draw (z1) node[anchor=west] {\Large $z_1$};
\coordinate (z2) at (2,1);
    \fill [black] (z2) circle (2.5pt);
    \draw (z2) node[anchor=west] {\Large $z_2$};
\coordinate (z3) at (2,0);
    \fill [black] (z3) circle (2.5pt);
    \draw (z3) node[anchor=west] {\Large $z_3$};

\coordinate (xx) at (3.5,2);
    \fill [black] (xx) circle (2.5pt);
\coordinate (xy) at (3.8,2);
    \fill [black] (xy) circle (1pt);
\coordinate (yy1) at (4.5,2.7);
    \fill [black] (yy1) circle (2.5pt);
\coordinate (yy2) at (4.5,1.3);
    \fill [black] (yy2) circle (2.5pt);
\coordinate (yz) at (4.8,1.3);
    \fill [black] (yz) circle (1pt);
\coordinate (zz1) at (5.5,2);
    \fill [black] (zz1) circle (2.5pt);
\coordinate (zz2) at (5.5,1.3);
    \fill [black] (zz2) circle (2.5pt);
\coordinate (zz3) at (5.5,0.6);
    \fill [black] (zz3) circle (2.5pt);

\draw [-stealth, dashed, line width=1.8pt] (2.5,1.5) -- (3,1.5);

\draw [line width=1.8pt] (x)-- (y1);
\draw [line width=1.8pt] (x)-- (y2);
\draw [line width=1.8pt] (y2)-- (z1);
\draw [line width=1.8pt] (y2)-- (z2);
\draw [line width=1.8pt] (y2)-- (z3);

\draw [line width=1.8pt] (xx)-- (xy);
\draw [line width=1.8pt] (xy)-- (yy1);
\draw [line width=1.8pt] (xy)-- (yy2);
\draw [line width=1.8pt] (yy2)-- (yz);
\draw [line width=1.8pt] (yy2)-- (zz1);
\draw [line width=1.8pt] (yz)-- (zz2);
\draw [line width=1.8pt] (yz)-- (zz3);

\draw [dashed, line width=1pt] (xx)-- (3.5,1);
\draw [dashed, line width=1pt] (xy)-- (3.8,1);
\draw [dashed, line width=1pt] (yy2)-- (4.5,0.3);
\draw [dashed, line width=1pt] (yz)-- (4.8,0.3);

\draw (3.65,1.5) node {\Large $\delta$};
\draw (4.65,0.8) node {\Large $\delta$};
\draw (1,0) node {\LARGE $\mathcal{B}$};
\draw (3.3,-0.1) node[anchor = west] {\Large $\delta$-glued partition tree};

\draw [rotate around={0:(x)},line width=1.2pt,dash pattern=on 2pt off 2pt] (0.2,2) ellipse (0.05 and 0.4);
\draw [rotate around={0:(y2)},line width=1.2pt,dash pattern=on 2pt off 2pt] (1.35,0.8	) ellipse (0.05 and 0.3);
\draw [rotate around={0:(y2)},line width=1.2pt,dash pattern=on 2pt off 2pt] (1.35,1.35	) ellipse (0.07 and 0.14);
\end{tikzpicture}
\caption{An illustration of a $\delta$-glued partition tree.}\label{F:delta_glued}
\end{figure}
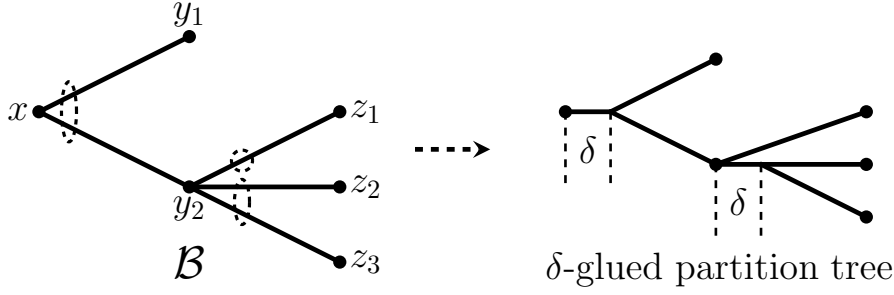
 
\begin{proposition} \label{P:LDH_BPDH}
$\phiLambda$ is a surjection from $\LDHone$ to $\BPDHone$. In addition, the image of $\phiLambda$ restricted to $\LDHtwo$, $\LDHthree$ and $\LDHfour$ are $\BPDHtwo$, $\BPDHthree$ and $\BPDHfour$ respectively.
\end{proposition}

\begin{proof}
The surjectivity of $\phiLambda$ follows from the construction of the $\delta$-glued partition tree with respect to any bifurcation partition system. 

 Now let us show that the image of $\phiLambda$ restricted to $\LDHtwo$ is contained in $\BPDHtwo$, and the image of $\phiLambda$ restricted to $\LDHthree$ is contained in $\BPDHthree$. Suppose $\mathcal{T}$ is Level-II (respectively, Level-III) compatible  to  $(\mathcal{D},\mathcal{H})$. For all $p\in\Gamma$ and each pair $t_1,t_2\in\Tan^{\rho+}_\Gamma(p)$, we have $\pi_{\mathcal{B}*}(t_1), \pi_{\mathcal{B}*}(t_2)\in\Tan^+_\mathcal{B}(\pi_\mathcal{B}(p))$. Also,  $\pi_{\mathcal{B}*}(t_1)\sim\pi_{\mathcal{B}*}(t_2)$ in $\phiLambda(\mathcal{T})$ if and only if $\pi_{\mathcal{T}*}(t_1)=\pi_{\mathcal{T}*}(t_2)$, since $\pi_\mathcal{T}$ factors through $\pi_\mathcal{B}$. Therefore, we derive that $t_1$ and $t_2$ in ${\Tan^{\rho+}_\Gamma(p)}$ are locally equivalent if (respectively, if and only if) $\pi_{\mathcal{B}*}(t_1)\sim\pi_{\mathcal{B}*}(t_2)$ in $\phiLambda(\mathcal{T})$. Therefore, $\phiLambda(\mathcal{T})$ is a bifurcation partition system Level-II (respectively, Level-III) compatible to $(\mathcal{D},\mathcal{H})$.

Next we will show that the image of $\phiLambda$  restricted to $\LDHfour$ is contained in $\BPDHfour$. Let $\mathcal{T}$ be a partition tree globally compatible with $(\mathcal{D},\mathcal{H})$ via  $\{g_p\}_{p\in\Gamma}\in \mathcal{H}$.  Then one can see that  $\phiLambda(\mathcal{T})$ will be a bifurcation partition system globally compatible to $(\mathcal{D},\mathcal{H})$  via $\{g_p\}_{p\in\Gamma}$. 

Conversely, we will show that the image of $\phiLambda$ restricted to $\LDHtwo$, $\LDHthree$ and $\LDHfour$  contains $\BPDHtwo$, $\BPDHthree$ and $\BPDHfour$ respectively. To this end,  we show that a $\delta$-glued partition tree with respect to a bifurcation partition system in $\BPDHtwo$, $\BPDHthree$ and $\BPDHfour$ respectively, is an element in $\LDHtwo$, $\LDHthree$ and $\LDHfour$ respectively. The first two cases for $\BPDHtwo$ and $\BPDHthree$ are straightforward by definitions.

Now let $\{\vec{P}_x\}_{x\in\mathcal{B}}\in \BPDHfour$ be the bifurcation partition system associated to an admissible $\{g_p\}_{p\in \Gamma}\in\mathcal{H}$ and $\mathcal{T}$ be its corresponding $\delta$-glued partition tree.  
We claim that $\mathcal{T}\in\LDHfour$. To show this, we want to tune $\{g_p\}_{p\in \Gamma}$ into $\{g'_p\}_{p\in \Gamma}\in\mathcal{H}$ and define a function $\xi:\coprod_{x\in\mathcal{T}} \Tan^+_\mathcal{T}(x)\rightarrow \kappa$ such that  $\mathcal{T}$ and $(\mathcal{D},\mathcal{H})$ are compatible via $\{g_p\}_{p\in\Gamma}$. 
\begin{enumerate}
\item Assign values $g_p(\red_p(t))$ to $\pi_{\mathcal{T}*}(t)$ for all $p\in\mathcal{E}_\rho$ and $t\in\Tan^{\rho+}_\Gamma(p)$
\item Assign values in $\kappa$ for the remaining elements in  $\coprod_{x\in\mathcal{T}} \Tan^+_\mathcal{T}(x)$ such that for each $x\in\mathcal{T}$, $\xi$ restricted to $\Tan^+_\mathcal{T}(x)$ is injective.
\item Let $g'_p=g_p$ for all $p\in\mathcal{E}_\rho$.
\item For an ordinary point $p\in\mathcal{O}_\rho$,  there is a unique forward tangent direction (denoted by $t$) at $p$. Therefore, we can always find a non-constant rational function $g'_p\in H_p$ such that $g'_p(\red_p(t))=\xi(\pi_{\mathcal{T}*}(t))$ (Lemma~\ref{L:tune}). 
\end{enumerate}

By the above construction,  $\mathcal{T}$ is globally compatible with $(\mathcal{D},\mathcal{H})$ via $\{g'_p\}_{p\in \Gamma}$, which means $\mathcal{T}\in\LDHfour$.
\end{proof}

\begin{remark}
We will employ the surjectivity of the map $\phiLambda|_{\LDHfour}:\LDHfour\rightarrow\BPDHfour$ in the proof of the smoothing criterion in the next section.
\end{remark}

\section{Proof of the Smoothing Criterion}\label{S:SmoothingProof}
We restate the smoothing criterion combining the two versions as follows.
\begin{theorem*}
Given a saturated metrized complex $\mathfrak{C}$ and a pre-limit $g^1_d$ represented by $(\mathcal{D},\mathcal{H})$, the following statements are equivalent:
\begin{enumerate}

\item  $(\mathcal{D},\mathcal{H})$ is smoothable.

\item  $(\mathcal{D},\mathcal{H})$ is solvable and $\LDHfour$ is nonempty.

\item $(\mathcal{D},\mathcal{H})$ is solvable and satisfies the intrinsic global compatibility conditions.
\end{enumerate}
\end{theorem*}

\begin{proof}
(2) is equivalent to saying that there exists a partition tree globally compatible to $(\mathcal{D},\mathcal{H})$, and (3) is equivalent to saying that there exists a bifurcation partition system globally compatible to $(\mathcal{D},\mathcal{H})$. Then the equivalence of (2) and (3) follows from Proposition~\ref{P:LDH_BPDH}.

(1) $\Rightarrow$ (2):

Let $(\mathcal{D},\mathcal{H})$ represent a smoothable pre-limit $g^1_d$ on $\mathfrak{C}$.
Then by Theorem~\ref{T:SmoothingHar}, we know that there exists a harmonic morphism $\mathfrak{C}\phi^\text{mod}=(\phi_{\Gamma^\text{mod}},\{\phi_p\}_{p \in \Gamma^\text{mod}})$ of degree $\deg(\mathcal{D})$ from a modification $\mathfrak{C^\text{mod}}$ of $\mathfrak{C}$ to a genus zero saturated metrized complex $\mathfrak{C}(T)$  such that (1) $\mathcal{D}$ is the retraction to $\mathfrak{C}$ of a divisor $\mathcal{D}^\text{mod}$ on $\mathfrak{C}\phi^\text{mod}$ which is a pullback divisor by $\mathfrak{C}\phi^\text{mod}$ of an effective degree one divisor on $\mathfrak{C}(T)$, and  (2) $\phi_p$ coincides with the morphism from $C_p$ to $\mathbb{P}^1_{\kappa}$ defined by $H_p$. Here the underlying metric graphs of $\mathfrak{C}$, $\mathfrak{C}^\text{mod}$ and $\mathfrak{C}(T)$ are denoted by $\Gamma$, $\Gamma^\text{mod}$ and $T$ respectively. Now it remains to show that $T$ must be an element in $\LDHfour$.

Denote by $r(T)$ the root of $T$ which is the image of $D_\Gamma^\text{mod}$ (the tropical part of $\mathcal{D}^\text{mod}$) under $\phi_{\Gamma^\text{mod}}$.
Then $r(T)$ and the map $\phi_\Gamma:\Gamma\rightarrow T$ which is the restriction to $\Gamma$ of the harmonic morphism $\phi_{\Gamma^\text{mod}}:\Gamma^\text{mod}\rightarrow T$  induces a global diagram on $\Gamma$ in the following way: for any point $p \in \Gamma$ and any tangent direction $t \in \Tan_\Gamma(p)$, the multiplicity $m_1(p,t)$ is the expansion factor with sign `$-$' if the pushforward of $t$ by $\phi_\Gamma$ coincides with the tangent direction on $T$ along the unique path from $\phi_\Gamma(p)$ to $r(T)$, and with sign `$+$' otherwise.

On the other hand, we also construct local diagrams from $(\mathcal{D},\mathcal{H})$ (Remark~\ref{R:LocalHp}). In particular, the multiplicity $m_2(p,t)$ in the local diagram at point $p$ associated to $(\mathcal{D},\mathcal{H})$ equals the ramification index of $\phi_p$ at $\red_p(t)$ with an appropriate sign. By the compatibility property of the harmonic morphisms, we know that the ramification index of $\phi_p$ at the marked point $\red_p(t)$ on $C_p$ corresponding to $t$ equals the expansion factor of $\phi_\Gamma$ at $t$. Therefore $m_1(p,t)=m_2(p,t)$, which means $(\mathcal{D},\mathcal{H})$ is diagrammatic and solvable with a solution $d^0_T\circ\phi_\Gamma$ where $d^0_T$ is the distance from points on $T$ to the root point $r(T)$.  Hence, $(T,\phi_\Gamma)$ is an element of $\LDHone$. In addition,   the compatibility property of the harmonic morphisms also guarantees that $(T,\phi_\Gamma)$ must be in $\LDHfour$.

(2) $\Rightarrow$ (1):
 Let $(T,\pi_T)\in\LDHfour$ be globally compatible with $(\mathcal{D},\mathcal{H})$. 
 
 To show that $(\mathcal{D},\mathcal{H})$ is smoothable, using Theorem~\ref{T:SmoothingHar}, we need to find some harmonic morphism  compatible to $(\mathcal{D},\mathcal{H})$. More precisely,  we want to construct a modification $\mathfrak{C}^\text{mod}$ of $\mathfrak{C}$ with underlying metric graph $\Gamma^\text{mod}$ such that (1) $\phi_{\Gamma^\text{mod}}:\Gamma^\text{mod}\rightarrow T$ is a harmonic morphism (between metric graphs) whose restriction to $\Gamma$ is $\pi_T$, and (2) $\phi_{\Gamma^\text{mod}}$ can be lifted to a harmonic morphism $\mathfrak{C}\phi^\text{mod}$ from $\mathfrak{C}^\text{mod}$ to a genus zero saturated metrized complex $\mathfrak{C}(T)$ whose underlying metric tree is $T$. In Subsection~\ref{S:harmmor}, we have also introduced the notion of pseudo-harmonic morphism which does not require the balancing condition as for harmonic morphisms. In the rest of the proof, we will show that we can first find a compatible pseudo-harmonic morphism from $\mathfrak{C}$ to a genus zero saturated metrized complex $\mathfrak{C}(T)$. Then we will show that we can always extend this pseudo-harmonic morphism to a desired harmonic morphism by generating a suitable modification $\mathfrak{C}^\text{mod}$ of $\mathfrak{C}$, while we single out the statement and proof in Proposition~\ref{P:modification} together with Example~\ref{E:Modification} to aid our exposition.

Assume $\{g_p\}_{p\in\Gamma}$ is a collection of rational functions $g_p\in H_p$ on $C_p$ that makes $(\mathcal{D},\mathcal{H})$ and $(T,\pi_T)$ compatible. There is a function $\xi:\coprod_{x\in T} \Tan^+_T(x)\rightarrow \kappa$ such that $\xi$ is injective restricted to $\Tan^+_T(x)$ for each $x\in T$, and $g_p\circ\red_p = \xi\circ\pi_{T*}$ for all $p\in\Gamma$. Let $\bar{g}_p:C_p\rightarrow \kappa_\infty$ be the function on $C_p$ extending $g_p$ to its poles, and let $\bar{\xi}:\coprod_{x\in T} \Tan_T(x)\rightarrow \kappa_\infty$ be the extension of $\xi$ such that for each $x\in T$, $\bar{\xi}$ maps the incoming tangent direction at $x$ to $\infty$ in $\kappa_\infty$. Then we also have $\bar{g}_p\circ\red_p = \bar{\xi}\circ\pi_{T*}$.

Since $\kappa_\infty$ is isomorphic to a projective line over $\kappa$, we can build a genus zero metrized complex $\mathfrak{C}(T)$ from $T$ (for all $x\in T$, the curve $C'_x$ associated to $x$ is a projective line) by letting $\gamma_x\circ\bar{\xi}_x$ be the reduction map at $x\in T$ where $\bar{\xi}_x$ is the function $\bar{\xi}$ restricted to $\Tan_T(x)$ and $\gamma_x:\kappa_\infty \stackrel{\sim}{\rightarrow} C'_x$ is the isomorphism between $\kappa_\infty$ and $C'_x$.

Now let $\phi_p=\gamma_{\pi_T(p)}\circ\bar{g}_p$. Then $(\pi_T,\{\phi_p\}_{p\in\Gamma})$ is a pseudo-harmonic morphism from $\mathfrak{C}$ to $\mathfrak{C}(T)$, since the compatibility conditions of a pseudo-harmonic morphism (Definition~\ref{D:pseu_har_morph}) are guaranteed by the solvability of $(\mathcal{D},\mathcal{H})$ and the relation $\phi_p\circ\red_p=\gamma_{\pi_T(p)}\circ\bar{g}_p\circ\red_p = \gamma_{\pi_T(p)}\circ \bar{\xi}_{\pi_T(p)}\circ\pi_{T*}$ where $\gamma_{\pi_T(p)}\circ \bar{\xi}_{\pi_T(p)}$ is the reduction map at $\pi_T(p)$ by the construction of $\mathfrak{C}(T)$.

By Proposition~\ref{P:modification}, we can extend the pseudo-harmonic morphism $(\pi_T,\{\phi_p\}_{p\in\Gamma})$ to a harmonic morphism $(\phi^\text{mod},\{\phi_p\}_{p\in\Gamma^\text{mod}})$ from a modification $\mathfrak{C}^\text{mod}$ of $\mathfrak{C}$ to $\mathfrak{C}(T)$.
\end{proof}

\begin{proposition}\label{P:modification}
Let $\mathfrak{C}\phi$ be a pseudo-harmonic morphism from a saturated metrized complex $\mathfrak{C}$ to a saturated metrized complex $\mathfrak{C}(T)$ of genus zero. If $\mathfrak{C}\phi$ is harmonic at all but finitely many points in its underlying metric graph, then there is a modification $\mathfrak{C}^\text{mod}$ of $\mathfrak{C}$ and a harmonic morphism $\mathfrak{C}\phi^\text{mod}$ from $\mathfrak{C}^\text{mod}$ to $\mathfrak{C}(T)$ such that $\mathfrak{C}\phi$ is the restriction of $\mathfrak{C}\phi^\text{mod}$ to $\mathfrak{C}$.
\end{proposition}

\begin{proof}
Assume that the underlying metric graphs of $\mathfrak{C}$ and $\mathfrak{C}(T)$ are $\Gamma$ and $T$ respectively and the associated curves of $\mathfrak{C}$ and $\mathfrak{C}(T)$  are $\{C_p\}_{p\in\Gamma}$ and $\{C'_x\}_{x\in T}$ respectively. Let $\mathfrak{C}\phi=(\phi_\Gamma,\{\phi_p\}_{p\in\Gamma})$ where $\phi_\Gamma:\Gamma\rightarrow T$ is the associated pseudo-harmonic morphism of metric graphs and $\phi_p:C_p\rightarrow C'_{\phi_\Gamma(p)}$ is the associated finite morphism of curves at $p$. We will derive a modification $\mathfrak{C}^\text{mod}$ of $\mathfrak{C}$ in the following way.

Consider a point $q\in\Gamma$. For each tangent direction $t'\in\Tan_T(\phi_\Gamma(q))$ at $\phi_\Gamma(q)$ on $T$, let $u\in C_q$ be a non-marked point of $C_q$ which is an element in the fiber $\phi_q^{-1}(\red_{\phi_\Gamma(q)}(t'))$. Suppose the ramification index of $\phi_q$ at $u$ is $m$.

Let $T'$ be the connected component of $T \setminus \{\phi_\Gamma(q)\}$ corresponding to the tangent direction $t$ at $\phi_\Gamma(q)$. Let $T'_1,\cdots,T'_m$ be $m$ copies of $T'$. For $i=1,\cdots,m$, let $x_i$ be the open end of $T'_i$ corresponding to the open end $\phi_\Gamma(q)$ of $T'$ and let $y_i$ be the point in $T'_i$ with a small distance $l$ to $x_i$ ($l$ is less than the minimum distance of branching points of $T'$ to $\phi_\Gamma(q)$).

Now we want to attach $\Gamma$ with an extra branch $\Gamma'_u$ with respect to $u$. Then by equipping $\Gamma'_u$ with projective lines, we will get a modification of $\mathfrak{C}$ with respect to $u$.

We construct $\Gamma'_u$ from $T'_1,\cdots,T'_m$ by first identifying the segments $(x_i,y_i]$'s and then shrinking the glued segment by a factor of $m$. Denote by $(x,y]$ the corresponding segment in $\Gamma'_u$ with $x$ being its open end. Then by this construction, the length of $(x,y]$ is $l/m$ and $\Gamma'_u\setminus (x,y]$ is a disjoint union of $T'_i\setminus(x_i,y_i]$'s. Forgetting the compactness restriction of a metric graph, we also call $T'$, $T'_i$'s and $\Gamma'_u$ metric graphs. Then there is a natural harmonic morphism $\phi_{\Gamma'_u}$ from $\Gamma'_u$ to $T'$ where the balancing condition (Definition~\ref{D:har_morph}) is automatically satisfied by the construction of $\Gamma'_u$.

Let $\mathfrak{C}(T')$ be $\mathfrak{C}(T)$ restricted to $T'$. We can construct a saturated metrized complex $\mathfrak{C}(\Gamma'_u)$ with underlying metric graph $\Gamma'_u$ by associating each point $p\in\Gamma'_u$ with a projective line $C_p$. Let $x'=\phi_{\Gamma'_u}(p)$. The reduction map $\red_p$ at $p$ is derived as follows.
\begin{enumerate}
  \item If $p\in(x,y)$, then there are two tangent directions $t_1$ and $t_2$ in $\Tan_{\Gamma'_u}(p)$ and two tangent directions $t'_1$ and $t'_2$ in $\Tan_{T'}(x')$ where $t_1$ and $t_2$ are pullbacks of $t'_1$ and $t'_2$ by $\phi_{\Gamma'_u}$. Let $\phi_p:C_p\rightarrow C'_{x'}$ be a degree $m$ morphism from $C_p$ to $C'_{x'}$ (the curve associated to $\phi_{\Gamma'_u}(p)$ in $\mathfrak{C}(T')$) such that there are two points $v_1$ and $v_2$ in $C_p$ with ramification index $m$ over the marked points $\red_{x'}(t'_1)$ and $\red_{x'}(t'_2)$ in $C'_{x'}$ respectively. Let the marked point $\red_p(t_1)$ associated to $t_1$ be $v_1$ and the marked point $\red_p(t_2)$ associated to $t_2$ be $v_2$.
  \item If $p=y$, then there are $m+1$ tangent directions $t_1,\cdots,t_{m+1}$ in  $\Tan_{\Gamma'_u}(p)$ and two tangent directions $t'_1$ and $t'_2$ in $\Tan_{T'}(x')$. We may assume that $t'_1$ is the tangent direction corresponding to the edge between $x'$ and the open end of $T'$, $t_1$ is the pullback of $t'_1$ by $\phi_{\Gamma'_u}$, and $\{t_2,\cdots,t_{m+1}\}$ is the pullback of $t'_2$ by $\phi_{\Gamma'_u}$. Let $\phi_p:C_p\rightarrow C'_{x'}$ be a degree $m$ morphism from $C_p$ to $C'_{x'}$ such that there is a point $v_1\in C_p$ with ramification index $m$ over the marked point $\red_{x'}(t'_1)$ and there are distinct points $v_2,\cdots,v_{m+1}\in C_p$ with ramification index $1$ over the marked point $\red_{x'}(t'_2)$. Then we let the marked point $\red_p(t_i)$ associated to $t_i$ be $v_i$ for $i=1,\cdots,m+1$.
  \item If $p\in\Gamma'_u\setminus (x,y]$, then $\Tan_{T'}(x')$ pulls back bijectively to $\Tan_{\Gamma'_u}(p)$ by $\phi_{\Gamma'_u}$. We let $\phi_p:C_p\rightarrow C'_{x'}$ be an isomorphism. For every pair of corresponding tangent directions $t_p\in\Tan_{\Gamma'_u}(p)$ and $t_{x'}\in\Tan_{T'}(x')$,  we let the marked point $\red_p(t_p)$ associated to $t_p$ be $\phi_p^{-1}(\red_{x'}(t_{x'}))$.
\end{enumerate}

Note that in all cases above, the morphism $\phi_p:C_p\rightarrow C'_{x'}$ always exists since $C_p$ and $C'_{x'}$ are projective lines.
We conclude that $(\phi_{\Gamma'_u},\{\phi_p\}_{p\in\Gamma'_u})$ is a harmonic morphism from $\mathfrak{C}(\Gamma'_u)$ to $\mathfrak{C}(T')$, since $\phi_{\Gamma'_u}$ is a harmonic morphism of metric graphs and the compatibility conditions of Definition~\ref{D:pseu_har_morph} are automatically satisfied by the above construction of $\mathfrak{C}(\Gamma'_u)$.

Now we get a modification of $\Gamma$ with respect to $u$ by attaching the open end of the extra branch $\Gamma'_u$ to $\Gamma$ at $q$, and a modification of $\mathfrak{C}$ with respect to $u$ by adding $u$ as a marked point of $C_q$ and attaching $\mathfrak{C}(\Gamma'_u)$ to $\mathfrak{C}$. Moreover, the pseudo-harmonic morphisms $\phi_\Gamma$ and $\mathfrak{C}\phi$ also naturally extend respectively to these modifications of $\Gamma$ and $\mathfrak{C}$ with respect to $u$, which are harmonic at all the points in the extra branch $\Gamma'_u$ (not necessarily at $q$).

Recall that $u$ is a non-marked point of $C_q$ which at the same time is an element of $\phi_q^{-1}(\red_{\phi_\Gamma(q)}(t'))$ where $t'$ is a tangent direction at $\phi_\Gamma(q)$ on $T$. Therefore, we can get modifications of $\Gamma$ and $\mathfrak{C}$ with respect to $q$ by performing modifications of $\Gamma$ and $\mathfrak{C}$ to all possible $u$'s in this sense at the same time. Moreover, the pseudo-harmonic morphisms $\phi_\Gamma$ and $\mathfrak{C}\phi$ also naturally extend respectively to these modifications of $\Gamma$ and $\mathfrak{C}$ with respect to $q$, which are harmonic at the point $q$ and all the points in the extra branches. Note that if $\mathfrak{C}\phi$ is already harmonic at $q$, then no modification is performed.

The final modifications of $\Gamma$ and $\mathfrak{C}$, denoted by of $\Gamma^\text{mod}$ and $\mathfrak{C}^\text{mod}$ respectively, are derived by performing modifications of $\Gamma$ and $\mathfrak{C}$ at the same time to all $q\in\Gamma$ at which $\mathfrak{C}\phi$ is not harmonic. In this way, we get a harmonic morphism $\mathfrak{C}\phi^\text{mod}:\mathfrak{C}^\text{mod}\rightarrow\mathfrak{C}(T)$ as required.

\end{proof}

\begin{example} \label{E:Modification}
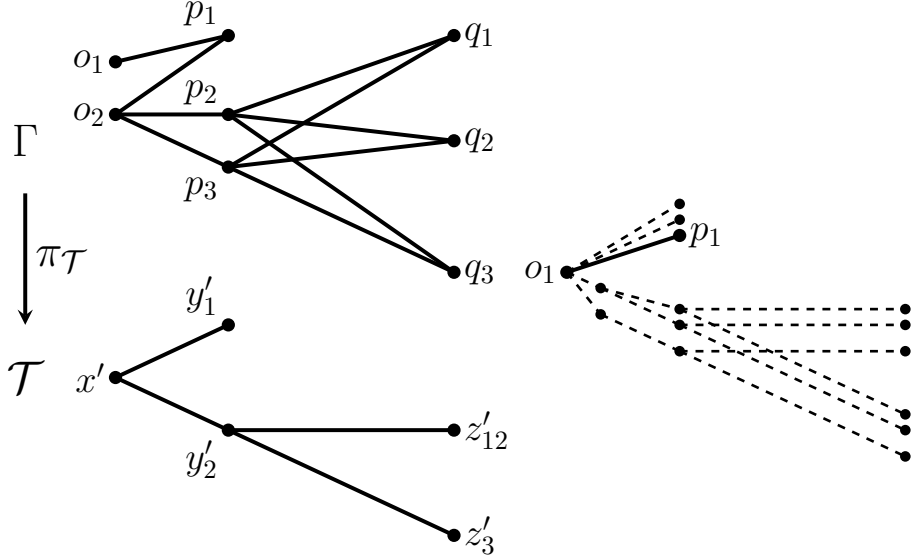
\begin{figure}[tbp]
\centering
\begin{tikzpicture}[>=to,x=1.5cm,y=0.7cm]
\coordinate (xx) at (0,3);
    \fill [black] (xx) circle (2.5pt);
    \draw (xx) node[anchor=east] {\Large $x'$};
\coordinate (yy1) at (1,4);
    \fill [black] (yy1) circle (2.5pt);
    \draw (yy1) node[anchor=south east] {\Large $y'_1$};
\coordinate (yy2) at (1,2);
    \fill [black] (yy2) circle (2.5pt);
    \draw (yy2) node[anchor=north east] {\Large $y'_2$};
\coordinate (zz12) at (3,2);
    \fill [black] (zz12) circle (2.5pt);
    \draw (zz12) node[anchor=west] {\Large $z'_{12}$};
\coordinate (zz3) at (3,0);
    \fill [black] (zz3) circle (2.5pt);
    \draw (zz3) node[anchor=west] {\Large $z'_3$};

\draw [line width=1.5pt] (xx)-- (yy1);
\draw [line width=1.5pt] (xx)-- (yy2);
\draw [line width=1.5pt] (yy2)-- (zz12);
\draw [line width=1.5pt] (yy2)-- (zz3);

\coordinate (o1) at (0,9);
    \fill [black] (o1) circle (2.5pt);
    \draw (o1) node[anchor=east] {\Large $o_1$};
\coordinate (o2) at (0,8);
    \fill [black] (o2) circle (2.5pt);
    \draw (o2) node[anchor=east] {\Large $o_2$};
\coordinate (p1) at (1,9.5);
    \fill [black] (p1) circle (2.5pt);
    \draw (p1) node[anchor=south east] {\Large $p_1$};
\coordinate (p2) at (1,8);
    \fill [black] (p2) circle (2.5pt);
    \draw (p2) node[anchor=south east] {\Large $p_2$};
\coordinate (p3) at (1,7);
    \fill [black] (p3) circle (2.5pt);
    \draw (p3) node[anchor=north east] {\Large $p_3$};
\coordinate (q1) at (3,9.5);
    \fill [black] (q1) circle (2.5pt);
    \draw (q1) node[anchor=west] {\Large $q_1$};
\coordinate (q2) at (3,7.5);
    \fill [black] (q2) circle (2.5pt);
    \draw (q2) node[anchor=west] {\Large $q_2$};
\coordinate (q3) at (3,5);
    \fill [black] (q3) circle (2.5pt);
    \draw (q3) node[anchor=west] {\Large $q_3$};

\draw [line width=1.5pt] (o1) -- (p1);
\draw [line width=1.5pt] (o2) -- (p1);
\draw [line width=1.5pt] (o2) -- (p2);
\draw [line width=1.5pt] (o2) -- (p3);
\draw [line width=1.5pt] (p2) -- (q1);
\draw [line width=1.5pt] (p2) -- (q2);
\draw [line width=1.5pt] (p2) -- (q3);
\draw [line width=1.5pt] (p3) -- (q1);
\draw [line width=1.5pt] (p3) -- (q2);
\draw [line width=1.5pt] (p3) -- (q3);

\draw (-0.8,3) node {\LARGE $\mathcal{T}$};
\draw (-0.8,7.5) node {\LARGE $\Gamma$};
\draw [-stealth, line width=1.5pt] (-0.8,6.5) -- (-0.8,4);
\draw (-0.8,5.25) node[anchor=west] {\LARGE $\pi_\mathcal{T}$};

\coordinate (mo1) at (4,5);
    \fill [black] (mo1) circle (2.5pt);
    \draw (mo1) node[anchor=east] {\Large $o_1$};
\coordinate (mp1) at (5,5.7);
    \fill [black] (mp1) circle (2.5pt);
    \draw (mp1) node[anchor=west] {\Large $p_1$};
\coordinate (mp12) at (5,6);
    \fill [black] (mp12) circle (2pt);
\coordinate (mp13) at (5,6.3);
    \fill [black] (mp13) circle (2pt);
\coordinate (mg1) at (4.3,4.7);
    \fill [black] (mg1) circle (2pt);
\coordinate (mg2) at (4.3,4.2);
    \fill [black] (mg2) circle (2pt);
\coordinate (my21) at (5,3.5);
    \fill [black] (my21) circle (2pt);
\coordinate (my22) at (5,4);
    \fill [black] (my22) circle (2pt);
\coordinate (my23) at (5,4.3);
    \fill [black] (my23) circle (2pt);
\coordinate (mz121) at (7,3.5);
    \fill [black] (mz121) circle (2pt);
\coordinate (mz122) at (7,4);
    \fill [black] (mz122) circle (2pt);
\coordinate (mz123) at (7,4.3);
    \fill [black] (mz123) circle (2pt);
\coordinate (mz31) at (7,1.5);
    \fill [black] (mz31) circle (2pt);
\coordinate (mz32) at (7,2);
    \fill [black] (mz32) circle (2pt);
\coordinate (mz33) at (7,2.3);
    \fill [black] (mz33) circle (2pt);

\draw [line width=1.5pt] (mo1) -- (mp1);
\draw [line width=1pt, dashed] (mo1) -- (mp12);
\draw [line width=1pt, dashed] (mo1) -- (mp13);
\draw [line width=1pt, dashed] (mo1) -- (mg1);
\draw [line width=1pt, dashed] (mo1) -- (mg2);
\draw [line width=1pt, dashed] (mg1) -- (my23);
\draw [line width=1pt, dashed] (mg1) -- (my22);
\draw [line width=1pt, dashed] (mg2) -- (my21);
\draw [line width=1pt, dashed] (my21) -- (mz121);
\draw [line width=1pt, dashed] (my22) -- (mz122);
\draw [line width=1pt, dashed] (my23) -- (mz123);
\draw [line width=1pt, dashed] (my21) -- (mz31);
\draw [line width=1pt, dashed] (my22) -- (mz32);
\draw [line width=1pt, dashed] (my23) -- (mz33);
\end{tikzpicture}
\caption{A modification performed at point $o_1$ based on the projection map $\pi_\mathcal{T}$ in Example~\ref{E:Partition} and the local data in $H_{o_1}$.}\label{F:Modification}
\end{figure}
In Figure~\ref{F:Modification}, we show how a modification is performed at point $o_1$ in Figure \ref{F:Partition} of Example ~\ref{E:Partition}. The image of $o_1$ under $\pi_\mathcal{T}$ is $x'$. First note that there are only one outgoing edge $o_1p_1$ from $o_1$ with expansion factor $1$ of the map $\pi_\mathcal{T}$. Suppose that the degree of the nonconstant rational function $g_{o_1}\in H_{o_1}$ is $3$. Suppose that the forward tangent direction from $x'$ to $y'_1$ corresponds to $c_1\in\kappa$ and the forward tangent direction from $x'$ to $y'_2$ corresponds to $c_2\in\kappa$. Suppose that $g_{o_1}^{-1}(c_1)=\{u_1,u_2,u_3\}$ and $u_1$ is the reduction of the tangent direction from $o_1$ to $p_1$. Then two copies of $x'y'_1$  will be attached to $o_1$ as extra branches corresponding to $u_2$ and $u_3$ respectively. Suppose that $g_{o_1}^{-1}(c_2)=\{v_1,v_2\}$ while the ramification index of $v_1$ is $1$ and the ramification index of $v_2$ is $2$. Let $T'$ be the subgraph of $\mathcal{T}$ connecting $x'$, $z'_{12}$ and $z'_3$. Then one copy of $T'$ is attached to $o_1$ as the extra branch corresponding to $v_1$. Accordingly, the extra branch corresponding to $v_2$ is made from two copies $T'_1$ and $T'_2$ of $T'$ by first gluing from the open ends of $T'_1$ and $T'_2$ along small segments of the same length  and then shrinking the glued segment by a factor of $2$.
\end{example}

\section{Applications}\label{application_sect}

We apply the smoothing criterion to the saturated metrized complex versions of certain types of curves : curves of compact type studied by Eisenbud and Harris, nodal curves with dual graph made of separate loops and curves considered by Harris and Mumford to characterize gonality stratification. We also extend the smoothing criterion to metrized complexes by showing a concrete example. 

\subsection{Saturated Metrized Complexes of Compact Type} \label{S:CompactType}
We show that every diagrammatic pre-limit $g^1_d$ on a saturated metrized complex of compact type (i.e., whose underlying metric graph is a metric tree) is smoothable.  This theorem is an analogue of Proposition 3.1 of \cite{EH86} by Eisenbud and Harris for curves of compact type.

\begin{theorem} \label{T:CompactType}
Every diagrammatic pre-limit $g^1_d$ represented by $(\mathcal{D},\mathcal{H})$ on a saturated metrized complex $\mathfrak{C}$ of compact type is smoothable.
\end{theorem}
\begin{proof}

The underlying metric graph $\Gamma$ of $\mathfrak{C}$ is a tree. Therefore $(\mathcal{D},\mathcal{H})$ must be solvable. Let $\rho$ be a solution to $(\mathcal{D},\mathcal{H})$ and $\mathcal{E}_\rho$ be the set of exceptional points of $\rho$ (Section~\ref{S:Finite}). Then we can subdivide $\Gamma$ into segments $L_1,\cdots,L_n$ by $\mathcal{E}_\rho$ (the end points of $L_i$ are exceptional points). The restrictions of $(\mathcal{D},\mathcal{H})$ to each $L_i$ must be smoothable since the intrinsic global compatibility conditions will be trivial. This theorem then follows directly from the following Proposition~\ref{P:DHglue}. 
\end{proof}

Consider a diagrammatic pre-limit $g^1_d$ represented by $(\mathcal{D},\mathcal{H})$ on a saturated metrized complex $\mathfrak{C}$  whose underlying metric graph is $\Gamma$. Let $\Gamma'$ be a connected closed metric subgraph of $\Gamma$.  A saturated metrized complex $\mathfrak{C}'$ is said to be the restriction of $\mathfrak{C}$ to $\Gamma'$ if  the underlying metric graph of $\mathfrak{C}'$ is $\Gamma'$, the associated curves of $\mathfrak{C}'$ and $\mathfrak{C}$ at $p$ are identical for all $p\in\Gamma'$, and the reduction maps of $\mathfrak{C}'$ and $\mathfrak{C}$ at $p$  restricted to $\Tan_{\Gamma'}(p)$  are identical for all $p\in\Gamma'$ (note that $\Tan_{\Gamma'}(p) \subsetneqq\Tan_\Gamma(p)$ when $p$ is a boundary point of $\Gamma'$ in  $\Gamma$). Moreover, a diagrammatic pre-limit $g^1_d$ represented by $(\mathcal{D}',\mathcal{H}')$ on $\mathfrak{C}'$ is said to be the restriction of  $(\mathcal{D},\mathcal{H})$ to  $\Gamma'$ (or $\mathfrak{C}'$) if the following are satisfied:
 
\begin{enumerate}
\item $H'_p$ (the $C_p$-part of $\mathcal{H}'$) is identical to  $H_p$ (the $C_p$-part of $\mathcal{H}$) for all $p\in\Gamma'$. 
\item  $D'_p$ (the $C_p$-part of $\mathcal{D}'$) is identical to $D_p$ (the $C_p$-part of $\mathcal{D}$) for all $p\in\Gamma'\setminus \partial \Gamma'$ (here $\partial \Gamma'$ stands for the boundary  of $\Gamma'$ in $\Gamma$). 
\item For all $p\in\partial \Gamma'$, $D'_p$ is modified from $D_p$  as $D_p=D'_p+\Sigma_{t\in\In_\Gamma(p)\setminus\In_{\Gamma'}(p)} (-(m(p,t))(\red_p(t))$ where $m(p,t)$ is the multiplicity of $t$ (which is negative if $t\in\In_\Gamma(p)$) in the local diagram induced by $\mathcal{H}$ at $p$. Note that this  modification guarantees the compatibility between $D_p$ and $H_p$ which further implies that $(\mathcal{D}',\mathcal{H}')$ is diagrammatic (Definition~\ref{D:DiaSolvDH}). 
\end{enumerate}

\begin{proposition}{\label{P:DHglue}}
Let $\Gamma_1$ and $\Gamma_2$ be connected metric subgraphs of a metric graph $\Gamma$ such that $\Gamma=\Gamma_1\cup\Gamma_2$ and $\Gamma_1\cap\Gamma_2$ is a singleton. For a saturated metrized complex  $\mathfrak{C}$  whose underlying metric graph is $\Gamma$, let
$(\mathcal{D},\mathcal{H})$ represents a diagrammatic pre-limit $g^1_d$ on $\mathfrak{C}$. Let  $(\mathcal{D}_1,\mathcal{H}_1)$ and  $(\mathcal{D}_2,\mathcal{H}_2)$ be the restrictions of $(\mathcal{D},\mathcal{H})$  to $\Gamma_1$ and $\Gamma_2$ respectively. Then $(\mathcal{D},\mathcal{H})$ is smoothable if and only if $(\mathcal{D}_1,\mathcal{H}_1)$ and  $(\mathcal{D}_2,\mathcal{H}_2)$ are both smoothable. 
\end{proposition}
\begin{proof}
Let $\Gamma_1\cap\Gamma_2=\{q\} $. 
Let $\mathfrak{C}_1$ and  $\mathfrak{C}_2$ be the restrictions of $\mathfrak{C}$ to $\Gamma_1$ and $\Gamma_2$ respectively. Then $(\mathcal{D}_1,\mathcal{H}_1)$ is on  $\mathfrak{C}_1$ and $(\mathcal{D}_2,\mathcal{H}_2)$ is on  $\mathfrak{C}_2$. It follows easily from the smoothing criterion that $(\mathcal{D},\mathcal{H})$ is smoothable implies that  $(\mathcal{D}_1,\mathcal{H}_1)$ and  $(\mathcal{D}_2,\mathcal{H}_2)$ are both smoothable. 

Now suppose that $(\mathcal{D}_1,\mathcal{H}_1)$ and  $(\mathcal{D}_2,\mathcal{H}_2)$ are both smoothable and we claim that $(\mathcal{D},\mathcal{H})$ is smoothable. This means $(\mathcal{D}_1,\mathcal{H}_1)$ and  $(\mathcal{D}_2,\mathcal{H}_2)$ are solvable,  and since $\Gamma=\Gamma_1\cup\Gamma_2$ and $\Gamma_1\cap\Gamma_2=\{q\}$, $(\mathcal{D},\mathcal{H})$ must be  solvable. Thus we may assume $\rho$ is a solution to the global diagram of $(\mathcal{D},\mathcal{H})$, while the restriction of $\rho$ to $\Gamma_1$ (respectively $\Gamma_2$), denoted by $\rho_1$ (respectively $\rho_2$),  is a solution to the global diagram of $(\mathcal{D}_1,\mathcal{H}_1)$ (respectively $(\mathcal{D}_2,\mathcal{H}_2)$). Let $\mathcal{B}$, $\mathcal{B}_1$ and $\mathcal{B}_2$ be the bifurcation trees with respect to $\rho$, $\rho_1$ and $\rho_2$ respectively. 


By the smoothing criterion, $\mathcal{H}_1$ contains an admissible collection  $\{f^{(1)}_p\}_{p \in \Gamma_1}$ of rational functions $f^{(1)}_p\in H_p$ and $\mathcal{H}_2$ contains an admissible collection $\{f^{(2)}_p\}_{p \in \Gamma_2}$ of rational functions $f^{(2)}_p\in H_p$. 

Note that $\mathcal{B}_1$ and $\mathcal{B}_2$ are subtrees of $\mathcal{B}$ and $\mathcal{B}_1\cup\mathcal{B}_2=\mathcal{B}$. Let $r(\mathcal{B})$, $r(\mathcal{B}_1)$ and $r(\mathcal{B}_2)$ be the roots of $\mathcal{B}$, $\mathcal{B}_1$ and $\mathcal{B}_2$ respectively. Clearly $r(\mathcal{B})$ must be either $r(\mathcal{B}_1)$ or $r(\mathcal{B}_2)$. Without loss of generality, we assume $r(\mathcal{B})=r(\mathcal{B}_1)$.  Let $y=\pi_\mathcal{B}(q)$ and $L$ be the closed segment connecting $r(\mathcal{B}_2)$ and $y$ in $\mathcal{B}$. Then one can observe that $\mathcal{B}_1\cap\mathcal{B}_2=L$. We construct a desirable admissible  $\{f_p\}_{p \in \Gamma}\in \mathcal{H}$ by clutching $\{f^{(1)}_p\}_{p \in \Gamma_1}$ and $\{f^{(2)}_p\}_{p \in \Gamma_2}$ as follows:

\begin{enumerate}
\item For $x\in \mathcal{B}_1 \setminus L$, let $\vec{P}_x=\vec{P}^{(1)}_x$.
\item For $x\in \mathcal{B}_2 \setminus L$, let $\vec{P}_x=\vec{P}^{(2)}_x$.
\item For $p\in\Gamma_1\setminus\pi_\mathcal{B}^{-1}(L)$, let $f_p=f^{(1)}_p$.
\item For $p\in\Gamma_2\setminus\pi_\mathcal{B}^{-1}(L)$, let $f_p=f^{(2)}_p$.
\item For $x\in L$, consider all the forward tangent directions in $\Tan^+_{\mathcal{B}_1}(x)=\{t^{(1)}_1,\cdots,t^{(1)}_k\}$ and all the forward tangent directions  in $\Tan^+_{\mathcal{B}_2}(x)=\{t^{(2)}_1,\cdots,t^{(2)}_l\}$. Then by the smoothing criterion on  $(\mathcal{D}_1,\mathcal{H}_1)$ and  $(\mathcal{D}_2,\mathcal{H}_2)$, we can assign values $c^{(1)}_1,\cdots,c^{(1)}_k\in \kappa$ to $t^{(1)}_1,\cdots,t^{(1)}_k$ respectively and values  $c^{(2)}_1,\cdots,c^{(2)}\in\kappa$ to $t^{(2)}_1,\cdots,t^{(2)}_l$ respectively such that  
\begin{enumerate}
\item $f^{(1)}_p(\red_p(t))=c^{(1)}_i$ for all $i=1,\cdots,k$ whenever $p\in\pi_{\mathcal{B}_1}^{-1}(x)$ and  $t\in\Tan^{\rho_1+}_{\Gamma_1}(p)\bigcap\pi_{\mathcal{B}_1*}^{-1}(t^{(1)}_i)$, and
\item $f^{(2)}_p(\red_p(t))=c^{(2)}_j$  for all $j=1,\cdots,l$ whenever $p\in\pi_{\mathcal{B}_2}^{-1}(x)$ and $t\in\Tan^{\rho_2+}_{\Gamma_2}(p)\bigcap\pi_{\mathcal{B}_2*}^{-1}(t^{(2)}_j)$. 
\end{enumerate}

 When $x\in L\setminus\{y\}$, we have $\pi_\mathcal{B}^{-1}(x)=\pi_{\mathcal{B}_1}^{-1}(x)\cup\pi_{\mathcal{B}_1}^{-1}(x)$, $\pi_{\mathcal{B}_1}^{-1}(x)\cap\pi_{\mathcal{B}_1}^{-1}(x)=\emptyset$, $\Tan^+_\mathcal{B}(x)=\Tan^+_{\mathcal{B}_1}(x)\cup \Tan^+_{\mathcal{B}_2}(x)$ and $\Tan^+_{\mathcal{B}_1}(x)\cap \Tan^+_{\mathcal{B}_2}(x)$ is a singleton. Without loss of generality, we let   $t_1=t^{(1)}_1=t^{(2)}_1$ be the forward tangent direction common to both $\Tan^+_{\mathcal{B}_1}(x)$ and $\Tan^+_{\mathcal{B}_2}(x)$, which means that 
 $\vec{\iota}_{\mathcal{B}_1}(t_1)$ and $\vec{\iota}_{\mathcal{B}_2}(t_1)$ are the open superlevel components (for $\Gamma_1$ and $\Gamma_2$ respectively) containing  $q$ and  $\vec{\iota}_\mathcal{B}(t_1)=\vec{\iota}_{\mathcal{B}_1}(t_1)\cup  \vec{\iota}_{\mathcal{B}_2}(t_1)$.  So we can let $f_p=f^{(1)}_p-c^{(1)}_1$ for all $p\in\pi_{\mathcal{B}_1}^{-1}(x)$ and $f_p=f^{(2)}_p-c^{(2)}_1$ for all $p\in\pi_{\mathcal{B}_2}^{-1}(x)$. In this way of clutching, we conclude:
 \begin{enumerate}
\item $f_p(\red_p(t))=c^{(1)}_i-c^{(1)}_1$ for all $i=1,\cdots,k$ whenever $p\in\pi_{\mathcal{B}_1}^{-1}(x)$ and  $t\in\Tan^{\rho_1+}_{\Gamma_1}(p)\bigcap\pi_{\mathcal{B}_1*}^{-1}(t^{(1)}_i)$, 
\item $f_p(\red_p(t))=c^{(2)}_j-c^{(2)}_1$  for all $j=1,\cdots,l$ whenever $p\in\pi_{\mathcal{B}_2}^{-1}(x)$ and $t\in\Tan^{\rho_2+}_{\Gamma_2}(p)\bigcap\pi_{\mathcal{B}_2*}^{-1}(t^{(2)}_j)$, and 
\item in particular, by (a) and (b), $f_p(\red_p(t))=0$ whenever $p\in\pi_\mathcal{B}^{-1}(x)$ and $t\in\Tan^{\rho+}_\Gamma(p)\bigcap\pi_{\mathcal{B}*}^{-1}(t_1)$. (Note that $\pi_{\mathcal{B}*}^{-1}(t_1) =\pi_{\mathcal{B}_1*}^{-1}(t^{(1)}_1)\cup \pi_{\mathcal{B}_2*}^{-1}(t^{(2)}_1)$.)
 \end{enumerate}
 
 When $x=y$, we have $\pi_\mathcal{B}^{-1}(x)=\pi_{\mathcal{B}_1}^{-1}(x)\cup\pi_{\mathcal{B}_1}^{-1}(x)$, $\pi_{\mathcal{B}_1}^{-1}(x)\cap\pi_{\mathcal{B}_1}^{-1}(x)=\{q\}$, $\Tan^+_\mathcal{B}(x)=\Tan^+_{\mathcal{B}_1}(x)\cup \Tan^+_{\mathcal{B}_2}(x)$ and $\Tan^+_{\mathcal{B}_1}(x)\cap \Tan^+_{\mathcal{B}_2}(x)=\emptyset$. Since $f^{(1)}_q, f^{(2)}_q\in H_q$, we must have $f^{(2)}_q=\alpha+\beta f^{(1)}_q$ for some $\alpha,\beta\in\kappa$. So we can let $f_p=\alpha+\beta f^{(1)}_p$ for all $p\in\pi_{\mathcal{B}_1}^{-1}(x)$ and $f_p=f^{(2)}_p$ for all $p\in\pi_{\mathcal{B}_2}^{-1}(x)$. In this way of clutching, we conclude:
 \begin{enumerate}
\item $f_p(\red_p(t))=\alpha+\beta c^{(1)}_i$ for all $i=1,\cdots,k$ whenever $p\in\pi_{\mathcal{B}_1}^{-1}(x)$ and  $t\in\Tan^{\rho_1+}_{\Gamma_1}(p)\bigcap\pi_{\mathcal{B}_1*}^{-1}(t^{(1)}_i)$, 
\item $f_p(\red_p(t))=c^{(2)}_j$  for all $j=1,\cdots,l$ whenever $p\in\pi_{\mathcal{B}_2}^{-1}(x)$ and $t\in\Tan^{\rho_2+}_{\Gamma_2}(p)\bigcap\pi_{\mathcal{B}_2*}^{-1}(t^{(2)}_j)$, and 
\item in particular, (a) and (b) coincide at $p=q$ by our assumption as $f_q(\red_q(t))=\alpha+\beta f^{(1)}_q(\red_q(t))=f^{(2)}_q(\red_q(t))$ for all $t\in\Tan^{\rho+}_\Gamma(q)=\Tan^{\rho_1+}_{\Gamma_1}(q)\cup\Tan^{\rho_2+}_{\Gamma_2}(q)$.
 \end{enumerate}
 
\end{enumerate}

The above construction guarantees that $\{f_p\}_{p \in \Gamma}$  is admissible in $\mathcal{H}$. Therefore, $(\mathcal{D},\mathcal{H})$  is smoothable. 
  
\end{proof}

\begin{example}
For a smootbable pre-limit $g^1_d$ represented by $(\mathcal{D},\mathcal{H})$, it is possible that we can construct different smoothings, or equivalently different pseudo-harmonic morphisms (see Section~\ref{S:harmmor} for a precise definition) from the saturated metrized complex to a genus $0$ metrized complex. Figure~\ref{F:CompactType} is an example for a case of a saturated metrized complex of compact type whose underlying metric graph is a segment. Fix a rational function $f_1\in H_{p_1}$. Suppose the value of $f_1$ on the marked point corresponding to $p_1p_0$ is $c_0$ and on the marked point corresponding to $p_1p_3$ is $c_3$.  Then the rational functions in $H_{p_2}$ which takes the value $c_0$ at the marked point corresponding to $p_2p_0$ form a one-dimensional subspace $H$ of $H_{p_2}$. Let $f_2\in H$ take value $c_3$ on the marked point corresponding to $p_2p_4$ and $f'_2\in H$ take a value other than $c_3$ on the marked point corresponding to $p_2p_4$. Then $f_1$ and $f_2$ can be used to construct a pseudo-harmonic morphism as in Figure~\ref{F:CompactType}(a), while $f_1$ and $f'_2$ can be used to construct a pseudo-harmonic morphism as in Figure~\ref{F:CompactType}(b).

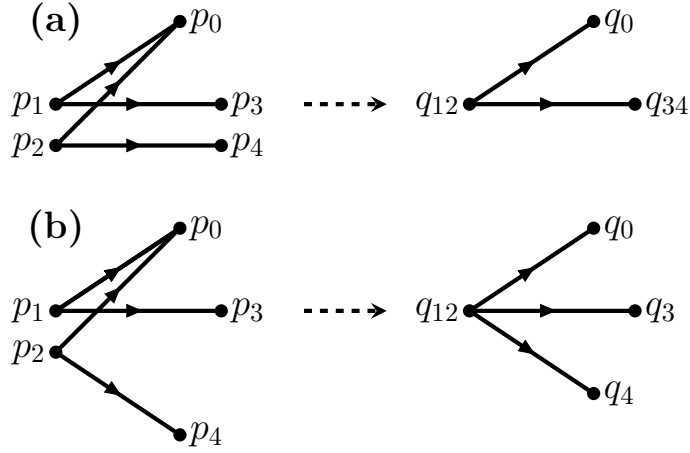
\begin{figure}[tbp]
\centering
\begin{tikzpicture}[>=to,x=1.1cm,y=1.1cm]
\draw (0,5.5) node {\Large \textbf{(a)}};
\draw (0,3) node {\Large \textbf{(b)}};

\coordinate (a_p0) at (1.5,5.5);
    \fill [black] (a_p0) circle (2.5pt);
    \draw (a_p0) node[anchor=west] {\Large $p_0$};
\coordinate (a_p1) at (0,4.5);
    \fill [black] (a_p1) circle (2.5pt);
    \draw (a_p1) node[anchor=east] {\Large $p_1$};
\coordinate (a_p2) at (0,4);
    \fill [black] (a_p2) circle (2.5pt);
    \draw (a_p2) node[anchor=east] {\Large $p_2$};
\coordinate (a_p3) at (2,4.5);
    \fill [black] (a_p3) circle (2.5pt);
    \draw (a_p3) node[anchor=west] {\Large $p_3$};
\coordinate (a_p4) at (2,4);
    \fill [black] (a_p4) circle (2.5pt);
    \draw (a_p4) node[anchor=west] {\Large $p_4$};
\coordinate (a_q0) at (6.5,5.5);
    \fill [black] (a_q0) circle (2.5pt);
    \draw (a_q0) node[anchor=west] {\Large $q_0$};
\coordinate (a_q12) at (5,4.5);
    \fill [black] (a_q12) circle (2.5pt);
    \draw (a_q12) node[anchor=east] {\Large $q_{12}$};
\coordinate (a_q34) at (7,4.5);
    \fill [black] (a_q34) circle (2.5pt);
    \draw (a_q34) node[anchor=west] {\Large $q_{34}$};

\coordinate (b_p0) at (1.5,3);
    \fill [black] (b_p0) circle (2.5pt);
    \draw (b_p0) node[anchor=west] {\Large $p_0$};
\coordinate (b_p1) at (0,2);
    \fill [black] (b_p1) circle (2.5pt);
    \draw (b_p1) node[anchor=east] {\Large $p_1$};
\coordinate (b_p2) at (0,1.5);
    \fill [black] (b_p2) circle (2.5pt);
    \draw (b_p2) node[anchor=east] {\Large $p_2$};
\coordinate (b_p3) at (2,2);
    \fill [black] (b_p3) circle (2.5pt);
    \draw (b_p3) node[anchor=west] {\Large $p_3$};
\coordinate (b_p4) at (1.5,0.5);
    \fill [black] (b_p4) circle (2.5pt);
    \draw (b_p4) node[anchor=west] {\Large $p_4$};
\coordinate (b_q0) at (6.5,3);
    \fill [black] (b_q0) circle (2.5pt);
    \draw (b_q0) node[anchor=west] {\Large $q_0$};
\coordinate (b_q12) at (5,2);
    \fill [black] (b_q12) circle (2.5pt);
    \draw (b_q12) node[anchor=east] {\Large $q_{12}$};
\coordinate (b_q3) at (7,2);
    \fill [black] (b_q3) circle (2.5pt);
    \draw (b_q3) node[anchor=west] {\Large $q_3$};
\coordinate (b_q4) at (6.5,1);
    \fill [black] (b_q4) circle (2.5pt);
    \draw (b_q4) node[anchor=west] {\Large $q_4$};

\begin{scope}[line width=1.6pt, every node/.style={sloped,allow upside down}]
  \draw (a_p1) -- node {\midarrow} (a_p0);
  \draw (a_p2) -- node {\midarrow} (a_p0);
  \draw (a_p1) -- node {\midarrow} (a_p3);
  \draw (a_p2) -- node {\midarrow} (a_p4);
  \draw (a_q12) -- node {\midarrow} (a_q0);
  \draw (a_q12) -- node {\midarrow} (a_q34);
  \draw (b_p1) -- node {\midarrow} (b_p0);
  \draw (b_p2) -- node {\midarrow} (b_p0);
  \draw (b_p1) -- node {\midarrow} (b_p3);
  \draw (b_p2) -- node {\midarrow} (b_p4);
  \draw (b_p1) -- node {\midarrow} (b_p0);
  \draw (b_q12) -- node {\midarrow} (b_q0);
  \draw (b_q12) -- node {\midarrow} (b_q3);
  \draw (b_q12) -- node {\midarrow} (b_q4);

  \draw [-stealth, dashed] (3,2) -- (4,2);
  \draw [-stealth, dashed] (3,4.5) -- (4,4.5);
\end{scope}
\end{tikzpicture}
\caption{Examples of pseudo-harmonic morphisms (see Section~\ref{S:harmmor} for a precise definition) that can be derived from the same $(\mathcal{D},\mathcal{H})$ on a saturated metrized complex of compact type.}\label{F:CompactType}
\end{figure}
\end{example}

\subsection{Saturated Metrized Complexes with Genus-$g$ Underlying Metric Graphs containing $g$ Separate Loops} \label{S:GenusOne}
For a generalization of saturated metrized complexes of compact type, we consider a saturated metrized complex $\mathfrak{C}$ whose underlying metric graph $\Gamma$ has genus $g$ and contains $g$ separate loops $\Omega_1,\cdots,\Omega_g$ (see Figure~\ref{F:GenusG} for such a metric graph of genus $6$).  Here $\Omega_i$ and  $\Omega_j$ are separate if the intersection of  $\Omega_i$ and  $\Omega_j$ is either empty or just a singleton.  By the smoothing criterion, one prerequisite for  being smoothable  is solvability, i.e., the integration along each $\Omega_i$ for $i=1,\cdots,g$ with respect to the global diagram induced by $(\mathcal{D},\mathcal{H})$ is $0$. We let $\rho$ be a solution to the global diagram and $\mathcal{B}$ be the corresponding bifurcation tree. Consider a loop $\Omega\in\{\Omega_1,\cdots,\Omega_g\}$. We let $\Omega_{\min}(\rho)$ be the set of points where $\rho$ restricted to $\Omega_i$ achieves minimum. Then $\Omega_{\min}(\rho)$ is a finite set with at least one element and for each point $p\in\Omega_{\min}(\rho)$, there are exact two forward tangent directions in $\Tan^{\rho+}_\Omega(p)$ (forward tangent directions restricted to $\Omega$). We say $p$ is a closing point if these two tangent directions are locally equivalent in the local diagram at $p$ induced by $\mathcal{H}$, and say $p$ is an opening point otherwise. Denote the set of closing points in $\Omega_{\min}(\rho)$ by $\Omega^c_{\min}(\rho)$, and the set of opening points in $\Omega_{i,\min}(\rho)$ by $\Omega^o_{\min}(\rho)$. The following theorem says that to verify whether $(\mathcal{D},\mathcal{H})$ is smoothable can be reduced to a purely combinatorial point-counting problem of  $\Omega^c_{\min}(\rho)$ and  $\Omega^o_{\min}(\rho)$.

\begin{figure}[tbp]
\centering
\begin{tikzpicture}[line cap=round,line join=round,>=to,x=.7cm,y=.7cm]
\draw [line width=1.4pt] (2.5,3.) ellipse (1.5 and 1.5);
\draw [line width=1.4pt] (5.,3.) ellipse (1 and 1);
\draw [line width=1.4pt] (10.,3.) ellipse (1 and 1);
\draw [line width=1.4pt] (10.,1.) ellipse (.5 and .5);
\draw [line width=1.4pt] (12.,4.) ellipse (.5 and .5);
\draw [line width=1.4pt] (7.5,0.) ellipse (1 and 1);

\fill  (4,3) circle (2.5pt);
\fill  (6,3) circle (2.5pt);
\fill  (9,3) circle (2.5pt);
\fill  (10,4) circle (2.5pt);
\fill  (9.5,1) circle (2.5pt);
\fill  (7.5,1) circle (2.5pt);
\fill  (7.5,3) circle (2.5pt);
\fill  (7.5,4) circle (2.5pt);
\fill  (7,4.5) circle (2.5pt);
\fill  (8,4.5) circle (2.5pt);
\fill  (11.5,4) circle (2.5pt);
\fill  (5,0) circle (2.5pt);

\draw (2.5,1) node {\LARGE $\Omega_1$};
\draw (5,1.5) node {\LARGE $\Omega_2$};
\draw (11,3) node[anchor=west] {\LARGE $\Omega_3$};
\draw (12.5,4) node[anchor=west] {\LARGE $\Omega_4$};
\draw (8.5,0) node[anchor=west] {\LARGE $\Omega_5$};
\draw (10.5,1) node[anchor=west] {\LARGE $\Omega_6$};

\draw [line width=1.4pt] (6,3)-- (9,3);
\draw [line width=1.4pt] (7.5,1)-- (7.5,4);
\draw [line width=1.4pt] (7,4.5)-- (7.5,4);
\draw [line width=1.4pt] (8,4.5)-- (7.5,4);
\draw [line width=1.4pt] (9.5,1)-- (7.5,1);
\draw [line width=1.4pt] (11.5,4)-- (10,4);
\draw [line width=1.4pt] (5,0)-- (6.5,0);

\end{tikzpicture}

\caption{A genus-$6$ metric graph containing $6$ separate loops $\Omega_1,\cdots,\Omega_6$. All the loops are disjoint with each other except that $\Omega_1$ and $\Omega_2$ intersect at a single point.}\label{F:GenusG}
\end{figure}
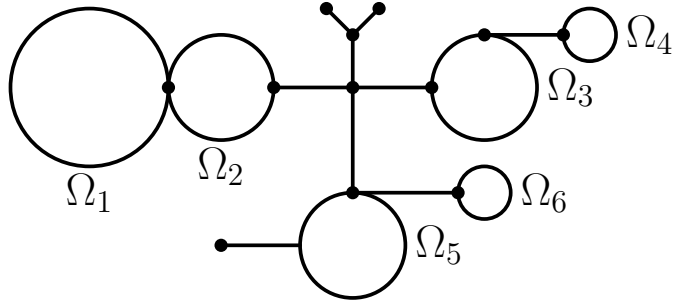

\begin{theorem}
Let $\mathfrak{C}$ be a saturated metrized complex  whose underlying metric graph $\Gamma$ has genus $g$ and contains $g$ separate loops $\Omega_1,\cdots,\Omega_g$. Let $(\mathcal{D},\mathcal{H})$ represent a solvable pre-limit $g^1_d$ on $\mathfrak{C}$ with a solution $\rho$. Then $(\mathcal{D},\mathcal{H})$ is smoothable if and only if  the following are satisfied on each loop $\Omega\in\{\Omega_1,\cdots,\Omega_g\}$:
\begin{enumerate}
\item In case that $\Omega_{\min}(\rho)$ is a singleton, the unique point $p\in\Omega_{\min}(\rho)$ is a closing point;
\item In case that the cardinality of $\Omega_{\min}(\rho)$ is at least $2$, either $\Omega^c_{\min}(\rho)=\emptyset$ or $\Omega^c_{\min}(\rho)$ has the same parity as $\Omega_{\min}(\rho)$.
\end{enumerate}
\end{theorem}

\begin{proof}
By Proposition~\ref{P:DHglue}, we can subdivide $\Gamma$ into segments and loops and test whether $(\mathcal{D},\mathcal{H})$ restricted to each segment or loop is smoothable. The case of segments can dealt with analogously as for compact-type saturated metrized complexes, and it only needs to show the case for $(\mathcal{D},\mathcal{H})$ restricted to a single loop $\Omega$. Without loss of generality, we may assume $\Gamma=\Omega$ and apply intrinsic global compatibility conditions to this specific graph. . 

Let $\rho$ be a solution to $(\mathcal{D},\mathcal{H})$ and $\mathcal{B}$ be the corresponding bifurcation tree. We need to test the possibility of constructing an admissible $\{f_p\}_{p \in \Gamma}\in\mathcal{H}$. Actually we will consider cases for all $p\in\pi_\mathcal{B}^{-1}(x)$ for each $x\in\mathcal{B}$ separately. 

First we consider the root $y$ of $\Omega$. Note that the closed superlevel component $\iota_\mathcal{B}(y)$ corresponding to  $y$ and  $\Omega_{\min}(\rho)=\pi_\mathcal{B}^{-1}(y)$.

\begin{enumerate}
\item When $\Omega_{\min}(\rho)=\{p\}$, there is only one open superlevel component with the boundary point $p$. Therefore, to pass the compatibility test, the two tangent directions in  $\Tan^{\rho+}_\Omega(p)$ must be locally equivalent, which means $p$ must be a closing point.

\item When $\Omega_{\min}(\rho)=\{p_1,\cdots, p_k\}$ with $k\geqslant 2$, there are exactly $k$ open superlevel components $\beta_1,\cdots,\beta_k$ with the boundary points in $\Omega_{\min}(\rho)$ (the $k$ open edges in $\Omega$ with end points in $\Omega_{\min}(\rho)$). We let the end points of $\beta_i$ be $p_i$ and $p_{i+1}$ for $i=1,\cdots,k-1$ and the end points of $\beta_k$ be $p_k$ and $p_1$.

First note that for any point $\Omega_{\min}(\rho)$, if the two tangent directions in $\Tan^{\rho+}_\Omega(p)$ are not locally equivalent, then for any two arbitrarily chosen distinct values, we can always find a rational function $f_p\in H_p$ taking these values respectively at the two reduction points in $C_p$ corresponding these two tangent directions (Lemma~\ref{L:tune}). Assigning values $c_1,\cdots, c_k\in\kappa$ to $\beta_1,\cdots,\beta_k$ respectively, we have the following cases:
\begin{enumerate}
\item If $c_1,\cdots, c_k$ are all distinct, then to pass the compatibility test, it is equivalent to say that the two tangent directions in  $\Tan^{\rho+}_\Omega(p_i)$ for each $i=1,\cdots,k$ must not be locally equivalent, i.e., $p_1,\cdots,p_k$ are all opening points.
\item If at least one point in $\Omega_{\min}(\rho)$ is a closing point, we may assume suppose $p_k$ is a closing point without loss of generality. Then to pass the compatibility test, $c_k$ must be equal to $c_{k-1}$ and the whole case reduces to dropping $p_k$ from $\Omega_{\min}(\rho)$ and assigning values $c_1,\cdots, c_{k-1}$ to $\beta_1,\cdots,\beta_{k-1}$.
\end{enumerate}
\end{enumerate}
It is straightforward to verify that the above arguments afford the conditions (1) and (2) stated in the theorem.

It remains to show that the compatibility test can always get passed for those points $x$ in $\mathcal{B}\setminus\{y\}$. 
Note that  each forward tangent direction $t\in\Tan^+_\mathcal{B}(x)$ corresponds to an open superlevel component $\vec{\iota}_\mathcal{B}(t)$ of $\rho$. If $\iota_\mathcal{B}(x)$ is a singleton $\{p\}$, then $\Tan^{\rho+}_\mathcal{B}(p)=\Tan^+_\mathcal{B}(x)=\emptyset$ and we simply let $f_p$ be any non-constant function in $H_p$. Otherwise, if $\partial\vec{\iota}_\mathcal{B}(t)$ is the set of boundary points of $\vec{\iota}_\mathcal{B}(t)$, we must have $\bigcup_{t\in\Tan^+_\mathcal{B}(x)}\partial\vec{\iota}_\mathcal{B}(t) = \pi_\mathcal{B}^{-1}(x)$, which is assumed in the following discussion.

Note that for $x\in \mathcal{B}\setminus\{y\}$, $\Tan^+_\mathcal{B}(x)$ is either empty or satisfies the ``ordering'' property: starting from an arbitrary $t_1\in\Tan^+_\mathcal{B}(x)$, there exists an ordering $t_1,t_2,\cdots,t_k$ of all the tangent directions in $\Tan^+_\mathcal{B}(x)$ such that $(\bigcup^{i-1}_{j=1}\partial\vec{\iota}_\mathcal{B}(t_j)) \bigcap \partial\vec{\iota}_\mathcal{B}(t_i)$ is a singleton $\{q_i\}$ for $i=2,\cdots,k$. In other words, we can rebuild the closed superlevel component $\iota_\mathcal{B}(x)$ from the open superlevel components $\vec{\iota}_\mathcal{B}(t)$ with $t\in\Tan^+_\mathcal{B}(x)$ one by one by choosing the attaching point from $\pi_\mathcal{B}^{-1}(x)$ properly.

Note that at each $q_i$, there exists two forward tangent directions in $\Tan^{\rho+}_\Gamma(q_i)$. On the other hand, if a point $p$ in $\pi_\mathcal{B}^{-1}(x)$ is not any of the $q_i$'s, then there is exactly one forward tangent direction in $\Tan^{\rho+}_\Gamma(p)$.

We use the following procedure to assign values to the tangent directions in $\Tan^+_\mathcal{B}(x)$ and find compatible rational functions $f_p\in H_p$ for all points $p\in\pi_\mathcal{B}^{-1}(x)$:

\begin{enumerate}
\item We assign an arbitrary value $c_1\in\kappa$ to $t_1$. For each $p\in\partial\vec{\iota}_\mathcal{B}(t_1)$, we are able to find a nonconstant rational function $f_p\in H_p$ such that $f_p(\red_p(t_{1,p})) = c_1$ where $t_{p,1}$ is the unique tangent in $\Tan^{\rho+}_\Gamma(p)$ such that $\pi_{\mathcal{B}*}(t_{1,p})=t_1$.

\item Now suppose we have already assigned values $c_1, \cdots, c_i \in\kappa$ to $t_1,\cdots,t_i$ respectively and found rational functions $f_p\in H_p$ for all $p\in \bigcup^{i}_{j=1}\partial\vec{\iota}_\mathcal{B}(t_j)$ such that for $j=1,\cdots, i$, we have $f_p(\red_{p}(t))=c_j$ as long as $t\in\Tan^{\rho+}_\Gamma(p)\bigcap\pi_{\mathcal{B}*}^{-1}(t_j)$.
    Note that $q_{i+1}$ is the unique element in both $\bigcup^i_{j=1}\partial\vec{\iota}_\mathcal{B}(t_j)$ and $\partial\vec{\iota}_\mathcal{B}(t_{i+1})$. Let $t_{q_{i+1}}$ be the unique tangent direction in $\Tan^{\rho+}_\Gamma(q_{i+1})$ such that $\pi_{\mathcal{B}*}(t_{q_{i+1}})=t_{i+1}$. We let $c_{i+1}=f_{q_{i+1}}(\red_{q_{i+1}}(t_{q_{i+1}}))$, and for each $p\in \partial\vec{\iota}_\mathcal{B}(t_{i+1})\setminus\{t_{q_{i+1}}\}$, we let $f_p$ be a nonconstant rational function in $H_p$ such that $f_p(\red_p(t))=c_{i+1}$ where $t$ is the unique tangent direction in $\Tan^{\rho+}_\Gamma(p)\bigcap\pi_{\mathcal{B}*}^{-1}(t_{i+1})$.
\end{enumerate}
In this way, we derive a sequence of $c_1, \cdots, c_k \in\kappa$ and a family of rational functions $\{f_p\}_{p\in \pi_\mathcal{B}^{-1}(x)}$. By our construction, for $j=1,\cdots, k$ and all $p\in\pi_\mathcal{B}^{-1}(x)$, we have $f_p(\red_{p}(t))=c_j$ as long as $t\in\Tan^{\rho+}_\Gamma(p)\bigcap\pi_{\mathcal{B}*}^{-1}(t_j)$. Hence, the compatibility test get passed at all $x\in\mathcal{B}\setminus\{y\}$. 
\end{proof}

\begin{remark}
Here are a few cases of the conditions in the above theorem: if $\#\Omega_{\min}(\rho)=1$, then $\#\Omega^c_{\min}(\rho)$ can only be $1$; if $\#\Omega_{\min}(\rho)=2$, then $\#\Omega^c_{\min}(\rho)$ can only be $0$ or $2$; if $\#\Omega_{\min}(\rho)=3$, then $\#\Omega^c_{\min}(\rho)$ can only be $0$, $1$ or $3$; if $\#\Omega_{\min}(\rho)=4$, then $\#\Omega^c_{\min}(\rho)$ can only be $0$, $2$ or $4$.
\end{remark}

\subsection{Saturated Metrized Complexes of Harris-Mumford Type}\label{S:harrismumapplication}

Here we study saturated metrized complexes arising from the construction of Harris and Mumford in Theorem~5 of \cite{HM82}.  These correspond to two types of saturated metrized complexes.

A \emph{Harris-Mumford saturated metrized complex of type I} is a saturated metrized complex $\mathfrak{C}$ with underlying metric graph $\Gamma$ homeomorphic to a topological bouquet graph obtained by gluing together finitely many ($0$ is allowed) circles along a single vertex $o$ (Figure~\ref{F:HarrisMumford}) while $g(\mathfrak{C})=g(C_o)+g(\Gamma)$. We call this central vertex $o$ the \emph{eye vertex} and $C_o$ the \emph{eye} of $\mathfrak{C}$. We also call the middle points of the attached circles \emph{pedal vertices}. In particular, we have $g(\Gamma)$ pedal vertices. We call the pair of marked points on the eye corresponding to the two edges connecting $o$ and $p_i$ the \emph{$p_i$-marked points} on the eye.

A \emph{Harris-Mumford saturated metrized complex of type II} is a saturated metrized complex $\mathfrak{C}$ with underlying metric graph $\Gamma$ homeomorphic to a topological graph obtained by gluing together finitely many ($0$ is allowed) circles along the endpoints $o_1$ and $o_2$ of a line segment (Figure~\ref{F:HarrisMumford}) while $g(\mathfrak{C})=g(C_{o_1})+g(C_{o_2})+g(\Gamma)$. We call the two vertices $o_1$ and $o_2$ the \emph{eye vertices} and the associated curves $C_{o_1}$ and $C_{o_2}$ the \emph{eyes} of $\mathfrak{C}$. Analogously, the middle points of the attached circles are called \emph{pedal vertices}. We also call the pair of marked points on the eyes corresponding to the two edges connecting the segment $o_1o_2$ and $p_i$ the \emph{\emph{$p_i$-marked points}}.

For both types, we say a global diagram on $\Gamma$ is \emph{regulated} if edge multiplicities only possibly change across the eye vertices and the pedal vertices (Figure~\ref{F:HarrisMumford}). The following theorem is an analogue of \cite[Theorem 5]{HM82} for Harris-Mumford saturated metrized complexes.

\begin{figure}[tbp]
\centering
\begin{tikzpicture}[>=to,x=2.2cm,y=2.2cm]
\def\AnglePone{210}
\def\LengthPone{0.8}
\def\AnglePtwo{90}
\def\LengthPtwo{1.2}
\def\AnglePthree{-30}
\def\LengthPthree{1}

\def\AnglePPone{220}
\def\LengthPPone{1}
\def\AnglePPtwo{120}
\def\LengthPPtwo{1.2}
\def\AnglePPthree{90}
\def\LengthPPthree{1.2}
\def\AnglePPfour{0}
\def\LengthPPfour{1}
\def\AnglePPfive{-100}
\def\LengthPPfive{0.8}

\coordinate (o) at (-1.5,0);
    \fill [black] (o) circle (2.5pt);
    \draw (o) node[anchor=south west] {\Large $o$};
\coordinate (p1) at ($(o)+\LengthPone*({cos(\AnglePone)},{sin(\AnglePone)})$);
    \fill [black] (p1) circle (2.5pt);
    \draw (p1) node[anchor=east] {\Large $p_1$};
\coordinate (p2) at ($(o)+\LengthPtwo*({cos(\AnglePtwo)},{sin(\AnglePtwo)})$);
    \fill [black] (p2) circle (2.5pt);
    \draw (p2) node[anchor=south] {\Large $p_2$};
\coordinate (p3) at ($(o)+\LengthPthree*({cos(\AnglePthree)},{sin(\AnglePthree)})$);
    \fill [black] (p3) circle (2.5pt);
    \draw (p3) node[anchor=west] {\Large $p_3$};

\coordinate (o1) at (1,0);
    \fill [black] (o1) circle (2.5pt);
    \draw (o1) node[anchor=south west] {\Large $o_1$};
\coordinate (o2) at (1.8,0);
    \fill [black] (o2) circle (2.5pt);
    \draw (o2) node[anchor=south east] {\Large $o_2$};
\coordinate (pp1) at ($(o1)+\LengthPPone*({cos(\AnglePPone)},{sin(\AnglePPone)})$);
    \fill [black] (pp1) circle (2.5pt);
    \draw (pp1) node[anchor=east] {\Large $p_1$};
\coordinate (pp2) at ($(o1)+\LengthPPtwo*({cos(\AnglePPtwo)},{sin(\AnglePPtwo)})$);
    \fill [black] (pp2) circle (2.5pt);
    \draw (pp2) node[anchor=south] {\Large $p_2$};
\coordinate (pp3) at ($(o2)+\LengthPPthree*({cos(\AnglePPthree)},{sin(\AnglePPthree)})$);
    \fill [black] (pp3) circle (2.5pt);
    \draw (pp3) node[anchor=south] {\Large $p_3$};
\coordinate (pp4) at ($(o2)+\LengthPPfour*({cos(\AnglePPfour)},{sin(\AnglePPfour)})$);
    \fill [black] (pp4) circle (2.5pt);
    \draw (pp4) node[anchor=west] {\Large $p_4$};
\coordinate (pp5) at ($(o2)+\LengthPPfive*({cos(\AnglePPfive)},{sin(\AnglePPfive)})$);
    \fill [black] (pp5) circle (2.5pt);
    \draw (pp5) node[anchor=north] {\Large $p_5$};

\begin{scope}[line width=1.6pt]
\path[-,font=\scriptsize]
(o) edge[out=\AnglePone-10,in=\AnglePone-90] node[pos=0.5,sloped,allow upside down]{\midarrow} node[pos=0.5,anchor=south east]{\Large $2$}  (p1)
(o) edge[out=\AnglePone+10,in=\AnglePone+90] node[pos=0.5,sloped,allow upside down]{\midarrow} node[pos=0.5,anchor=north west]{\Large $2$}  (p1)
(o) edge[out=\AnglePtwo-10,in=\AnglePtwo-90] node[pos=0.5,sloped,allow upside down]{\midarrow} node[pos=0.5,anchor=west]{\Large $1$}  (p2)
(o) edge[out=\AnglePtwo+10,in=\AnglePtwo+90] node[pos=0.5,sloped,allow upside down]{\midarrow} node[pos=0.5,anchor=east]{\Large $1$}  (p2)
(o) edge[out=\AnglePthree-10,in=\AnglePthree-90] node[pos=0.5,sloped,allow upside down]{\midarrow} node[pos=0.5,anchor=north east]{\Large $2$}  (p3)
(o) edge[out=\AnglePthree+10,in=\AnglePthree+90] node[pos=0.5,sloped,allow upside down]{\midarrow} node[pos=0.5,anchor=south west]{\Large $2$}  (p3)

(o1) edge[out=0,in=180] node[pos=0.5,sloped,allow upside down]{\midarrow} node[pos=0.5,anchor=north]{\Large $2$}  (o2)
(o1) edge[out=\AnglePPone-10,in=\AnglePPone-90] node[pos=0.5,sloped,allow upside down]{\midarrow} node[pos=0.5,anchor=south east]{\Large $2$}  (pp1)
(o1) edge[out=\AnglePPone+10,in=\AnglePPone+90] node[pos=0.5,sloped,allow upside down]{\midarrow} node[pos=0.5,anchor=north west]{\Large $2$}  (pp1)
(pp2) edge[in=\AnglePPtwo-10,out=\AnglePPtwo-90] node[pos=0.5,sloped,allow upside down]{\midarrow} node[pos=0.5,anchor=west]{\Large $1$}  (o1)
(pp2) edge[in=\AnglePPtwo+10,out=\AnglePPtwo+90] node[pos=0.5,sloped,allow upside down]{\midarrow} node[pos=0.6,anchor=east]{\Large $1$}  (o1)
(o2) edge[out=\AnglePPthree-10,in=\AnglePPthree-90] node[pos=0.5,sloped,allow upside down]{\midarrow} node[pos=0.5,anchor=west]{\Large $2$}  (pp3)
(o2) edge[out=\AnglePPthree+10,in=\AnglePPthree+90] node[pos=0.5,sloped,allow upside down]{\midarrow} node[pos=0.5,anchor=east]{\Large $2$}  (pp3)
(pp4) edge[in=\AnglePPfour-10,out=\AnglePPfour-90] node[pos=0.5,sloped,allow upside down]{\midarrow} node[pos=0.5,anchor=north]{\Large $2$}  (o2)
(pp4) edge[in=\AnglePPfour+10,out=\AnglePPfour+90] node[pos=0.5,sloped,allow upside down]{\midarrow} node[pos=0.5,anchor=south]{\Large $2$}  (o2)
(o2) edge[out=\AnglePPfive-10,in=\AnglePPfive-90] node[pos=0.5,sloped,allow upside down]{\midarrow} node[pos=0.5,anchor=east]{\Large $2$}  (pp5)
(o2) edge[out=\AnglePPfive+10,in=\AnglePPfive+90] node[pos=0.5,sloped,allow upside down]{\midarrow} node[pos=0.5,anchor=west]{\Large $2$}  (pp5);
\end{scope}

\draw ($(o)+(0,-1)$) node {\Large Type I};
\draw ($(o1)+(0,-1.1)$) node {\Large Type II};
\end{tikzpicture}

\caption{Example of the underlying metric graphs of Harris-Mumford saturated metrized complex of type I and  II. For type I, $o$ is the eye vertex, $p_1$, $p_2$ and $p_3$ are the three pedal vertices. For type II, $o_1$ and $o_2$ are the eye vertices, $p_1$, $p_2$, $p_3$, $p_4$ and $p_5$ are the three pedal vertices. Examples of regulated global diagrams are also shown with edge multiplicities marked.} \label{F:HarrisMumford}
\end{figure}
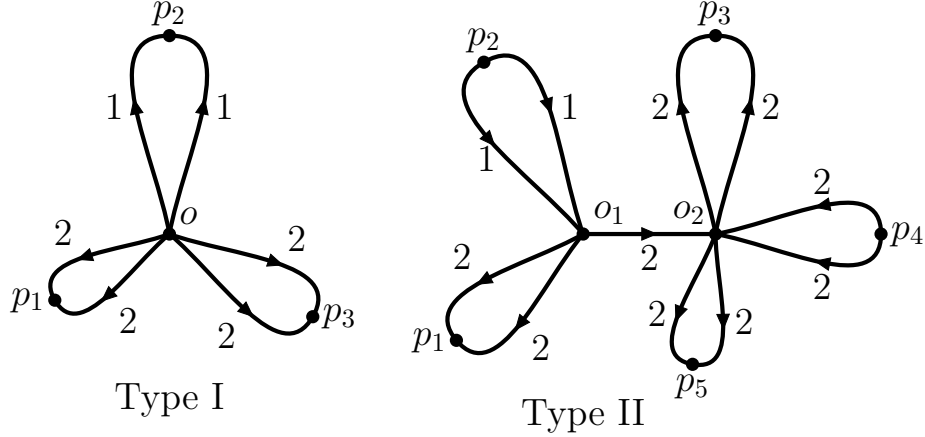

\begin{theorem} \label{T:HarrisMumford}
\begin{enumerate}
  \item A Harris-Mumford saturated metrized complex $\mathfrak{C}$ of type I has a base-point-free smoothable limit $g^1_d$ whose global diagram is regulated if and only if there exists a rational function $f$ of degree $d$ on the eye which has the same value and ramification indices on each pair of the $p_i$-marked points for all pedal vertices $p_i$.
  \item A Harris-Mumford saturated metrized complex $\mathfrak{C}$ of type II has a base-point-free smoothable limit $g^1_d$ whose global diagram is regulated if and only if there exist rational functions $f_1$ and $f_2$ on the two eyes respectively such that (i) for all pedal vertices $p_i$, the pair of $p_i$-marked points have the same value and ramification indices for the corresponding $f_j$ ($j=1$ or $2$), and (ii) $d=d_1+d_2-l$ where $d_1$ and $d_2$ are the degrees of $f_1$ and $f_2$ respectively and both ramification indices of $f_1$ and $f_2$ at respective marked points corresponding to the central segment take the same value $l$.
\end{enumerate}

\end{theorem}

\begin{proof}
For (1), first let $(\mathcal{D},\mathcal{H})$ represent a base-point-free smoothable limit $g^1_d$ with regulated global diagram. Then the only possible global diagrams are those with the same multiplicity along each pair of edges connecting the eye vertex $o$ and a pedal vertex. Consider a nonconstant function $f$ in $H_o$ where $o$ is the eye vertex. By the smoothing criterion (version II, Theorem~\ref{T:main}), $f$ has the same value and ramification indices on each pair of the $p_i$-marked points for all pedal vertices $p_i$. In addition, since $(\mathcal{D},\mathcal{H})$  is base-point free, $f$ must also have degree $d$.

On the other hand, if $f$ is a rational function of degree $d$ on the eye $C_o$ which has the same value and ramification indices on each pair of the $p_i$-marked points for all pedal vertices $p_i$, then we can construct a base-point-free smoothable limit $g^1_d$ represented by $(\mathcal{D},\mathcal{H})$  with regulated global diagram in the following way. First we let $H_o$ be the space of rational functions spanned by a constant function and $f$. Since for each pedal vertex $p_i$, the distances of the two edges connecting $p_i$ and $o$ are the same and $f$ has the same multiplicities on the pair of $p_i$-marked points, we can construct a regulated global diagram compatible with the local diagram for $H_p$ (Remark~\ref{R:LocalHp}). Associate projective lines to the points $p$ other than $o$ and we can always construct a two dimensional space $H_p$ of rational functions on $C_p$ whose local diagram is compatible with the global diagram. Now for all $p\in\Gamma$, let $D_p=\Sigma_{t\in\In(p)} m(p,t)(\red_p(t)) + D^-_{f_p}$ based on the notations in Remark~\ref{R:LocalHp}. By the smoothing criterion, the $(\{D_p\}_{p\in\Gamma},\{H\}_{p\in\Gamma})$ constructed in this way represents a smoothable $g^1_d$.

For (2), suppose $(\mathcal{D},\mathcal{H})$  represent a base-point-free smoothable limit $g^1_d$ with regulated global diagram. Then the only possible global diagrams are those with the same multiplicity along each pair of edges connecting a pedal vertex to its corresponding eye vertex $o_1$ or $o_2$ and with unchanged multiplicity along $o_1o_2$. Consider nonconstant functions $f_1\in H_{o_1}$ and $f_2\in H_{o_2}$. Therefore condition (i) is satisfied using a similar argument as for (1) and the ramification indices of $f_1$ and $f_2$ at respective marked points $u_1$ and $u_2$ corresponding to the segment $o_1o_2$ take the same value, say $l$. Moreover, if $u_1$ is a pole of $f_1$, then $u_2$ is not a pole of $f_2$ and vice versa. Then the total degree $d$ of $(\mathcal{D},\mathcal{H})$  must be $d=d_1+d_2-l$ where $d_1$ and $d_2$ are the degrees of $f_1$ and $f_2$ respectively.

On the other hand, suppose we have rational functions $f'_1$ on $C_{o_1}$ and $f'_2$ on $C_{o_2}$ satisfying conditions (i) and (ii). If none of $u_1$ and $u_2$ are respectively poles of $f_1$ and $f_2$, we let $f_1=f'_1$ and $f_2=1/(f'_2-f'_2(u_2))$. If both $u_1$ and $u_2$ are respectively poles of $f_1$ and $f_2$, we let $f_1=1/f'_1$ and $f_2=f'_2$. Now $u_1$ is not a pole of $f_1$ and $u_2$ is a pole of $f_2$ while both $f_1$ and $f_2$ satisfy conditions (i) and (ii). Moreover, the ramification index of $f_1$ at $u_1$ is the same as the ramification index of $f_2$ at $u_2$. For $i=1,2$, let $H_{o_i}$ be spanned by a constant function and $f_i$. Then we can construct a regulated global diagram compatible to the local diagrams associated to $H_{o_1}$ and $H_{o_2}$ and build a base-point-free smoothable limit $g^1_d$ on this global diagram following a similar approach as in (1).
\end{proof}

\subsection{The Smoothing Criterion on Metrized Complexes} \label{S:SmoothingMC}
Recall that in Remark~\ref{R:restriction}, we compared the notion of metrized complexes introduced by Amini and Baker \cite{AB12} and the notion of saturated metrized complexes. Consider a metric graph $\Gamma$ and a vertex set $A$ of $\Gamma$. Suppose $\mathfrak{C}_A$ is a  metrized complex with its underlying metric graph being $\Gamma$ and each point $p\in A$ is associated with a curve $C_p$. As for pre-limit $g^r_d$s defined on saturated metrized complexes (Definition~\ref{D:LimLinSeries}), we say that a pre-limit $g^r_d$  on the metrized complex $\mathfrak{C}_A$  is represented by the data $(\mathcal{D}_A, \mathcal{H}_A)$  where $\mathcal{D}_A= (D_\Gamma,\{ D_p\}_{p \in A})$ is an effective divisor of degree $d$ on $\mathfrak{C}_A$ and  $\mathcal{H}_A = \{ H_p\}_{p \in A}$ where $H_p$ is an $(r+1)$-dimensional subspace of the function field of $C_p$. (See \cite{AB12} for more details about the divisor theory on metrized complexes.) In addition, as in Definition~\ref{D:smoothable}, we say $(\mathcal{D}_A, \mathcal{H}_A)$ is smoothable if there exists a $g^r_d$ represented by $(D,H)$ on exists a smooth proper curve $X/\mathbb{K}$ specialized to $(\mathcal{D}_A, \mathcal{H}_A)$ on $\mathfrak{C}_A$. 

Consider a saturated metrized complex $\mathfrak{C}$ which is a saturation of $\mathfrak{C}_A$ (Remark~\ref{R:restriction}). We say a divisor $\mathcal{D}= (D'_\Gamma,\{ D'_p\}_{p \in \Gamma})$ on $\mathfrak{C}$ is a saturation of $\mathcal{D}_A= (D_\Gamma,\{ D_p\}_{p \in A})$ on $\mathfrak{C}_A$ if $D_\Gamma=D'_\Gamma$ and $D_p=D'_p$ for all $p\in A$, and we say $\mathcal{H} = \{ H'_p\}_{p \in \Gamma}$  is a saturation of $\mathcal{H}_A = \{ H_p\}_{p \in A}$ if $H'_p=H_p$ for all $p\in A$. Then naturally we have the following statement. 

\begin{lemma}
A pre-limit $g^r_d$ represented by $(\mathcal{D}_A, \mathcal{H}_A)$ on a metrized complex $\mathfrak{C}_A$ is smoothable if and only if there exists a saturation $(\mathcal{D}, \mathcal{H})$ of $(\mathcal{D}_A, \mathcal{H}_A)$ on a saturation $\mathfrak{C}$ of $\mathfrak{C}_A$ such that $(\mathcal{D}, \mathcal{H})$ is smoothable. 
\end{lemma}

Our smoothing criterion for the rank one case on saturated metrized complexes can be extended to the case for metrized complexes. The subtlety here is that we need to consider all possible saturations of $(\mathcal{D}_A, \mathcal{H}_A)$ and each saturation $(\mathcal{D}, \mathcal{H})$ affords its own global diagram which might either be solvable or not solvable. So to say $(\mathcal{D}_A, \mathcal{H}_A)$ is smoothable, we should be able to single out a solvable 
saturation $(\mathcal{D}, \mathcal{H})$ which satisfies the intrinsic global compatibility conditions. Fortunately, even though the process of determining rank-one smoothability on metrized complexes is usually more complicated, it is still finitely verifiable, as in the following example. 

\begin{example}
\begin{figure}[tbp] 
\centering
\begin{tikzpicture}[>=to,x=2cm,y=2cm]
\def\LengthPQ{2.2}
\def\Gap{0.45}

\coordinate (a) at (0,0);
\coordinate (b) at (-0.8,-3);
\coordinate (c) at (3.2,0);
\coordinate (d) at (3.2,-2.3);
\coordinate (e) at (3.2,-4.6);

\draw (a) node {\Large (a)};
\draw ($(a)+(1.2,-2)$) node {\LARGE $\Gamma$};

\coordinate (p) at ($(a)+(0,-0.8)$);
    \fill [black] (p) circle (2.5pt);
    \draw (p) node[anchor=east] {\Large $p$};
    
\coordinate (q) at ($(p)+(\LengthPQ,0)$);
    \fill [black] (q) circle (2.5pt);
    \draw (q) node[anchor=west] {\Large $q$};
    
\coordinate (o) at ($(p)+(1/2*\LengthPQ,0.29*\LengthPQ)$);
    \fill [black] (o) circle (2.5pt);
    \draw (o) node[anchor=north] {\Large $o$};
    \draw (o) node[anchor=south] {\Large $3$};

\begin{scope}[line width=1.5pt]
\path[-,font=\scriptsize]
(p) edge[out=90,in=90] node[pos=0.87,sloped,allow upside down]{\midarrow} node[pos=0.87,anchor=south]{\large $1$} node[pos=0.35,anchor=south]{\Large $L_1$}   (q)
(q) edge[out=90,in=90] node[pos=0.87,sloped,allow upside down]{\midarrow} node[pos=0.87,anchor=south]{\large $1$}  (p)

(p) edge[out=20,in=160] node[pos=0.13,sloped,allow upside down]{\midarrow} node[pos=0.13,anchor=south]{\large $2$} node[pos=0.35,anchor=south]{\Large $L_2$}  (q)
(q) edge[out=160,in=20] node[pos=0.13,sloped,allow upside down]{\midarrow} node[pos=0.13,anchor=south]{\large $1$}  (p)

(p) edge[out=-20,in=-160] node[pos=0.13,sloped,allow upside down]{\midarrow} node[pos=0.13,anchor=north]{\large $1$} node[pos=0.35,anchor=north]{\Large $L_3$}  (q)
(q) edge[out=-160,in=-20] node[pos=0.13,sloped,allow upside down]{\midarrow} node[pos=0.13,anchor=north]{\large $1$}  (p)

(p) edge[out=-90,in=-90] node[pos=0.13,sloped,allow upside down]{\midarrow} node[pos=0.13,anchor=north]{\large $1$} node[pos=0.35,anchor=north]{\Large $L_4$}  (q)
(q) edge[out=-90,in=-90] node[pos=0.13,sloped,allow upside down]{\midarrow} node[pos=0.13,anchor=north]{\large $2$}  (p);

\end{scope}

\draw (b) node {\Large (b)};

 \fill [black] ($(b)+(0.8*\LengthPQ,0)$) circle (2.5pt);
 \draw ($(b)+(0.8*\LengthPQ,0)$) node[anchor=south] {\Large $o$};
 \fill [black] ($(b)+(0.3*\LengthPQ,0)$) circle (2.5pt);
 \draw ($(b)+(0.3*\LengthPQ,0)$) node[anchor=east] {\Large $p$};
 \fill [black] ($(b)+(1.3*\LengthPQ,0)$) circle (2.5pt);
 \draw ($(b)+(1.3*\LengthPQ,0)$) node[anchor=west] {\Large $q$};
 
 \fill [black] ($(b)+(0.8*\LengthPQ,-\Gap)$) circle (2.5pt);
 \draw ($(b)+(0.8*\LengthPQ,-\Gap)$) node[anchor=south] {\Large $o$};
 \fill [black] ($(b)+(0.3*\LengthPQ,-\Gap)$) circle (2.5pt);
 \draw ($(b)+(0.3*\LengthPQ,-\Gap)$) node[anchor=east] {\Large $p$};
 \fill [black] ($(b)+(1.3*\LengthPQ,-\Gap)$) circle (2.5pt);
 \draw ($(b)+(1.3*\LengthPQ,-\Gap)$) node[anchor=west] {\Large $q$};
 \fill [black] ($(b)+(0.55*\LengthPQ,-\Gap)$) circle (2.5pt);
 
 \fill [black] ($(b)+(0.8*\LengthPQ,-2*\Gap)$) circle (2.5pt);
 \draw ($(b)+(0.8*\LengthPQ,-2*\Gap)$) node[anchor=south] {\Large $o$};
 \fill [black] ($(b)+(0.3*\LengthPQ,-2*\Gap)$) circle (2.5pt);
 \draw ($(b)+(0.3*\LengthPQ,-2*\Gap)$) node[anchor=east] {\Large $p$};
 \fill [black] ($(b)+(1.3*\LengthPQ,-2*\Gap)$) circle (2.5pt);
 \draw ($(b)+(1.3*\LengthPQ,-2*\Gap)$) node[anchor=west] {\Large $q$};
 \fill [black] ($(b)+(1.05*\LengthPQ,-2*\Gap)$) circle (2.5pt);
 
 \fill [black] ($(b)+(0.3*\LengthPQ,-3*\Gap)$) circle (2.5pt);
 \draw ($(b)+(0.3*\LengthPQ,-3*\Gap)$) node[anchor=east] {\Large $p$};
 \fill [black] ($(b)+(0.6*\LengthPQ,-3*\Gap)$) circle (2.5pt);
  \fill [black] ($(b)+(0.9*\LengthPQ,-3*\Gap)$) circle (2.5pt); 
 \fill [black] ($(b)+(1.3*\LengthPQ,-3*\Gap)$) circle (2.5pt);
 \draw ($(b)+(1.3*\LengthPQ,-3*\Gap)$) node[anchor=west] {\Large $q$};

  \fill [black] ($(b)+(0.3*\LengthPQ,-4*\Gap)$) circle (2.5pt);
 \draw ($(b)+(0.3*\LengthPQ,-4*\Gap)$) node[anchor=east] {\Large $p$};
 \fill [black] ($(b)+(0.8*\LengthPQ,-4*\Gap)$) circle (2.5pt);
 \fill [black] ($(b)+(1.3*\LengthPQ,-4*\Gap)$) circle (2.5pt);
 \draw ($(b)+(1.3*\LengthPQ,-4*\Gap)$) node[anchor=west] {\Large $q$};
 
  \fill [black] ($(b)+(0.3*\LengthPQ,-5*\Gap)$) circle (2.5pt);
 \draw ($(b)+(0.3*\LengthPQ,-5*\Gap)$) node[anchor=east] {\Large $p$};
 \fill [black] ($(b)+(0.7*\LengthPQ,-5*\Gap)$) circle (2.5pt);
  \fill [black] ($(b)+(1.0*\LengthPQ,-5*\Gap)$) circle (2.5pt); 
 \fill [black] ($(b)+(1.3*\LengthPQ,-5*\Gap)$) circle (2.5pt);
 \draw ($(b)+(1.3*\LengthPQ,-5*\Gap)$) node[anchor=west] {\Large $q$};
 
\begin{scope}[line width=1.5pt]
\path[-,font=\scriptsize]
($(b)+(0.8*\LengthPQ,0)$)  edge node[pos=0.5,sloped,allow upside down]{\midarrow}  node[pos=0.5,anchor=south]{\large $1$}  ($(b)+(0.3*\LengthPQ,0)$)
($(b)+(0.8*\LengthPQ,0)$)  edge node[pos=0.5,sloped,allow upside down]{\midarrow}  node[pos=0.5,anchor=south]{\large $1$}   ($(b)+(1.3*\LengthPQ,0)$)

($(b)+(0.8*\LengthPQ,-\Gap)$)  edge node[pos=0.5,sloped,allow upside down]{\midarrow} node[pos=0.5,anchor=south]{\large $1$}   ($(b)+(1.3*\LengthPQ,-\Gap)$)
($(b)+(0.8*\LengthPQ,-\Gap)$)  edge node[pos=0.5,sloped,allow upside down]{\midarrow}  node[pos=0.5,anchor=south]{\large $2$}  ($(b)+(0.55*\LengthPQ,-\Gap)$)
($(b)+(0.55*\LengthPQ,-\Gap)$)  edge node[pos=0.5,sloped,allow upside down]{\midarrow}  node[pos=0.5,anchor=south]{\large $1$}  ($(b)+(0.3*\LengthPQ,-\Gap)$)

($(b)+(0.8*\LengthPQ,-2*\Gap)$)  edge node[pos=0.5,sloped,allow upside down]{\midarrow}  node[pos=0.5,anchor=south]{\large $1$}   ($(b)+(0.3*\LengthPQ,-2*\Gap)$)
($(b)+(0.8*\LengthPQ,-2*\Gap)$)  edge node[pos=0.5,sloped,allow upside down]{\midarrow} node[pos=0.5,anchor=south]{\large $2$}    ($(b)+(1.05*\LengthPQ,-2*\Gap)$)
($(b)+(1.05*\LengthPQ,-2*\Gap)$)  edge node[pos=0.5,sloped,allow upside down]{\midarrow} node[pos=0.5,anchor=south]{\large $1$}    ($(b)+(1.3*\LengthPQ,-2*\Gap)$)

($(b)+(0.3*\LengthPQ,-3*\Gap)$) edge node[pos=0.5,sloped,allow upside down]{\midarrow}  node[pos=0.5,anchor=south]{\large $2$}   ($(b)+(0.6*\LengthPQ,-3*\Gap)$)
($(b)+(0.6*\LengthPQ,-3*\Gap)$)  edge node[pos=0.5,sloped,allow upside down]{\midarrow} node[pos=0.5,anchor=south]{\large $1$}    ($(b)+(0.9*\LengthPQ,-3*\Gap)$)
($(b)+(1.3*\LengthPQ,-3*\Gap)$)  edge node[pos=0.5,sloped,allow upside down]{\midarrow} node[pos=0.5,anchor=south]{\large $1$}    ($(b)+(0.9*\LengthPQ,-3*\Gap)$)

($(b)+(0.3*\LengthPQ,-4*\Gap)$) edge node[pos=0.5,sloped,allow upside down]{\midarrow}  node[pos=0.5,anchor=south]{\large $1$}   ($(b)+(0.8*\LengthPQ,-4*\Gap)$)
($(b)+(1.3*\LengthPQ,-4*\Gap)$)  edge node[pos=0.5,sloped,allow upside down]{\midarrow} node[pos=0.5,anchor=south]{\large $1$}    ($(b)+(0.8*\LengthPQ,-4*\Gap)$)

($(b)+(0.3*\LengthPQ,-5*\Gap)$) edge node[pos=0.5,sloped,allow upside down]{\midarrow}  node[pos=0.5,anchor=south]{\large $1$}   ($(b)+(0.7*\LengthPQ,-5*\Gap)$)
($(b)+(1.0*\LengthPQ,-5*\Gap)$)  edge node[pos=0.5,sloped,allow upside down]{\midarrow} node[pos=0.5,anchor=south]{\large $1$}    ($(b)+(0.7*\LengthPQ,-5*\Gap)$)
($(b)+(1.3*\LengthPQ,-5*\Gap)$)  edge node[pos=0.5,sloped,allow upside down]{\midarrow} node[pos=0.5,anchor=south]{\large $2$}    ($(b)+(1.0*\LengthPQ,-5*\Gap)$);
\end{scope}

 \draw [line width=1pt, decorate,decoration={brace,amplitude=8pt},xshift=0pt,yshift=0pt]
($(b)+(1.4*\LengthPQ,0)$)  -- ($(b)+(1.4*\LengthPQ,-2*\Gap)$) node [black,midway,xshift=0.6cm] 
{\Large $L_1$};

 \draw [decorate,decoration={brace,amplitude=8pt},xshift=0pt,yshift=0pt]
($(b)+(0.8*\LengthPQ,-\Gap)$)  -- ($(b)+(0.55*\LengthPQ,-\Gap)$) node [black,midway,yshift=-0.5cm] 
{\large $x_1$};
 \draw [decorate,decoration={brace,amplitude=8pt},xshift=0pt,yshift=0pt]
($(b)+(0.55*\LengthPQ,-\Gap)$)  -- ($(b)+(0.3*\LengthPQ,-\Gap)$) node [black,midway,yshift=-0.5cm] 
{\large $y_1$}; 
 
 \draw [decorate,decoration={brace,amplitude=8pt,mirror},xshift=0pt,yshift=0pt]
($(b)+(0.8*\LengthPQ,-2*\Gap)$)  -- ($(b)+(1.05*\LengthPQ,-2*\Gap)$) node [black,midway,yshift=-0.5cm] 
{\large $x'_1$};
 \draw [decorate,decoration={brace,amplitude=8pt},xshift=0pt,yshift=0pt]
($(b)+(1.3*\LengthPQ,-2*\Gap)$)  -- ($(b)+(1.05*\LengthPQ,-2*\Gap)$) node [black,midway,yshift=-0.5cm] 
{\large $y'_1$};

 \draw [decorate,decoration={brace,amplitude=8pt},xshift=0pt,yshift=0pt]
($(b)+(1.3*\LengthPQ,-3*\Gap)$)  -- ($(b)+(0.9*\LengthPQ,-3*\Gap)$) node [black,midway,yshift=-0.5cm] 
{\large $z _2$};
 \draw [decorate,decoration={brace,amplitude=8pt},xshift=0pt,yshift=0pt]
($(b)+(0.9*\LengthPQ,-3*\Gap)$)  -- ($(b)+(0.6*\LengthPQ,-3*\Gap)$) node [black,midway,yshift=-0.5cm] 
{\large $y_2$};
 \draw [decorate,decoration={brace,amplitude=8pt},xshift=0pt,yshift=0pt]
($(b)+(0.6*\LengthPQ,-3*\Gap)$)  -- ($(b)+(0.3*\LengthPQ,-3*\Gap)$) node [black,midway,yshift=-0.5cm] 
{\large $x_2$};

 \draw [decorate,decoration={brace,amplitude=8pt},xshift=0pt,yshift=0pt]
($(b)+(1.3*\LengthPQ,-4*\Gap)$)  -- ($(b)+(0.8*\LengthPQ,-4*\Gap)$) node [black,midway,yshift=-0.5cm] 
{\large $y _3$};
 \draw [decorate,decoration={brace,amplitude=8pt},xshift=0pt,yshift=0pt]
($(b)+(0.8*\LengthPQ,-4*\Gap)$)  -- ($(b)+(0.3*\LengthPQ,-4*\Gap)$) node [black,midway,yshift=-0.5cm] 
{\large $x_3$};

 \draw [decorate,decoration={brace,amplitude=8pt},xshift=0pt,yshift=0pt]
($(b)+(1.3*\LengthPQ,-5*\Gap)$)  -- ($(b)+(1.0*\LengthPQ,-5*\Gap)$) node [black,midway,yshift=-0.5cm] 
{\large $x _4$};
 \draw [decorate,decoration={brace,amplitude=8pt},xshift=0pt,yshift=0pt]
($(b)+(1.0*\LengthPQ,-5*\Gap)$)  -- ($(b)+(0.7*\LengthPQ,-5*\Gap)$) node [black,midway,yshift=-0.5cm] 
{\large $y_4$};
 \draw [decorate,decoration={brace,amplitude=8pt},xshift=0pt,yshift=0pt]
($(b)+(0.7*\LengthPQ,-5*\Gap)$)  -- ($(b)+(0.3*\LengthPQ,-5*\Gap)$) node [black,midway,yshift=-0.5cm] 
{\large $z_4$};

\draw ($(b)+(1.5*\LengthPQ,-3*\Gap)$) node {\Large $L_2$};
\draw ($(b)+(1.5*\LengthPQ,-4*\Gap)$) node {\Large $L_3$};
\draw ($(b)+(1.5*\LengthPQ,-5*\Gap)$) node {\Large $L_4$};

\draw (c) node {\Large (c)};

\coordinate (o) at ($(c)+(0.2,-0.5)$);
    \draw (o) node[anchor=east] {\Large $o$};
\coordinate (p) at ($(c)+(1.2,-0.3)$);
    \draw (p) node[anchor=south east] {\Large $p$};
\coordinate (q) at ($(c)+(1.2,-0.7)$);
    \draw (q) node[anchor=north east] {\Large $q$};
\coordinate (pq1) at ($(c)+(2.2,0)$);
\coordinate (pq2) at ($(c)+(2.2,-0.5)$);
\coordinate (pq3) at ($(c)+(2.2,-1)$);

\coordinate (x) at ($(c)+(0.2,-1.7)$);
\coordinate (y) at ($(c)+(1.2,-1.7)$);
\coordinate (z1) at ($(c)+(2.2,-1.4)$);
\coordinate (z2) at ($(c)+(2.2,-1.7)$);
\coordinate (z3) at ($(c)+(2.2,-2.0)$);

\begin{scope}[line width=1.6pt, every node/.style={sloped,allow upside down}]
  \draw (o) --  (p);
  \draw (o) --  (q);
  \draw (p) --  (pq1);
  \draw (p) --  (pq2);
  \draw (p) --  (pq3);
  \draw (q) --  (pq1);
  \draw (q) --  (pq2);
  \draw (q) --  (pq3);
  \draw (x) --  (y);
  \draw (y) --  (z1);
  \draw (y) --  (z2);
  \draw (y) --  (z3);
\end{scope}

\begin{scope}[line width=1pt, dashed]
\draw (o)-- (x);
\draw (p)-- (y);
\end{scope}

\draw ($(c)+(2.5,-1.7)$) node {\LARGE $\mathcal{B}$};
\draw ($(c)+(2.5,-0.5)$) node {\LARGE $\Gamma$};
\draw [-stealth, line width=1.2pt] ($(c)+(2.5,-0.8)$)  -- ($(c)+(2.5,-1.4)$);

\draw [rotate around={0:(o)},line width=1pt,dash pattern=on 2pt off 2pt] ($(o)+(0.18,0)$) ellipse (0.05 and 0.14);

\draw (d) node {\Large (d)};

\coordinate (o) at ($(d)+(0.2,-0.7)$);
    \draw (o) node[anchor=east] {\Large $o$};
\coordinate (p) at ($(d)+(1.2,-0.7)$);
    \draw (p) node[anchor=north east] {\Large $p$};
\coordinate (q) at ($(d)+(1,-0.3)$);
    \draw (q) node[anchor=south east] {\Large $q$};
\coordinate (pq1) at ($(d)+(2.2,0)$);
\coordinate (pq2) at ($(d)+(2.2,-0.5)$);
\coordinate (pq3) at ($(d)+(2.2,-1)$);

\coordinate (x) at ($(d)+(0.2,-1.7)$);
\coordinate (y1) at ($(d)+(1,-1.7)$);
\coordinate (y2) at ($(d)+(1.2,-1.7)$);
\coordinate (z1) at ($(d)+(2.2,-1.4)$);
\coordinate (z2) at ($(d)+(2.2,-1.7)$);
\coordinate (z3) at ($(d)+(2.2,-2.0)$);

\begin{scope}[line width=1.6pt, every node/.style={sloped,allow upside down}]
  \draw (o) --  (p);
  \draw (o) --  (q);
  \draw (p) --  (pq1);
  \draw (p) --  (pq2);
  \draw (p) --  (pq3);
  \draw (q) --  (pq1);
  \draw (q) --  (pq2);
  \draw (q) --  (pq3);
  \draw (x) --  (y2);
  \draw (y2) --  (z1);
  \draw (y2) --  (z2);
  \draw (y2) --  (z3);
\end{scope}

\begin{scope}[line width=1pt, dashed]
\draw (o)-- (x);
\draw (p)-- (y2);
\draw (q)-- (y1);
\end{scope}

\draw ($(d)+(2.5,-1.7)$) node {\LARGE $\mathcal{B}$};
\draw ($(d)+(2.5,-0.5)$) node {\LARGE $\Gamma$};
\draw [-stealth, line width=1.2pt] ($(d)+(2.5,-0.8)$)  -- ($(d)+(2.5,-1.4)$);

\draw [rotate around={0:(p)},line width=1pt,dash pattern=on 2pt off 2pt] ($(q)+(0.18,-0.04)$) ellipse (0.05 and 0.18);
\draw [rotate around={0:(o)},line width=1pt,dash pattern=on 2pt off 2pt] ($(o)+(0.18,0.04)$) ellipse (0.05 and 0.14);

\draw (e) node {\Large (e)};

\coordinate (o) at ($(e)+(0.2,-0.7)$);
    \draw (o) node[anchor=east] {\Large $o$};
\coordinate (p) at ($(e)+(1,-0.3)$);
    \draw (p) node[anchor=south east] {\Large $p$};
\coordinate (q) at ($(e)+(1.2,-0.7)$);
    \draw (q) node[anchor=north east] {\Large $q$};
\coordinate (pq1) at ($(e)+(2.2,0)$);
\coordinate (pq2) at ($(e)+(2.2,-0.5)$);
\coordinate (pq3) at ($(e)+(2.2,-1)$);

\coordinate (x) at ($(e)+(0.2,-1.7)$);
\coordinate (y1) at ($(e)+(1,-1.7)$);
\coordinate (y2) at ($(e)+(1.2,-1.7)$);
\coordinate (z1) at ($(e)+(2.2,-1.4)$);
\coordinate (z2) at ($(e)+(2.2,-1.7)$);
\coordinate (z3) at ($(e)+(2.2,-2.0)$);

\begin{scope}[line width=1.6pt, every node/.style={sloped,allow upside down}]
  \draw (o) --  (p);
  \draw (o) --  (q);
  \draw (p) --  (pq1);
  \draw (p) --  (pq2);
  \draw (p) --  (pq3);
  \draw (q) --  (pq1);
  \draw (q) --  (pq2);
  \draw (q) --  (pq3);
  \draw (x) --  (y2);
  \draw (y2) --  (z1);
  \draw (y2) --  (z2);
  \draw (y2) --  (z3);
\end{scope}

\begin{scope}[line width=1pt, dashed]
\draw (o)-- (x);
\draw (p)-- (y1);
\draw (q)-- (y2);
\end{scope}

\draw ($(e)+(2.5,-1.7)$) node {\LARGE $\mathcal{B}$};
\draw ($(e)+(2.5,-0.5)$) node {\LARGE $\Gamma$};
\draw [-stealth, line width=1.2pt] ($(e)+(2.5,-0.8)$)  -- ($(e)+(2.5,-1.4)$);

\draw [rotate around={0:(p)},line width=1pt,dash pattern=on 2pt off 2pt] ($(p)+(0.18,-0.04)$) ellipse (0.05 and 0.18);
\draw [rotate around={0:(o)},line width=1pt,dash pattern=on 2pt off 2pt] ($(o)+(0.18,0.04)$) ellipse (0.05 and 0.14);

\end{tikzpicture}
\caption{An example of smoothability test on a metrized complex with its underlying metric graph being a banana graph of genus $3$ and its vertex set being the two valence-$4$ points.} \label{F:ExMetCom}
\end{figure}
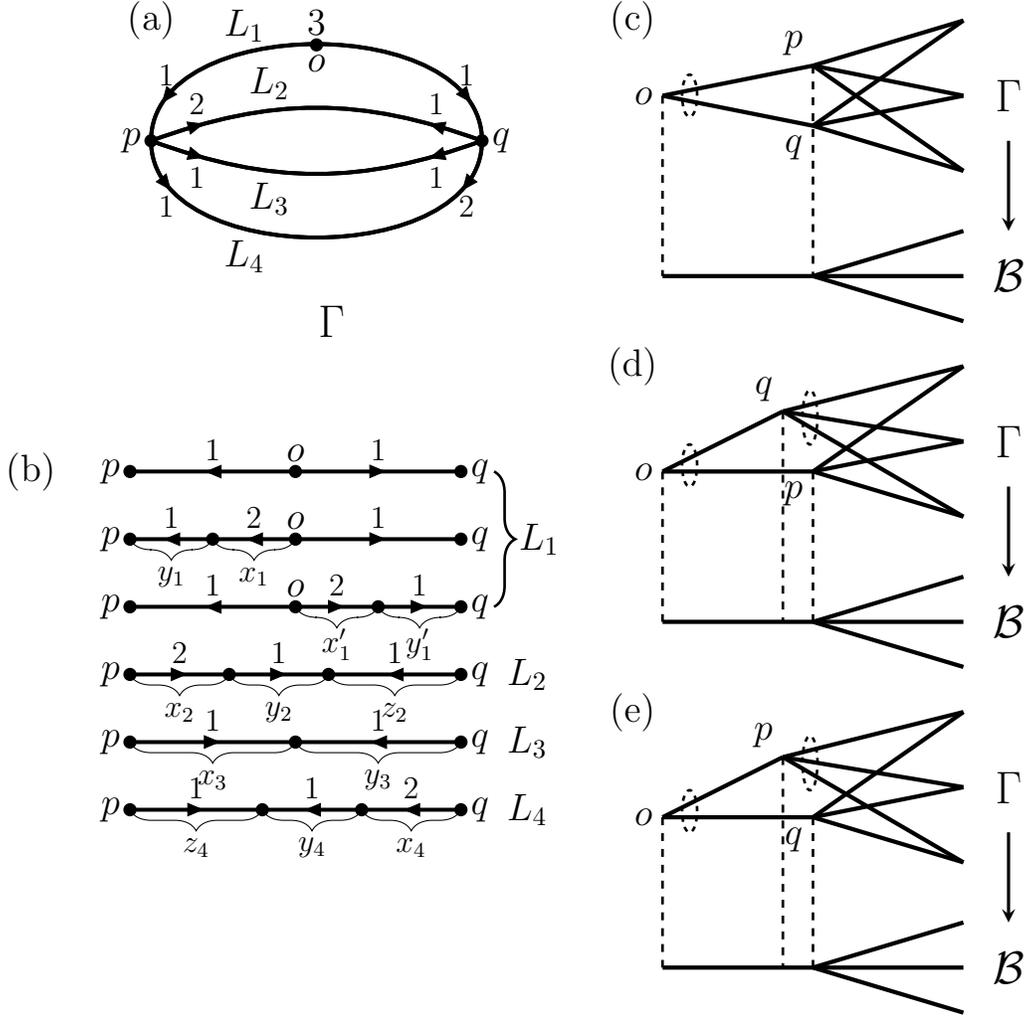
Consider a metric banana graph $\Gamma$ of genus $3$ as shown in Figure~\ref{F:ExMetCom}(a). Let $A=\{p,q\}$ be a vertex set of $\Gamma$ and the suppose the four edges $L_1$, $L_2$, $L_3$ and $L_4$ connecting $p$ and $q$ have the same length $1$. Let $o$ be the middle point of $L_1$. Let $\mathfrak{C}_A$ be a metrized complex with underlying metric graph being $\Gamma$ and associated curves being $C_p$ and $C_q$ with respect to $p$ and $q$ respectively. Let $(\mathcal{D}_A,\mathcal{H}_A)$ where $\mathcal{D}_A=(D_\Gamma,\{D_p,D_q\})$ has degree $d$ and $\mathcal{H}_A=\{H_p,H_q\}$ represent a pre-limit $g^1_d$ on $\mathfrak{C}_A$.  More specifically,   we assume that  $D_\Gamma(o)=3$, $D_p$ is compatible to $H_p$, $D_q$ is compatible to $H_q$, and $D_\Gamma(p')=0$ for every $p'\in\Gamma\setminus\{o,p,q\}$. The local diagrams at $p$ and $q$ induced by $H_p$ and $H_q$ respectively are also shown in Figure~\ref{F:ExMetCom}(a), i.e., the tangent directions at $p$ corresponding to  $L_1$, $L_2$, $L_3$ and $L_4$ have multiplicities $-1$, $2$, $1$ and $1$ respectively, and  the tangent directions at $q$ corresponding to  $L_1$, $L_2$, $L_3$ and $L_4$ have multiplicities $-1$, $1$, $1$ and $2$ respectively. 

Let $\mathfrak{C}$ be a saturation of  $\mathfrak{C}_A$. Then $(\mathcal{D}_A,\mathcal{H}_A)$ is smoothable if and only if there exists a saturation  of $(\mathcal{D}_A,\mathcal{H}_A)$ on $\mathfrak{C}$ which is smoothable.  Let $(\mathcal{D},\mathcal{H})$ be a diagrammatic saturation of $(\mathcal{D}_A,\mathcal{H}_A)$ on $\mathfrak{C}$. First we need to determine whether $(\mathcal{D},\mathcal{H})$ is solvable. Figure~\ref{F:ExMetCom}(b) shows the allowable cases of the discrete 1-form $\omega$ in the global diagram of $(\mathcal{D},\mathcal{H})$ restricted to the four edges $L_1$, $L_2$, $L_3$ and $L_4$. Note that the variation of the multiplicity along $L_2$, $L_3$ or $L_4$ from $p$ to $q$ must be non-increasing since $D_\Gamma(p')=0$ for $p'$ in the interior of these edges. In addition, we have the following restrictions for the lengths of segments with uniform multiplicities: $x_2+y_2+z_2=x_3+y_3=x_4+y_4+z_4=1$, $x_2,z_2,x_3,y_3,x_4,z_4>0$ and $y_2,y_4\geqslant 0$.  The case for $L_1$ is a little bit special since $D_\Gamma$ has value $3$ at the middle point $o$ of $L_1$. Again, the variation of  multiplicity from $o$ to $p$ or $q$ must be non-increasing. 
Thus the two tangent directions at $o$ (one from $o$ to $p$ and the other from $o$ to $q$) must be outgoing tangent directions in  the local diagram at $o$ induced by $(\mathcal{D},\mathcal{H})$ . In addition, these two tangent directions must be locally equivalent if we want  $(\mathcal{D},\mathcal{H})$ to be smoothable since $o$ must map to the root of the corresponding bifurcation tree and there is only one forward tangent direction from the root (see Figure~\ref{F:ExMetCom}(c) (d) and (e)). The three possible cases on $L_1$ are shown in Figure~\ref{F:ExMetCom}(b) and the restrictions are $x_1+y_1=x'_1+y'_1=1/2$, $x_1,x'_1\geqslant 0$ and $y_1,y'_1>0$. 

Let $\int_{L_i}\omega$ be the integration of $\omega$ along $L_i$ from $p$ to $q$. Then $(\mathcal{D},\mathcal{H})$ is solvable if and only if $\int_{L_1}\omega=\int_{L_2}\omega=\int_{L_3}\omega=\int_{L_4}\omega$. Depending on the cases, $\int_{L_1}\omega=0$ or $1/2-2x_1-y_1$ or $2x'+y'-1/2$. Therefore, $-1/2<\int_{L_1}\omega<1/2$. Similarly, $-1<\int_{L_2}\omega=2x_2+y_2-z_2<2$, $-1<\int_{L_3}\omega=x_3-y_3<1$ and 
$-2<\int_{L_4}\omega=-2x_4-y_4+z_4<1$. Therefore,  by adjusting the values of $x_2,y_2,z_2,x_3,y_3,x_4,y_4,z_4$, for the three cases of the global diagram restricted to $L_1$ in Figure~\ref{F:ExMetCom}(b), we can always find a solvable global diagram and the corresponding $(\mathcal{D},\mathcal{H})$. Moreover, the projection from $\Gamma$ to a bifurcation tree corresponding to each of the three solvable cases are sketched in  Figure~\ref{F:ExMetCom}(c), (d) and (e) respectively. 

For the case in Figure~\ref{F:ExMetCom}(c), $p$ and $q$ maps to the same point which has three forward tangent directions in the bifurcation tree. Therefore, to determine the smoothability of a solvable $(\mathcal{D},\mathcal{H})$ whose projection to its bifurcation tree is as in Figure~\ref{F:ExMetCom}(c), the intrinsic global compatibility conditions are trivially satisfied and only need to be tested for $H_p$ and $H_q$. More precisely, suppose $H_p$ has a basis $\{1,f_p\} $, $H_q$ has a basis $\{1,f_q\}$, $u_2$, $u_3$ and $u_4$ are the marked points on $C_p$ which are the reductions of the tangent directions at $p$ corresponding to the edges $L_2$, $L_3$ and $L_4$ respectively, and $v_2$, $v_3$ and $v_4$ are the marked points on $C_q$ which are the reduction of the tangent directions at $q$ corresponding to the edges $L_2$, $L_3$ and $L_4$ respectively. Then by Algorithm~\ref{A:IGC}, $(\mathcal{D},\mathcal{H})$  is smoothable if and only if the linear equations $\alpha_p+f_p(u_i)\beta_p=\alpha_q+f_q(v_i)\beta_q$   have a solution for the unknowns $\alpha_p,\beta_p,\alpha_q,\beta_q$ such that $\beta_p\neq 0$  and $\beta_q\neq 0$. 

To determine the smoothability of a solvable $(\mathcal{D},\mathcal{H})$ whose projection to its bifurcation tree is as in Figure~\ref{F:ExMetCom}(d) (respectively, Figure~\ref{F:ExMetCom}(e)), one only need to test the intrinsic global compatibility conditions at $p$ (respectively, at $q$) which reduces to saying that   $(\mathcal{D},\mathcal{H})$ is smoothable if and only if the three forward tangent directions at $p$ (respectively, at $q$)  are locally equivalent. 

Finally, let us sum up the smoothability test for  $(\mathcal{D}_A,\mathcal{H}_A)$ based on the above discussion on all cases of possible saturations of $(\mathcal{D}_A,\mathcal{H}_A)$ as follows: $(\mathcal{D}_A,\mathcal{H}_A)$ is smoothable if and only if either (1) at one of $p$ and $q$, the three forward tangent directions  are locally equivalent, or (2) the linear equations $\alpha_p+f_p(u_i)\beta_p=\alpha_q+f_q(v_i)\beta_q$   have a solution for the unknowns $\alpha_p,\beta_p,\alpha_q,\beta_q$ such that $\beta_p\neq 0$  and $\beta_q\neq 0$. \qed
\end{example}

\section*{Acknowledgments}
We are very thankful to Matthew Baker for his guidance and support during the course of
this work. We thank Janko Boehm, Christian Haase, Eric Katz,  Yoav Len,  Hannah Markwig, Dhruv Ranganathan and Frank-Olaf Schreyer for helpful discussions.  Thanks to Bernd Sturmfels for his encouragement and support.

\bibliographystyle{alpha}
\bibliography{citation}

\appendix
\section{Berkovich Skeleta and Saturated Metrized Complexes}\label{S:Berkovich}
\subsection{Saturated Metrized Complex associated to a Berkovich Skeleton}\label{subS:SkelRefMet}

 We begin by briefly recalling the concept of skeleton of the Berkovich analytic curve.  A semistable vertex set $V$ of $X^{\rm an}$ is a finite set of type-II points of $X^{\rm an}$ such that the complement of $V$ in $X^{\rm an}$ is a disjoint union of a finite number of open annuli and an infinite number of open balls.  Let $\Sigma(X^{\rm an},V)$ be the skeleton of $X^{\rm an}$ with respect to a semistable vertex set $V$.

In order to associate a saturated metrized complex $\mathfrak{C}(\Sigma)$ to $\Sigma(X^{\rm an},V)$, we must associate the following data to it: a metric graph $\Gamma$, a smooth algebraic curve $C_p$ for each point $p \in \Gamma$ and for each $C_p$, we must specify a set $A_p$ of marked points that are in bijection with the set of tangent directions at $p$. The metric graph $\Gamma$ underlying the saturated metrized complex is defined as being isometric to $\Sigma(X^{\rm an},V)$.  We associate the algebraic curve $C_p$ to each point $p \in \Gamma$ as follows: since the value group of $\mathbb{K}$ is $\mathbb{R}$, every point in $\Sigma(X^{\rm an},V)$ is a type-II point \cite{Berkovich12}. Hence, the double residue field has transcendence degree one over $\kappa$ and is isomorphic to the function field of a smooth curve over $\kappa$. This smooth curve is well defined up to isomorphism and we associate this curve $C_p$ to the point $p \in \Gamma$. We define marked points associated to the algebraic curve $C_p$ as follows: let $x$ be the type-II point corresponding to $p$,  the set of tangent directions at any type-II point in $X^{\rm an}$ has a canonical bijection with the set of discrete valuations of the double residue field at that point  \cite[Chapter 1]{Berkovich12}.  The set of discrete valuations of the double residue field is in turn in bijection with the set of closed points of $C_p$ \cite[Chapter 1]{Berkovich12}.  For each tangent direction $t \in \Tan_\Gamma(p)$, we define its marked point  as the point in $C_p$ associated to the corresponding tangent direction in the skeleton $\Sigma(X^{\rm an},V)$. Note that the marked point associated to each tangent direction is distinct.

\begin{lemma}
For any skeleton $\Sigma(X^{\rm an},V)$ of $X^{\rm an}$,  the data $\mathfrak{C}(\Sigma)$ defines a saturated metrized complex. In particular, for all but a finite number of points in $\Gamma$, the curve $C_p$ is a projective line over $\kappa$.
\end{lemma}

\begin{proof}
To show that $\mathfrak{C}(\Sigma)$ is a saturated metrized complex, we must verify that the curve $C_p$ has genus zero for all but finitely many points of $\Gamma$. Using Formula (5.45.1) of \cite{BPR11}, we have $g(X)=g(\Gamma)+\sum_{p \in \Gamma} g(C_p)$. Hence, $g(C_p)=0$ for all but finitely many $p$.
\end{proof}

\begin{remark}
The semistable vertex sets of $X$ are in one to one correspondence with the semistable models of $X$, we refer to \cite[Sections 5.14 and 5.29]{BPR11} for a detailed treatment of the topic.  Via this correspondence, we can associate  a saturated metrized complex to a semistable model of $X$. This saturated metrized complex is the ``limit" of the metrized complexes associated to semistable models obtained by successively blowing up the special fiber at its nodes.
\end{remark}

We define a morphism from $\tau_{*}: {\rm Div}(X) \rightarrow {\rm Div}(\mathfrak{C}(\Sigma))$ called the \emph{specialization map} and a map that takes a rational function on $X$ to a rational function on $\mathfrak{C}(\Sigma)$ called the \emph{reduction map}. We follow the analogous construction for metrized complexes by Amini and Baker \cite[Section 4]{AB12}.

\subsection{Specialization Map}
Suppose that $r_{V}: X^{\rm an} \rightarrow \Sigma(X^{\rm an},V)$ is the retraction map and let $\{r_{V,s}\}_{s \in [0,1]}$ be the family of retraction maps associated with the deformation retraction from $X^{\rm an}$ to $\Sigma(X^{\rm an},V)$. In particular, $r_{V,1}=r_{V}$.  For a closed point $z \in X$, the point $r_{V}(z)$ has a unique  tangent direction $t^{\rm an}_{V}(z)$ in $X^{\rm an}$ that lies in the image of the retraction map $r_{V,s}$  where $s$ is in an open neighborhood of $1$.  The map $\tau_*$ takes $z$ to the point $(r_{V}(z), \red_p(t^{\rm an}_{V}(z))$ on $\mathfrak{C}(\Sigma)$ where $\red_p(t^{\rm an}_{V}(z))$ is the marked point in $C_p$ corresponding to the tangent direction $t^{\rm an}_{V}(z)$. We extend this map linearly to define a \emph{specialization} map from ${\rm Div}(X)$ to ${\rm Div}(\mathfrak{C}(\Sigma))$.

\begin{lemma}\label{L:specialization}
The specialization map $\tau_{*}$ is a homomorphism from $\Div(X)$ to $\Div(\mathfrak{C}(\Sigma))$ that takes effective divisors on ${\rm Div}(X)$ to effective divisors on ${\rm Div}(\mathfrak{C}(\Sigma))$. The image of $\tau_{*}$ is the set of all divisors $(D_{\Gamma}, \{D_p\}_{p \in \Gamma}) \in {\rm Div}(\mathfrak{C}(\Sigma))$ such that the support of $D_p$ is contained in the set $C_p \setminus A_p$ for all $p \in \Gamma$.
\end{lemma}

\subsection{Reduction of Rational Functions}

Consider a point $p \in \Gamma$ and let $x$ be the corresponding type (2) point in $\Sigma(X^{\rm an},V)$. By $f(x)$, we denote the multiplicative semi-norm defined by $x$ evaluated at $f$ and let $c \in \mathbb{K}^{*}$ such that $|c|=|f(x)|$.  Let ${\tilde{H}}(x)$ be the double residue field of $x$ and note that the field ${\tilde{H}}(x)$ is isomorphic to the function field of $C_p$.  Suppose that $f$ maps to $f_x$ in ${\tilde{H}}(x)$. The reduction map takes $f$ to $(c^{-1}f)_x$, we denote $(c^{-1}f)_x$ by $\tilde{f}_x$ and the corresponding rational function in $C_p$ by $\tilde{f}_p$. Note that $f_x$ is only defined up to multiplication by $\kappa^{*}$ and hence, its divisor is well-defined.

\begin{lemma}\label{dimpre_lem}\cite[Lemma 4.3]{AB12}  The dimension of any finite dimensional subspace of $\kappa(X)$ is preserved by reduction.\end{lemma}

Given a rational function $f$ on $X$, we let $f_{\Gamma}$ be a rational function on $\Gamma$ given by the restriction to the skeleton $\Gamma=\Sigma(X^{\rm an},V)$ of the function ${\rm log}|f|:X^{\rm an}\rightarrow \mathbb{R}\bigcup\{\pm\infty\}$. Hence, given a rational function $f$ on $X$ we associate a rational function $\mathfrak{f}=(f_{\Gamma}, \{\tilde{f}_p\}_{p \in \Gamma})$ on $\mathfrak{C}(\Sigma)$.
The following  version of the Poincar\'e-Lelong Formula for saturated metrized complexes establishes a compatibility between the specialization and the reduction maps.

\begin{theorem}{\rm ({\bf Poincar\'e-Lelong Formula})} \label{T:PoicareLelong}
For any non-zero rational function $f$ on $X$, suppose that $\mathfrak{f}$ is the reduction of $f$ on $\mathfrak{C}(\Sigma)$, we have $\tau_{*}({\rm div}(f))={\rm div}(\mathfrak{f})$.  Hence, the map $\tau_*$ takes principal divisors in $X$ to principal divisors in $\mathfrak{C}(\Sigma)$.
 \end{theorem}
\begin{proof} For a point $x$ in the skeleton $\Sigma(X^{\rm an},V)$, we  partition  the set of $\Tan_x$ of tangent directions at $x$ into the tangent directions in $\Sigma(X^{\rm an},V)$  and its complement and denote them by $\Tan_{i,x}$ and $\Tan_{r,x}$ respectively. By parts (2) and (5) of the slope formula \cite[Theorem 5.69]{BPR11}, we note that ${\rm ord}_t(\tilde{f}_x)=0$ for all but points $x \in \Sigma(X^{\rm an},V)$ and $t \in \Tan_{r,x}$ except those that lie in the image (under the retraction map) of the support of ${\rm div}(f)$. By part (2) of the slope formula, ${\rm sl}_{t}(f_{\Gamma})={\rm ord}_t(f_x)$.  Hence,  ${\rm div}(\mathfrak{f})$ has support at a finite number of points and its support coincides with the support of $\tau_*({\rm div}(f))$.  Hence, ${\rm div}(\mathfrak{f})$ is a divisor (not just a pseudo-divisor).  Let $S$ be the union of the support of ${\rm div}(f_{\Gamma})$ and the points of $\Gamma$ with valence at least three. Thus, $\tau_{*}({\rm div}(f))$ and ${\rm div}(\mathfrak{f})$ coincide on points in $\Gamma \setminus S$.  Consider the metrized complex $\mathfrak{C}(\Sigma)|S$ obtained by restricting $\mathfrak{C}$ to $S$. More precisely, $\mathfrak{C}(\Sigma)|S$ is a metrized complex whose metric graph is $\Gamma$ with the model given by the set $S$ and the algebraic curves $C_v$ for every point in $v \in S$ and the marked points exactly as in $\mathfrak{C}(\Sigma)$. By the Poincare-Lelong formula shown in Amini and Baker \cite{AB12}, we have $\tau_{*}({\rm div}(f))$ and ${\rm div}(\mathfrak{f})$ coincide on $\mathfrak{C}(\Sigma)|S$. \end{proof}

\section{$\LDHone$, $\LDHtwo$, $\LDHthree$, and $\LDHfour$ as Partially Ordered Sets}\label{S:TreeSpace}
In this section, we show that a partial order can be naturally imposed to the spaces $\LDHone$, $\LDHtwo$, $\LDHthree$, and $\LDHfour$ of partition trees.

Let $\rho$ be a solution to $(\mathcal{D},\mathcal{H})$ and  $\mathcal{B}$ be the corresponding partition tree. Recall that for a partition tree $(\mathcal{T},\pi_\mathcal{T})$  in $\LDHone=\Lambda_\rho$, we have a  canonical projection $\Theta^{\mathcal{B}}_{\mathcal{T}}$ from $\mathcal{B}$ to $\mathcal{T}$ (Proposition~\ref{P:CanoSurj}) which induces a partition $P_c$ of $(d^\rho_\mathcal{B})^{-1}(c)$ for any $c\in\imag\rho$ (Remark~\ref{R:partition}).

For two partition trees $(\mathcal{T}_1,\pi_{\mathcal{T}_1})$ and $(\mathcal{T}_2,\pi_{\mathcal{T}_2})$, we say $(\mathcal{T}_1,\pi_{\mathcal{T}_1})\leqslant (\mathcal{T}_2,\pi_{\mathcal{T}_2})$ or simply $\mathcal{T}_1\leqslant \mathcal{T}_2$ if the partition of $(d^\rho_\mathcal{B})^{-1}(c)$ induced by $\mathcal{T}_1$ is a refinement of the partition induced by $\mathcal{T}_1$ for each $c\in\imag\rho$.  Moreover, if $\mathcal{T}_1\leqslant \mathcal{T}_2$, there is a natural map $\Theta^{\mathcal{T}_1}_{\mathcal{T}_2}:\mathcal{T}_1\rightarrow\mathcal{T}_2$ with $x\mapsto y$ if  $(\Theta^{\mathcal{B}}_{\mathcal{T}})^{-1}(x)\subseteq(\Theta^{\mathcal{B}}_{\mathcal{T}})^{-1}(y)$. Clearly, in this sense, the coarsest partition tree $(\imag\rho,\rho)$ and the finest partition tree  $(\mathcal{B},\pi_\mathcal{B})$ are the maximum and minimum
 of $\Lambda_\rho$ respectively (recall that as a rooted metric tree, $\imag\rho=[\min_{p \in \Gamma}\rho(p),\max_{p \in \Gamma} \rho(p)]$ has its root at $\min_{p \in \Gamma}\rho(p)$).

The following lemma is a natural consequence of the definitions of partition trees and the maps between them.

\begin{lemma}
If $\mathcal{T}_1\leqslant\mathcal{T}_2\leqslant\mathcal{T}_3$ as partition trees, the following diagram commutes.
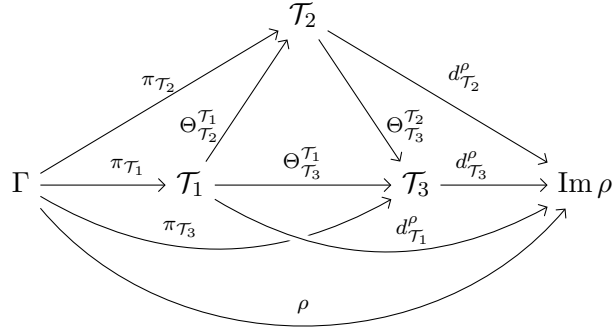
\[
\begin{tikzpicture}[scale=1.5,
back line/.style={solid},
cross line/.style={preaction={draw=white, -,line width=4pt}},
text height=1.6ex, text depth=0.5ex]
\node (Gamma) at (-2.5,0) {$\Gamma$};
\node (P) at (-1,0) {$\mathcal{T}_1$};
\node (Q) at (0,1.5) {$\mathcal{T}_2$};
\node (R) at (1,0) {$\mathcal{T}_3$};
\node (ImRho) at (2.5,0) {$\imag\rho$};
\path[->,font=\scriptsize,>=angle 90]
(Gamma) edge node[pos=0.7,above]{$\pi_{\mathcal{T}_1}$} (P)
(Gamma) edge node[above]{$\pi_{\mathcal{T}_2}$} (Q)
(Gamma) edge[bend right, back line] node[pos=0.4,above]{$\pi_{\mathcal{T}_3}$} (R)
(Gamma) edge[bend right=50] node[auto]{$\rho$} (ImRho)

(P) edge node[pos=0.3,left]{$\Theta^{\mathcal{T}_1}_{\mathcal{T}_2}$} (Q)
(P) edge node[auto]{$\Theta^{\mathcal{T}_1}_{\mathcal{T}_3}$} (R)
(P) edge[bend right, cross line] node[pos=0.6,above]{$d^\rho_{\mathcal{T}_1}$} (ImRho)
(Q) edge node[pos=0.7,right]{$\Theta^{\mathcal{T}_2}_{\mathcal{T}_3}$} (R)
(Q) edge node[auto]{$d^\rho_{\mathcal{T}_2}$} (ImRho)
(R) edge node[pos=0.3,auto]{$d^\rho_{\mathcal{T}_3}$} (ImRho);
\end{tikzpicture}
\]\qed
\end{lemma}

The following lemma says that $\LDHtwo$ is lower closed.
\begin{lemma} \label{L:LDH2_lowerclosed}
If $\mathcal{T}\in\LDHtwo$, then any element $\mathcal{T}'\in\LDHone$ with $\mathcal{T}'\leqslant\mathcal{T}$ is also in $\LDHtwo$.
\end{lemma}

\begin{proof}
Recall that by definition, to say $\mathcal{T}\in\LDHtwo$ is equivalent to saying that for every  point $p \in \Gamma$ and each pair of tangent directions $t_1,t_2\in\Tan^+_\Gamma(p)$,  $t_1$ is locally equivalent to $t_2$ if  $\pi_{\mathcal{T}*}(t_1)=\pi_{\mathcal{T}*}(t_2)$.

On the other hand, $\mathcal{T}'\leqslant\mathcal{T}$ means that $\mathcal{T}'$ induces finer partitions on forward tangent directions of $\mathcal{B}$ than $\mathcal{T}$. Thus if $\pi_{\mathcal{T'}*}(t_1)=\pi_{\mathcal{T'}_2*}(t_2)$ then $\pi_{\mathcal{T}*}(t_1)=\pi_{\mathcal{T}*}(t_2)$ and we conclude $\mathcal{T}'\in\LDHtwo$.
\end{proof}

Since $\mathcal{B}$ is the minimum element of $\LDHone$, we have the following corollary.
\begin{corollary}
If $\LDHtwo$ is nonempty, then $\mathcal{B}$ is an element of $\LDHtwo$.
\end{corollary}

\begin{example} \label{E:lattice}
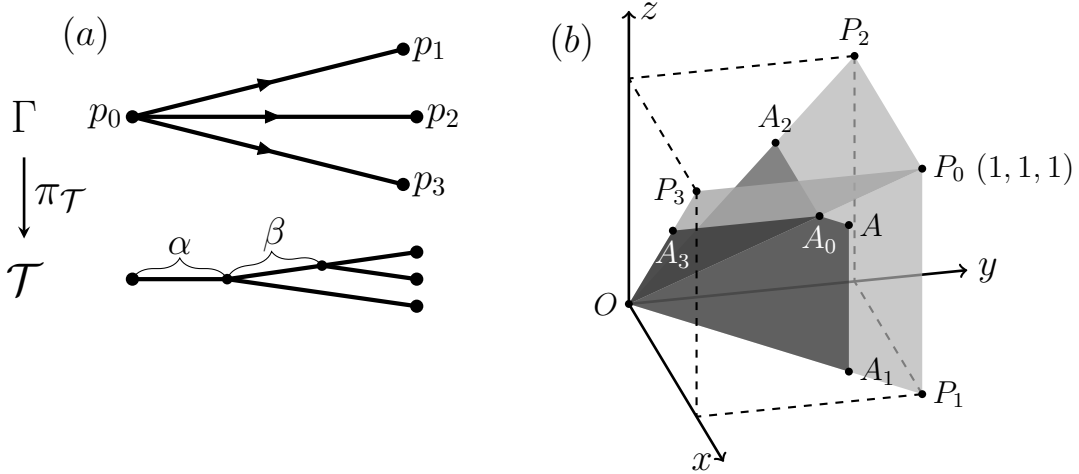
\begin{figure}[tbp]
 \centering
\begin{tikzpicture}[x=1.8cm,y=1.8cm]
\coordinate (o) at (0,4);
\draw ($(o)+(-0.6,-0.4)$) node[anchor=west] {\Large $(a)$};

\coordinate (p0) at ($(o)+(0,-1)$);
    \fill [black] (p0) circle (2.5pt);
    \draw (p0) node[anchor=east] {\Large $p_0$};
\coordinate (p1) at ($(p0)+(2,+0.5)$);
    \fill [black] (p1) circle (2.5pt);
    \draw (p1) node[anchor=west] {\Large $p_1$};
\coordinate (p2) at ($(p0)+(2.1,0)$);
    \fill [black] (p2) circle (2.5pt);
    \draw (p2) node[anchor=west] {\Large $p_2$};
\coordinate (p3) at ($(p0)+(2,-0.5)$);
    \fill [black] (p3) circle (2.5pt);
    \draw (p3) node[anchor=west] {\Large $p_3$};

\begin{scope}[line width=1.6pt, every node/.style={sloped,allow upside down}]
  \draw (p0) -- node {\midarrow} (p1);
  \draw (p0) -- node {\midarrow} (p2);
  \draw (p0) -- node {\midarrow} (p3);
\end{scope}

\coordinate (x0) at ($(o)+(0,-2.2)$);
    \fill [black] (x0) circle (2.5pt);
\coordinate (x123) at ($(x0)+(0.7,0)$);
    \fill [black] (x123) circle (2pt);
\coordinate (x12) at ($(x123)+(0.7,0.1)$);
    \fill [black] (x12) circle (2pt);
\coordinate (x1) at ($(x12)+(0.7,0.1)$);
    \fill [black] (x1) circle (2.5pt);
\coordinate (x2) at ($(x12)+(0.7,-0.1)$);
    \fill [black] (x2) circle (2.5pt);
\coordinate (x3) at ($(x123)+(1.4,-0.2)$);
    \fill [black] (x3) circle (2.5pt);
\begin{scope}[line width=1.6pt]
  \draw (x0) --  (x123);
  \draw (x123) --  (x12);
  \draw (x12) --   (x1);
  \draw (x12) --   (x2);
  \draw (x123) --  (x3);
\end{scope}
\draw [decorate,decoration={brace,amplitude=8pt},xshift=0pt,yshift=0pt]  (x0) -- (x123) node [black,midway,yshift=0.45cm] {\Large $\alpha$};
\draw [decorate,decoration={brace,amplitude=8pt},xshift=0pt,yshift=0pt]  (x123) -- (x12) node [black,midway,yshift=0.5cm] {\Large $\beta$};

\draw ($(p0)+(-0.8,0)$) node {\LARGE $\Gamma$};
\draw ($(x0)+(-0.8,0)$) node {\LARGE $\mathcal{T}$};
\draw [-stealth, line width=1pt] ($(p0)+(-0.8,-0.3)$) -- ($(x0)+(-0.8,0.3)$);
\draw ($1/2*(p0)+1/2*(x0)+(-0.8,0)$) node[anchor=west] {\LARGE $\pi_\mathcal{T}$};

\begin{scope}[x  = {(0.9cm,-1.5cm)},
                    y  = {(3cm,0.3cm)},
                    z  = {(0cm,3cm)},
                    axis/.style={->,black,line width=1pt}, line width=0.8pt]

\coordinate (o) at (0,2.2,0.75);
\draw ($(o)+(0,-0.4,1.2)$) node[anchor=west] {\Large $(b)$};

\coordinate (xx) at ($(o)+(1,0,0)$);
\coordinate (xy) at ($(o)+(1,1,0)$);
\coordinate (xz) at ($(o)+(1,0,1)$);
\coordinate (yy) at ($(o)+(0,1,0)$);
\coordinate (yz) at ($(o)+(0,1,1)$);
\coordinate (zz) at ($(o)+(0,0,1)$);
\coordinate (xyz) at ($(o)+(1,1,1)$);

\def\ratiobeta{0.75}
\def\ratioalpha{0.65}

\coordinate (A) at ($(o)+(\ratiobeta,\ratiobeta,\ratioalpha)$);
\coordinate (A0) at ($(o)+(\ratioalpha,\ratioalpha,\ratioalpha)$);
\coordinate (A1) at ($(o)+(\ratiobeta,\ratiobeta,0)$);
\coordinate (A2) at ($(o)+(0,\ratioalpha,\ratioalpha)$);
\coordinate (A3) at ($(o)+(\ratioalpha,0,\ratioalpha)$);

\draw[axis] (o) -- ($-0.4*(o)+1.4*(xx)$) node[anchor=east]{\Large $x$};
\draw[axis] (o) -- ($-0.5*(o)+1.5*(yy)$) node[anchor=west]{\Large $y$};
\draw[axis] (o) -- ($-0.3*(o)+1.3*(zz)$) node[anchor=west]{\Large $z$};

\draw[dashed] (yz) --  (yy);
\fill [black!30,opacity=0.7] (o) -- (yz) -- (xyz);
\fill [black!60,opacity=0.7] (o) -- (A2) -- (A0);

\draw[dashed] (xy) --  (yy);
\fill [black!30,opacity=0.7] (o) -- (xy) -- (xyz);
\fill [black!80,opacity=0.7] (o) -- (A1) -- (A) -- (A0);
\fill [black!35,opacity=0.8] (o) -- (xz) -- (xyz);
\fill [black!80,opacity=0.8] (o) -- (A3) -- (A0);
\draw[dashed] (yz) --  (zz);

\draw[dashed] (xy) --  (xx);

\draw[dashed] (xz) --  (xx);
\draw[dashed] (xz) --  (zz);

\draw (o) node[anchor=east] {\large $O$};
\fill [black] (o) circle (1.5pt);

\draw (A) node[anchor=west] {\large $A$};
\fill [black] (A) circle (1.5pt);
\draw (A0) node[anchor=north,white] {\large $A_0$};
\fill [black] (A0) circle (1.5pt);
\draw (A1) node[anchor=west] {\large $A_1$};
\fill [black] (A1) circle (1.5pt);
\draw (A2) node[anchor=south] {\large $A_2$};
\fill [black] (A2) circle (1.5pt);
\draw (A3) node[anchor=north,white] {\large $A_3$};
\fill [black] (A3) circle (1.5pt);

\draw (xyz) node[anchor=west] {\large $P_0~(1,1,1)$};
\fill [black] (xyz) circle (1.5pt);
\draw (xy) node[anchor=west] {\large $P_1$};
\fill [black] (xy) circle (1.5pt);
\draw (yz) node[anchor=south] {\large $P_2$};
\fill [black] (yz) circle (1.5pt);
\draw (xz) node[anchor=east] {\large $P_3$};
\fill [black] (xz) circle (1.5pt);
\end{scope}

\end{tikzpicture}
 \caption{(a) A metric tree $\Gamma$ with the projection onto a partition tree. (b) $\LDHone$ is a union of three triangles $\triangle OP_0P_1$, $\triangle OP_0P_2$ and $\triangle OP_0P_3$, and the space $\{T\in\LDHone\mid T\leqslant A\}$ is the union of triangles $\triangle OA_0A_2$, $\triangle OA_0A_3$ and trapezoid $OA_0AA_1$.}\label{F:Lattice}

\end{figure}

We consider a simple metric graph $\Gamma$ which is a metric tree with root $p_0$ and leaves $p_1$, $p_2$ and $p_3$ and all edge lengths being $1$ as shown in Figure~\ref{F:Lattice}(a). Let $\mathfrak{C}$ be a saturated metrized complex of genus $0$ with underlying metric graph $\Gamma$. Suppose the global diagram of a base-point-free diagrammatic pre-limit $g^1_d$ $(\mathcal{D},\mathcal{H})$ on $\mathfrak{C}$ has its differential form part with multiplicity $1$ on each direct edge from $p_0$ to $p_i$ ($i=1,2,3$). Then each partition tree $\mathcal{T}$ in $\LDHone$ can be derived by the following procedure: (1) glue the three edges $p_0p_1$, $p_0p_2$ and $p_0p_3$ continuously from $p_0$ of length $\alpha$, and then (2) choose two edges and continue the gluing on the selected edges for length $\beta$ (as shown in Figure~\ref{F:Lattice}(a)). Note that we have three cases for step (2) based on which two edges are selected: Case-1 for $p_0p_1$ and $p_0p_2$ being selected, Case-2 for $p_0p_2$ and $p_0p_3$ being selected, and Case-3 for $p_0p_1$ and $p_0p_3$ being selected. Let $x$ be the total length glued for $p_0p_1$, $y$ be the total length glued for $p_0p_2$, and $z$ be the total length glued for $p_0p_3$. Then we can represent $\mathcal{T}$ uniquely by a point with coordinates $(x,y,z)$. In particular, the coordinates are $(\alpha+\beta,\alpha+\beta,\alpha)$ for Case-1, $(\alpha,\alpha+\beta,\alpha+\beta)$ for Case-2, and $(\alpha+\beta,\alpha,\alpha+\beta)$ for Case-3. Therefore $\LDHone$ is a union of three triangles inside a unit cube. As shown in Figure~\ref{F:Lattice}(b), $\LDHone=\triangle OP_0P_1\bigcup \triangle OP_0P_2 \bigcup\triangle OP_0P_3$ where $\triangle OP_0P_1$ corresponds to Case-1, $\triangle OP_0P_2$ corresponds to Case-2 and $\triangle OP_0P_3$ corresponds to Case-3. Moreover, let point $A$ with coordinates $(\alpha_A+\beta_A,\alpha_A+\beta_A,\alpha_A)$ be a point in $\triangle OP_0P_1$. In Figure~\ref{F:Lattice}(b), the space $\{T\in\LDHone\mid T\leqslant A\}$ is shown as the darker region (polyhedral complex with vertices $O$, $A$, $A_0$, $A_1$, $A_2$ and $A_3$). 

For this simple metric graph, to derive $\LDHtwo$, $\LDHthree$ and $\LDHfour$, the only local data need to be examined is for the curve at $p_0$. Note that this also means $\LDHthree=\LDHfour$. Denote the outgoing tangent directions at $p_0$ by $t_1$, $t_2$ and $t_3$ where $t_i$ is the tangent direction from $p_0$ to $p_i$. We have the following cases under the assumption that $(\mathcal{D},\mathcal{H})$ is base-point free (the conventional notations of open, closed and half-open-half-closed intervals are used for those of segments):
\begin{enumerate}
\item $(\mathcal{D},\mathcal{H})$ has degree $1$. Then the local partition at $p_0$ is the finest partition $\{\{t_1\},\{t_2\},\{t_3\}\}$. We have $\LDHtwo=\LDHthree=\LDHfour=\{O\}$.
\item $(\mathcal{D},\mathcal{H})$ has degree $2$. Then the local partition at $p_0$ is made of a singleton and a set of two elements.
    \begin{enumerate}
    \item The local partition at $p_0$ is $\{\{t_1,t_2\},\{t_3\}\}$. Then $\LDHtwo=[O,P_1]$ and $\LDHthree=\LDHfour=(O,P_1]$.
    \item The local partition at $p_0$ is $\{\{t_2,t_3\},\{t_1\}\}$. Then $\LDHtwo=[O,P_2]$ and $\LDHthree=\LDHfour=(O,P_2]$.
    \item The local partition at $p_0$ is $\{\{t_1,t_3\},\{t_2\}\}$. Then $\LDHtwo=[O,P_3]$ and $\LDHthree=\LDHfour=(O,P_3]$.
    \end{enumerate}
\item $(\mathcal{D},\mathcal{H})$ has degree $3$. Then all tangent directions at $p_0$ are locally equivalent. We have $\LDHtwo=\LDHone$ and $\LDHthree=\LDHfour=\LDHone\setminus([O,P_1]\bigcup[O,P_2]\bigcup[O,P_3])$. One can consider $\LDHone$ as the compactification of $\LDHfour$'s in case (3) by $\LDHfour$ in case (1) and (2).
\end{enumerate}

\end{example}

\begin{remark}
The spaces $\LDHone$, $\LDHtwo$, $\LDHthree$, and $\LDHfour$ have much richer structures (e.g., lattice structure, metric and convexity) which will be presented in our follow-up work.
Moreover, we expect that the space $\LDHfour$ has an interpretation in terms of a skeleton of the analytification of the following moduli space of maps: the moduli space of all maps of the form $X \rightarrow \mathbb{P}^1_{\mathbb{K}}$ where $X$ is the smooth curve in the commutative diagram \ref{commutdia}.  Results of this flavor have been obtained in Cavalieri et al. \cite{CMR14} for spaces of admissible covers and for moduli spaces of curves by Abramovich et al. \cite{ACP12}. As we vary over all limit $g^1_d$s on the saturated metrized complex $\mathfrak{C}$. The space $\cup_{(\mathcal{D},\mathcal{H})}\LDHfour$  parameterizes the space of metric trees underlying all smoothings of limit $g^1_d$s in $\mathfrak{C}$.
 \end{remark}

\end{document}